\documentclass[onefignum,onetabnum]{siamart190516}

\usepackage{lipsum}
\usepackage{amsfonts}
\usepackage{amssymb}
\usepackage{graphicx}
\usepackage[export]{adjustbox}
\usepackage{epstopdf}
\usepackage{algorithmic}
\ifpdf
  \DeclareGraphicsExtensions{.eps,.pdf,.png,.jpg}
\else
  \DeclareGraphicsExtensions{.eps}
\fi

\newsiamremark{remark}{Remark}
\newsiamremark{hypothesis}{Hypothesis}
\crefname{hypothesis}{Hypothesis}{Hypotheses}
\newsiamthm{claim}{Claim}

\headers{A least-squares RBF-FD method}{I. Tominec, E. Larsson, and A. Heryudono}

\title{A least squares radial basis function finite difference method with improved stability properties \thanks{This is a preprint. The final print is published in SIAM journal on Scientific Computing.}}

\author{Igor Tominec\thanks{Uppsala  University, Department  of  Information  Technology, SE-751 05 Uppsala, Sweden
  (\email{igor.tominec@it.uu.se}, \email{elisabeth.larsson@it.uu.se}).}
\and Elisabeth Larsson\footnotemark[1]
\and Alfa Heryudono\thanks{University  of  Massachusetts  Dartmouth,  Department  of  Mathematics,  Dartmouth,  MA~02747, USA 
(\email{aheryudono@umassd.edu}).} \funding{The work of the first and the second author was supported by the Swedish Research Council, grant
no.~2016-04849. The work of the third author was supported by the National Science Foundation grant
NSF-DMS 2012011. }}

\usepackage{amsopn}

\newcommand{\LL}{\mathcal{L}}

\DeclareMathOperator*{\argmin}{arg\,min}
\newcommand{\new}[1]{\textcolor{black}{ #1}}
\newcommand{\nnew}[1]{\textcolor{black}{ #1}}

\ifpdf
\hypersetup{
  pdftitle={A least-squares RBF-FD method},
  pdfauthor={I. Tominec, E. Larsson, and A. Heryudono}
}
\fi

\begin{document}

\maketitle

\begin{abstract}
  Localized collocation methods based on radial basis functions (RBFs) 
  for elliptic problems appear to be non-robust in the presence of Neumann boundary conditions. 
  In this paper, we overcome this issue by formulating the RBF-generated finite difference method in a discrete least-squares 
  setting instead. 
  This allows us to prove high-order convergence under node refinement and to numerically verify that 
  the least-squares formulation is more accurate and robust than the collocation formulation. 
  The implementation effort for the modified algorithm is comparable to that for the collocation method.
\end{abstract}

\begin{keywords}
  radial basis function, least-squares, partial differential equation, elliptic problem, Neumann condition, RBF-FD
\end{keywords}

\begin{AMS}
  65N06, 65N12, 65N35
\end{AMS}

\section{Introduction}
Radial basis function-generated finite difference methods \newline(RBF-FD) generalize classical finite difference methods (FD) to scattered node settings. However, while FD uses tensor products of one-dimensional derivative approximations, RBF-FD directly computes multivariate approximations, which is advantageous when differentiation is not aligned with a coordinate direction~\cite{FoFlyRu10}. In this paper, we generalize RBF-FD to a 
least squares setting (RBF-FD-LS), which improves stability and accuracy.

RBF-FD is a meshfree method, which provides flexibility with respect to the geometry. In contrast to FD methods where an entire coordinate dimension is affected by adaptive refinement, RBF-FD allows for coordinate independent local adaptivity~\cite{MivSy18}.

The RBF-FD method was first introduced by Tolstykh in 2000~\cite{Tol00}, and other early papers include~\cite{ShuDiYe03,WriFo06}. The method is based on the idea that given scattered nodes $x_j\in\mathbb{R}^d$, $j=1,\ldots,n$, in the neighborhood of a point $x$, we can create a localized RBF approximation of the function $u(x)$ using these 'stencil points',
\begin{equation}
  u_h(x)=\sum_{j=1}^nc_j\phi(\|x-x_j\|)\equiv\sum_{j=1}^nc_j\phi_j(x),
\end{equation}
where $h$ is a measure of the inter-nodal distance, $\phi(r)$ is a radial basis function, and $c_j$ are unknown coefficients. The interpolation conditions $u_h(x_i)=u(x_i)$ lead to the linear system
\begin{equation}
  \underbrace{\begin{pmatrix}
    \phi_1(x_1) & \cdots &\phi_n(x_1)\\
    \vdots && \vdots\\
    \phi_1(x_n) & \cdots & \phi_n(x_n)
  \end{pmatrix}}_A
  \begin{pmatrix}
    c_1\\\vdots\\c_n
  \end{pmatrix}
  =
  \begin{pmatrix}
    u(x_1)\\\vdots\\ u(x_n)
  \end{pmatrix}.
  \label{eq:int1}
\end{equation}  
If we let $\underline{c}=(c_1,\ldots,c_n)^T$ and $\underline{u}=(u(x_1),\ldots,u(x_n))^T$, we have that $\underline{c}=A^{-1}\underline{u}$. A benefit of using RBFs is that for commonly used radial functions $\phi(r)$ the matrix $A$ is guaranteed to be non-singular for distinct node points~\cite{Scho38,Micchelli86}. We can then proceed to apply an operator to the approximation:
\begin{eqnarray}
  \mathcal{L}u_h(x) &=& \sum_{j=1}^nc_j\mathcal{L}\phi_j(x) \nonumber \\
                    &=&\underbrace{(\mathcal{L}\phi_1(x),\ldots, \mathcal{L}\phi_n(x))}_{a^\mathcal{L}}\underline{c}= a^\mathcal{L}A^{-1}\underline{u}\equiv(\mathcal{L}\psi_1(x),\ldots,\mathcal{L}\psi_n(x))\underline{u} \equiv w^\mathcal{L}\underline{u},
  \label{eq:w1}
\end{eqnarray}
where $\{\psi_j(x)\}_{j=1}^n$, forms a cardinal basis for the local interpolant, i.e., $\psi_j(x_i)=\delta_{ij}$, and $w^\mathcal{L}$ are the stencil weights used for approximating the operator at the point $x$.

  
In the early work on RBF-FD, infinitely smooth RBFs as the Gaussian RBF with $\phi(r)=\exp(-r^2)$ or the multiquadric RBF with $\phi(r)=\sqrt{1+r^2}$ were used. Lately, there has been an increasing interest in using piecewise smooth polyharmonic splines (PHS) with $\phi(r)=|r|^{2k-1}$, $k\geq1$. These are conditionally positive definite functions. It was shown in~\cite{Micchelli86} that by adding a polynomial basis of a degree corresponding to the order of conditional positive definiteness and constraining the RBF coefficients $\underline{c}$ to be orthogonal to this basis, we can guarantee strict positive definiteness of the quadratic form $\underline{c}^TA\underline{c}$, which is important when proving optimality results. The RBF approximation 
then takes the form
\begin{equation}
u_h(x)=\sum_{j=1}^nc_j\phi_j(x) + \sum_{j=1}^m\new{\mu}_jp_j(x),\quad \sum_{j=1}^nc_jp_k(x_j)=0.
\label{eq:PHSpoly}
\end{equation}
where the second equation is the constraint.
The dimension $m$ of the polynomial space is given by the degree $p$ of the polynomial as $m=\binom{p+d}{d}$, where $d$ is the number of spatial dimensions. In the Ph.D thesis~\cite{Barnett15}, and the subsequent papers~\cite{FBW16,FFBB16,BFFB17,BFF19}, it was shown that it is beneficial to append a polynomial of a higher degree $p$ than strictly required. First, the convergence order of the method depends on $p$~\cite{Bayona19}. Secondly, the behavior near boundaries is improved compared with classical polynomial-based FD~\cite{BFF19}. It was suggested in~\cite{FFBB16} that for a two-dimensional problem, using a stencil size $n=2m$, leads to a robust method. We use this strategy in this paper.


The interpolation relation corresponding to~\eqref{eq:int1} for the polynomially augmented case becomes
\begin{equation}
	\label{eq:M}
	\underbrace{\begin{pmatrix}
		A & P \\
		P^T & 0
	\end{pmatrix}}_{\tilde{A}}
	\begin{pmatrix}
		\underline{c}\\
		\underline{\new{\mu}}
        \end{pmatrix}
	=	
        \begin{pmatrix}
          \underline{u}\\
		0
	\end{pmatrix}
,
\end{equation}
where $P_{ij}=p_j(x_i)$, and $\underline{\new{\mu}}=(\new{\mu}_1,\ldots,\new{\mu}_m)^T$. Similarly to~\eqref{eq:w1}, using~\eqref{eq:M} for the coefficient vectors, we get the differentiation relation
\begin{eqnarray}
\mathcal{L}u_h(x) &=&
\underbrace{
  \begin{pmatrix}
a^\mathcal{L} & p^\mathcal{L} 
\end{pmatrix}}_{b^\mathcal{L}}
\begin{pmatrix}
  	\underline{c}\\
	\underline{\new{\mu}}
\end{pmatrix}
\nonumber \\
&=&b^\mathcal{L}\tilde{A}^{-1}
\begin{pmatrix}
  \underline{u}\\
  0
\end{pmatrix}=(b^\mathcal{L}\tilde{A}^{-1})_{1:n}\underline{u} \equiv(\mathcal{L}\psi_1(x)\cdots\mathcal{L}\psi_n(x))\underline{u}\equiv w^\mathcal{L}\underline{u},
\label{eq:w2}
\end{eqnarray}  
where $p^\mathcal{L}=(\mathcal{L}p_1(x),\ldots,\mathcal{L}p_m(x))$. 
The PHS + polynomial RBF-FD method works well, but there is some sensitivity to the node layout, e.g., $P$ can become rank deficient for Cartesian node layouts. 
Several authors have developed algorithms for high quality scattered node generation~\cite{FoFly15,ShaKiFo18,SlaKo18,vdSFo19}. Another issue that we have encountered, and that was also noted in~\cite{LiKeChu06} is that errors become large at boundaries with Neumann boundary conditions. 


%
%
In this paper, we propose to improve the performance of the PHS + polynomial RBF-FD method by introducing least squares approximation (oversampling) at the PDE level. The least squares approach is also applicable to RBF-FD with other types of basis functions. A related study is~\cite{LaShcher17}, where least squares approximation is introduced in an RBF partition of unity method (RBF-PUM). It was shown that least squares RBF-PUM is numerically stable under patch refinement, which is not the case for collocation RBF-PUM. In~\cite{Schaback16} it is shown under quite general conditions that given enough oversampling, a broad class of discretizations is uniformly stable. 

A recent paper~\cite{Davydov19} analyses a least squares RBF-FD method formulated over a closed manifold. The formulation of the method is different from ours in that node points and evaluation points are the same; the oversampling is determined by the stencil size, and the theoretical analysis is performed using other strategies. \new{Another recent paper is~\cite{Mirzaei20}, where RBF-FD and RBF-PUM is combined to construct a method that is related, but uses a different approximation strategy.}
Least squares approximation has been used together with RBF-FD by other authors to address some specific problems. In~\cite{LiKeChu06}, an over-determined linear system is formed by enforcing both the PDE and the Dirichlet boundary conditions on the boundary, to improve the stability of the method. In~\cite{PeLiPiRu19} the context is the closest point method applied to a problem with a moving boundary in combination with RBF-FD. Enforcement of both the PDE and the constant-along-a-normal property of the closest point solution leads to an over-determined system and a robust method. 

The main contributions of this paper are
\begin{itemize}
%
\item The RBF-FD-LS algorithm that performs better than collocation-based RBF-FD in terms of efficiency and stability for the tested PDE problems.
\item Error estimates that have been validated numerically for RBF-FD-LS approximations when using the PHS + polynomial basis.
\item A better understanding of the properties of RBF-FD approximations in terms of a piecewise continuous trial space. 
\end{itemize}
The outline of this paper is as follows: In \cref{sec:model}, we define a Poisson problem with Dirichlet and Neumann boundary conditions. Then in \cref{sec:method}, we derive the RBF-FD-LS method. \Cref{sec:trial} focuses on the properties of the RBF-FD trial space, and then convergence and error estimates are derived in \cref{sec:theor}. Numerical experiments that validate the theoretical results are shown in \cref{sec:numericalstudy}. The paper ends with final remarks on the method and results in \cref{sec:finalremarks}.

\section{The model problem}\label{sec:model}
We build our understanding on a model problem, the Poisson equation with Dirichlet and Neumann boundary data:
\begin{equation}
  \label{eq:methods:Poisson}
  \begin{array}{rcrclr@{\,}c@{\,}l}
    \mathcal{L}_2u(y)&\equiv&\Delta u(y) &=& f_2(y), &y &\in& \Omega,\\
    \mathcal{L}_0u(y)&\equiv& u(y) &=& f_0(y), & y &\in& \partial\Omega_0, \\
    \mathcal{L}_1u(y)&\equiv&\nabla u(y) \cdot n &=& f_1(y),& y &\in& \partial\Omega_1.
   \end{array}     
\end{equation}
\new{We also use the notation $\Omega_i$ for the domain associated with $\mathcal{L}_i$.}
When working with the PDE problem, it is practical to have a unified formulation. We reformulate the system above as 
\begin{equation}
  D(y) u(y) = F(y),
  \label{eq:pde}
\end{equation}
where the specific operator $D(y)=\LL_i$ and right-hand-side function $F(y)=f_i(y)$ depend on the location of $y$.


The regularity of the problem depends on the geometry of the domain $\Omega$ in combination with the given right-hand-side functions.
In the problems that we solve in this paper, the domain is either smooth or convex, and the data is chosen such that the solution has bounded and continuous second derivatives. This ensures that the PDE problem~\eqref{eq:methods:Poisson} is well-defined pointwise.
In order to achieve high-order convergence, we require the solution to have additional smoothness.  We define the $L_2$-norm over a domain $\Omega$ as $\|u\|^2_{L_2(\Omega)}=\int_\Omega u(y)^2\,dy$, and use the notation $\|u\|_{L_2(\Omega)}=\|u\|_\Omega$ for brevity. We require $u\in W_\infty^{p+1}(\Omega)\subset W_2^{p+1}(\Omega) =\{u\,|\,\|D^\alpha u\|_{L_2(\Omega)}<\infty,\ |\alpha|\leq p+1\}$, where $p\geq 2$ is the degree of the polynomial basis added to the PHS approximation~\eqref{eq:PHSpoly} that we use in the numerical method.

Since we solve the discretized problem in the least squares sense, it is convenient for the theoretical results derived in \cref{sec:theor} to state also the continuous problem in least squares form. We require $\tilde{u}\in \new{V\subset W_2^2(\Omega)}$ for the least squares solution, \new{where the subspace $V$ is determined by the selected representation of the solution. Since we only use the continuous problem at a conceptual level, we are not specifying the subspace further.} 
%
%
%
%
%
%
The squared $L_2$-norm of the residual of the PDE problem for a function $v\in V$ is given by
\begin{eqnarray}
  \|r(v)\|_{L_2(\Omega)}^2&=&\int_{\partial\Omega_0}(\LL_0v(y)-f_0(y))^2+\int_{\partial\Omega_1}(L_1v(y)-f_1(y))^2+\int_\Omega(\LL_2v(y)-f_2(y))^2\nonumber\\
 &=& \int_{\partial\Omega_0}\left(\LL_0(v-u)\right)^2\,dy+\int_{\partial\Omega_1}\left(\LL_1(v-u)\right)^2\,dy+\int_\Omega\left(\LL_2(v-u)\right)^2\,dy,
\label{eq:contres}
\end{eqnarray}
%
where $f_i=L_iu$ was used in the second equality. If we introduce the bilinear form
\begin{equation}
  a(u,v) = \int_{\partial\Omega_0}uv\,dy+\int_{\partial\Omega_1}\frac{\partial u}{\partial n}\frac{\partial v}{\partial n}\,dy+\int_\Omega\Delta u\Delta v\,dy,
  \label{eq:bilinearcont}
\end{equation}
and note that $\|r(v)\|_{L_2(\Omega)}^2=a(v-u,v-u)$, the least squares solution of~\eqref{eq:methods:Poisson} is:
\begin{equation}
  \tilde{u}=
  \argmin_{v\in V}a(v-u,v-u).
  \label{eq:lscont}
\end{equation}
Alternatively, using that the residual is $a$-orthogonal to $V$, we can write
\begin{equation}
a(\tilde{u}-u,v)=0,\quad \forall v\in V.
\end{equation}
%
%
When $u\in V$, the least squares problem solves the PDE problem exactly, but in general for a numerical approximation, $u$ and $\tilde{u}$ reside in different subspaces, leading to a non-zero residual.

\section{Formulation of RBF-FD-LS in practice}\label{sec:method}

We start with generating a node set $X = \{x_k\}_{k=1}^N$ that covers the domain $\Omega$, on which we solve the PDE problem~\eqref{eq:pde}. It is beneficial for the approximation quality if the node distance is nearly uniform or varies smoothly over the domain. We associate each $x_k$ with a stencil and denote the $n$ points (including $x_k$) in the local neighborhood of $x_k$ that contribute to the stencil by $X_k$.
%
An example of a global node set and a stencil is given in the left part of Figure~\ref{fig:methods:rbf-fd:domainStencil}.
\begin{figure}
	\centering
        \includegraphics[height=0.27\linewidth]{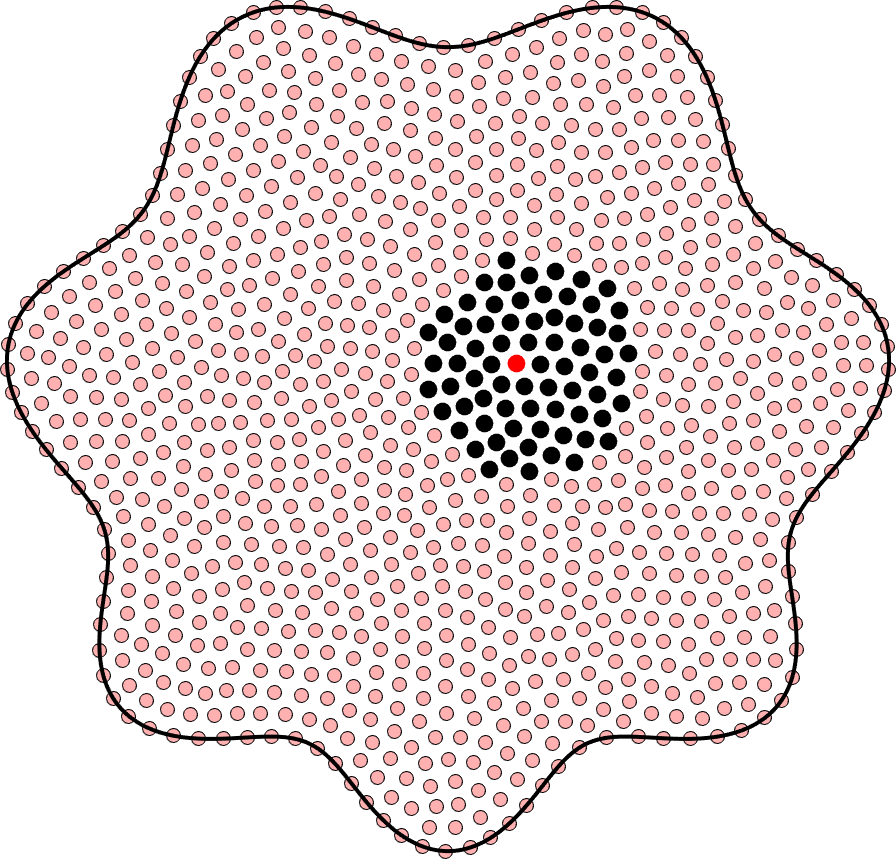}
        \hspace*{5mm}
        \includegraphics[height=0.27\linewidth]{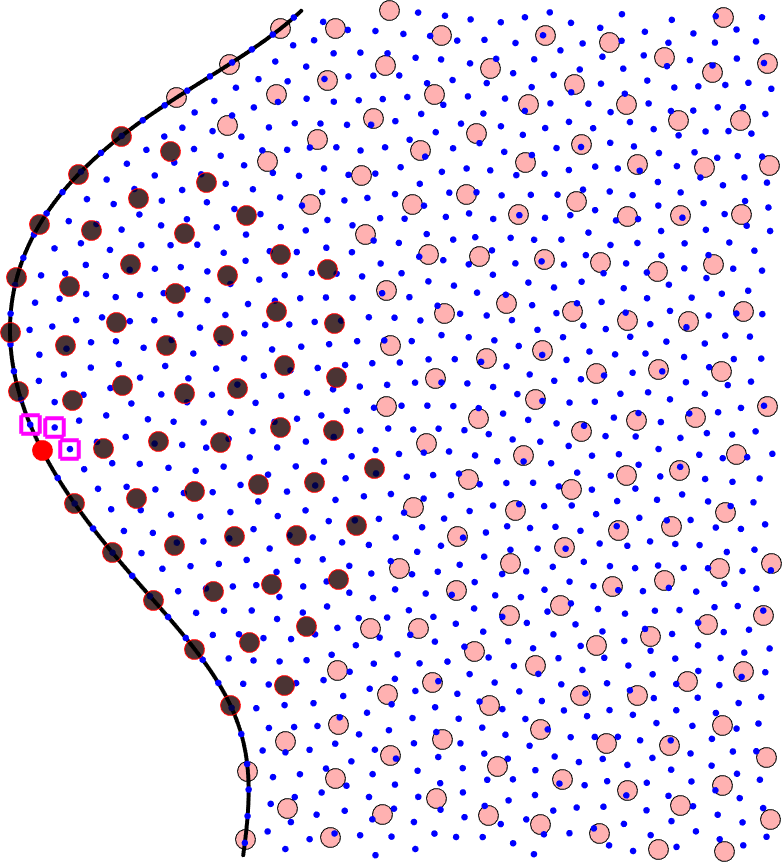}
	\caption{The black curve represents the domain boundary $\partial \Omega$. The pale-red dots distributed over $\Omega$ are the points in the node set $X$. The black dots are in each case the node points belonging to the stencil centered at the red point. The evaluation points in the node set $Y$ are shown in the right subfigure as blue dots. It is also indicated with square markers which evaluation points select this particular stencil for evaluating $u_h$.}	\label{fig:methods:rbf-fd:domainStencil}
\end{figure}

To evaluate the RBF-FD approximation at a point $y\in\bar{\Omega}$, we need a stencil selection method. In our algorithm we choose the stencil associated with the point $x_k$ that is closest to $y$. That is,
\begin{equation}
  k(y) = \argmin_{i}\|y-x_i\|. 
\end{equation}
A practical issue is that there are always points that are equally close to two or more stencils. Therefore we also need to break the tie, such that each $y$ is uniquely associated with one stencil. We then use~\eqref{eq:w2} to write down the global RBF-FD approximation to the solution of the PDE problem evaluated at the point $y$ as
\begin{equation}
u_h(y) =(b_{k}^{\mathcal{L}_0}\tilde{A}_{k}^{-1})_{1:n}u_h(X_k)=\left(\psi_1^k(y),\ldots,\psi_n^k(y)\right)u_h(X_k)=
w_{k}^{\mathcal{L}_0}u_h(X_k)
,
\label{eq:loceval}
\end{equation}
where the subscript or superscript $k=k(y)$ indicates quantities computed in the stencil associated with $x_k$, and where $u_h(X_k)$ is a column vector with $u_h$ evaluated at the local node set. The expression for the action of a differential operator $D(y)$ on the global RBF approximation $u_h$, evaluated at the point $y$, follows from~\eqref{eq:loceval}:
\begin{equation}
  D(y)u_h(y) =
  \left(D(y)\psi_1^k(y),\ldots, D(y)\psi_n^k(y)\right)u_h(X_k)=
w_{k}^{D(y)}u_h(X_k).
\label{eq:locdiff}
\end{equation}
We note that the local matrices $\tilde{A}_k$ can be reused for all points $y$ that select the same stencil, and for all operators.

To solve the PDE problem~\eqref{eq:pde}, we sample the approximation~\eqref{eq:locdiff} and the PDE data $F(y)$ at a global node set $Y=\{y_i\}_{i=1}^M$ that has to contain nodes in $\Omega$, at $\partial\Omega_0$, and at $\partial\Omega_1$. An example of an evaluation node set $Y$ is shown in the right part of Figure~\ref{fig:methods:rbf-fd:domainStencil}. We construct a sparse global linear system
\begin{equation}
  D_h(Y,X)u_h(X) = F(Y),
  \label{eq:discretePDE}
\end{equation}  
where row $i$ contains the equation for $D(y_i)u_h(y_i)=F(y_i)$ and the corresponding weights from~\eqref{eq:locdiff} are entered into the columns corresponding to the global indices of the nodes in $X_k$. In the same way, we form the relation
\begin{equation}
  u_h(Y)=E_h(Y,X)u_h(X),
  \label{eq:discreteEval}
\end{equation} 
using weights from~\eqref{eq:loceval}. If the number of evaluation points $M>N$, both $D_h(Y,X)$ and $E_h(Y,X)$ are rectangular $M\times N$ matrices. 
In \cite{Tominec_rbffdcode2021} we provide MATLAB code to generate rectangular RBF-FD matrices such as $E_h(Y,X)$ or $D_h(Y,X)$.

In the discretized PDE problem, $u_h(X)$ is the vector of unknowns, and we formally write the least squares solution of the linear system as
\begin{equation}
  u_h(X) = D_h^+(Y,X)\,F(Y),
  \label{eq:ls}
\end{equation}
where the $N\times M$ matrix $D_h^+ (Y,X)=(D_h^TD_h)^{-1}D_h^T$ is the pseudo inverse of $D_h(Y,X)$. To evaluate the solution at $Y$ we add the step~\eqref{eq:discreteEval} to get
\begin{equation}
  u_h(Y) = E_h(Y,X)\,D_h^+(Y,X)\,F(Y).
  \label{eq:stabnorm}
\end{equation}
The over-determined linear system~\eqref{eq:discretePDE} can also be formulated as a discrete residual minimization problem. We define the residual
\begin{equation}
  r(Y) =  D_h(Y,X)u_h(X) - F(Y).
  \label{eq:residual}
\end{equation}  
Then the solution~\eqref{eq:ls} minimizes $\|r(Y)\|_2^2$, and it also holds that
\begin{equation}
  D_h^T\new{(Y,X)}r(Y) = 0,
\end{equation}
due to the orthogonality property of the least squares residual.

The collocation RBF-FD method, where $Y=X$, is a special case of the derivation above, where the stencil selected for $y_k=x_k$ is always $k$, $D_h^+(X,X)=D_h^{-1}(X,X)$, and $E_h(X,X)=I_h$. This leads to
\begin{equation}
  r(X) = D_h(X,X) u_h(X) - F(X) = D_h(X,X) (D_h^{-1}(X,X) F(X)) - F(X) = \new{0}.
\end{equation}

In the discrete minimization of the residual~\eqref{eq:residual}, each equation has the same weight. This may cause problems with convergence to the PDE solution under node refinement.
We start by introducing a weighted discrete $\ell_2$-norm that corresponds to the continuous $L_2$-norm with the integral replaced by a discrete quadrature formula. The error in this approximation is further discussed in \cref{sec:discretenorm}.
We leave place holders $\beta_i$ for additional balancing of the different parts of the residual, and discuss these further in \cref{sec:errors:scaling}. Let a domain $\Omega$ be discretized by $M$ points $y_i$, $i=1,\ldots,M$. Then
\begin{equation}
  (u,v)_{\ell_2(\Omega)}=\frac{|\Omega|}{M}\sum_{i=1}^{M}u(y_i)v(y_i), \quad \|u\|^2_{\ell_2(\Omega)}=(u,u)_{\ell_2(\Omega)},
\end{equation}
where $|\Omega|=\int_\Omega 1\,dy$.
We denote the number of evaluation points that discretize the operator $\mathcal{L}_i$ in~\eqref{eq:methods:Poisson} by $M_i$ and note that if the evaluation points are quasi uniform with node distance $h_y$, then \new{for $d>1$}
\new{
\begin{equation}
  h_y=c_0\left(\frac{|\partial\Omega_0|}{M_0}\right)^\frac{1}{d-1} =c_1\left(\frac{|\partial\Omega_1|}{M_1}\right)^\frac{1}{d-1}=c_2
  \left(\frac{|\Omega|}{M_2}\right)^\frac{1}{d},
\end{equation}
}
where $c_0\approx c_1\approx c_2 \approx 1$. 
Scaling the evaluation matrix as $\bar{E}_h=\new{\left(\frac{|\Omega|}{M}\right)^\frac{1}{2}}E_h$ leads to 
\begin{eqnarray}
  \|u_h\|_{\ell_2(\Omega)}^2 &=&\frac{|\Omega|}{M}\|u_h(Y)\|_2^2= \|\bar{E}_hu_h(X)\|_2^2 = u_h(X)^T\bar{E}_h^T\bar{E}_hu_h(X).
\end{eqnarray}
For $D_h$, we scale according to the location of $y$, such that $\bar{D}_h=\mathrm{diag}(\beta(Y))D_h$, where
\new{\begin{equation}
  \label{eq:method:scaling}
  \beta(y)=\left\{
    \begin{array}{rcrl}
      \left(\frac{|\partial\Omega_0|}{M_0}\right)^\frac{1}{2}\beta_0&\approx&h_y^\frac{d-1}{2}\beta_0, &y\in\partial\Omega_0,\\
      \left(\frac{|\partial\Omega_1|}{M_1}\right)^\frac{1}{2}\beta_1&\approx&h_y^\frac{d-1}{2}\beta_1,&y\in\partial\Omega_1,\\
      \left(\frac{|\Omega|}{M_2}\right)^\frac{1}{2}\beta_2&\approx& h_y^\frac{d}{2}\beta_2, &y\in\Omega,
    \end{array}
    \right.
\end{equation}}%
and, similarly, we let $\bar{F}(y)=\beta(y)F(y)$. For the scaled residual $\bar{r}(Y)$, noting that $D_h(Y,X)u_h(X)=D(Y)u_h(Y)$ and $F(Y)=D(Y)u(Y)$, we get
\begin{eqnarray}
  \|\bar{r}(Y)\|_2^2&=&\|\bar{D}_h(Y,X)u_h(X) - \bar{F}(Y)\|_2^2 = \|\beta(Y)D(Y)(u_h(Y)-u(Y))\|_2^2\nonumber\\
  &=& \beta_0^2\|\mathcal{L}_0(u_h-u)\|^2_{\ell_2(\partial\Omega_0)} +   \beta_1^2\|\mathcal{L}_1(u_h-u)\|^2_{\ell_2(\partial\Omega_1)} +  \beta_2^ 2\|\mathcal{L}_2(u_h-u)\|^2_{\ell_2(\Omega)}.
\end{eqnarray}
Comparing with the residual of the continuous problem~\eqref{eq:contres} and the continuous bilinear form~\eqref{eq:bilinearcont}, we introduce the discrete bilinear form
\begin{equation}
  a_h(u,v) = \beta_0^2\left(u,v\right)_{\ell_2(\partial\Omega_0)} + \beta_1^2\left(\frac{\partial u}{\partial n},\frac{\partial v}{\partial n}\right)_{\ell_2(\partial\Omega_1)} + \beta_2^2\left(\Delta u,\Delta v\right)_{\ell_2(\Omega)}.
  \label{eq:ah}
\end{equation}
So far, we have assumed that the Dirichlet boundary conditions are enforced in the least squares sense. It has been noted, e.g, in~\cite{PlaDri06} that when Dirichlet conditions are imposed strongly, the overall accuracy is improved. Assuming that there are node points $X_{\partial\Omega_0}\subset X$ that discretize the Dirichlet boundary, we let $\tilde{X}=X\setminus X_{\partial\Omega_0}$ and $u_h(X_{\partial\Omega_0})=0$. Then we rewrite the discretized least squares PDE problem~\eqref{eq:discretePDE} as:
\begin{equation}
    \bar{D}_h(Y,\tilde{X})u_h(\tilde{X}) = \bar{F}(Y)-\bar{D}_h(Y,X_{\partial\Omega_0})u_h^0(X_{\partial\Omega_0})\equiv\tilde{F}(Y),
\label{eq:nonhomres}  
\end{equation}  
where $u_h^0(X_{\partial\Omega_0})=f_0(X_{\partial\Omega_0})$ is a subset of the Dirichlet boundary data.
Note that $u_h$ is in general non-zero at the Dirichlet boundary between the data points. 
\new{We denote the trial space containing all functions of the form~\cref{eq:loceval} by $V_h$ and we denote the subspace with zero Dirichlet data by $V_h^0$.} The solution to the original problem is given by $u_h+u_h^0\in V_h$, where $u_h\in V_h^0$. 
%
Similarly to~\eqref{eq:lscont}, we write the least squares problem on the form
\begin{equation}
  u_h = \argmin_{v_h\in V_h^0}a_h(v_h+u_h^0-u,v_h+u_h^0-u),
  \label{eq:ah1}
\end{equation}
where $V_h$ is the RBF-FD trial space. We have the orthogonality property
\begin{equation}
  a_h(u_h+u_h^0-u,v_h) = 0,\quad \forall v_h\in V_h^0.
 \label{eq:ah2}
\end{equation}  
To see how this relates to the matrix-based description of the discrete least squares problem, we introduce a (non-orthogonal) basis for $V_h$. For each evaluation point $y$ there is a unique representation of $u_h$ in terms of the local cardinal functions~\eqref{eq:loceval}. We define global cardinal functions as
\begin{equation}
  \Psi_j(y) = \left\{\begin{array}{ll}
    \psi_{i}^{k}(y), &x_j\in X_k\\
    0,  & x_j\not\in X_k,
  \end{array}\right .
\end{equation}
where $k=k(y)$ is the stencil selected for the evaluation point $y$, and $i(j)$ is the local index $i$ in $X_k$ of $x_j\in X$. We represent a non-homogeneous function $u_h\in V_h$ as
\begin{equation}
  u_h(y) = \sum_{j=1}^Nu_h(x_j)\Psi_j(x).
  \label{eq:uhpsi}
\end{equation}
We note that $D_h(y_i,x_j)=D(y_i)\Psi_j(y_i)$ and $E_h(y_i,x_j)=\Psi_j(y_i)$. If we insert~\eqref{eq:uhpsi} in~\eqref{eq:ah2} and let $v_h=\Psi_i$, we get
\begin{equation}
  \sum_{j=1}^N a_h(\Psi_j,\Psi_i) u_h(x_j) = a_h(u,\Psi_i),
\end{equation}  
where $a_h(\Psi_j,\Psi_i)$ is an element of the matrix $\bar{D}_h(Y,X)^T\bar{D}_h(Y,X)$, and $a_h(u,\Psi_i)$ is an element of the right-hand-side vector $\bar{D}_h(Y,X)^T\bar{F}(Y)$ in the weighted normal equations. The specific properties of the trial space and the cardinal basis functions are further discussed in the following section. 
\section{The discontinuous trial space}\label{sec:trial}
The trial space $V_h$ is a piecewise space. The stencil selection algorithm that we use for the evaluation points results in the domain being divided into Voronoi regions $\mathcal{V}_k$ around each stencil center point $x_k\in X$. For an illustration of the Voronoi regions in two dimensions, see~\cref{fig:experiments:Poisson:exactsolution}. Locally we have $u_h\in W_\infty^2(\mathcal{V}_k)\subset W_2^2(\mathcal{V}_k)$ due to the smoothness of the at least cubic PHS basis.
\begin{theorem}
  Assume that for each stencil underlying the trial space approximation, the node set $X_k$ is unisolvent with respect to polynomials of degree $p$. Then $u_h|_{\mathcal{V}_k}=0$ if and only if $u_h(X_k)=0$, and $u_h|_{\Omega}=0$ if and only if $u_h(X)=0$.
  \label{theor:v0}
\end{theorem}
\begin{proof}
  The results follow from the uniqueness of the local interpolation problems.
\end{proof}  

%
A scattered node set is quantified by its fill distance $h$, measuring the radius of the largest ball empty of nodes in $\Omega$, and its separation distance $q$, defined by:
\begin{equation}
  h= \sup_{x\in\Omega}\min_{x_j\in X}\|x-x_j\|_2\geq q= \frac{1}{2}\min_{\stackrel{j\neq k}{x_j,x_k\in X}}\|x_j-x_k\|_2.
\end{equation}
\new{The quality of a node set is related to $c_q=q/h< 1$. The trial space approximation improves with increasing node quality. }
 \new{In the following subsections, we derive the results that we need for the error estimates in~\cref{sec:theor}, in terms of the fill distance $h$ of $X$, the fill distance $h_y$ of $Y$, and the node quality $c_q$.}
\subsection{\new{Interpolation errors}}
We define the interpolant $I_h(u)\in V_h$ of a function
$u$ as $I_h(u) = \sum_{j=1}^Nu(x_j)\Psi_j$. \nnew{For a function $u\in W_\infty^{p+1}(\Omega)$ that allows Taylor series expansion around $x_k\in\mathcal{V}_k$, we can assess the local interpolation error $e_I=I_h(u)-u$, and its derivatives using a result from~\cite{Bayona19}. When $h$ is small enough, we have that}
\begin{equation}
  |\mathcal{L}_i\left(I_h(u(y))-u(y)\right)|\leq \alpha_{k,i}h^{p+1-i}|u|_{W_\infty^{p+1}(\mathcal{V}_k)},\quad y\in \mathcal{V}_k,
  \label{eq:errors:interpolationError}
\end{equation}  
where $\mathcal{L}_i$ is a differential operator of order $i$, $|u|_{W_\infty^{q}(\Omega)}=\sum_{|\alpha|=q} \|D^\alpha u\|_{L_\infty(\Omega)}$, and $\alpha_{k,i}$ are constants that depend on the degree $p$ of the polynomial basis, and on the node quality $c_q$ of the stencil node set $X_k$. If the node layout is non-uniform, indicated by a small value of $c_q$, the interpolation problem has a large Lebesgue constant~\cite{Shankar17}, and consequently a larger interpolation error. The error is also larger for skewed stencils that are evaluated close to their support boundary.

When we use the interpolation error in the global error estimate, we take a norm over the domain. If we let $\alpha_i=\nnew{|\Omega|}\max_k\alpha_{k,i}$, we have
\begin{equation}
  \|\mathcal{L}_i(\new{I_h(u)}-u)\|_{\ell_2(\Omega)}\leq \alpha_ih^{p+1-i}|u|_{W_\infty^{p+1}(\Omega)}.
  \label{eq:eint}
\end{equation}

%
At the edge of a Voronoi region $I_h(u)$ takes slightly different values from each side. 
\new{That is, if $u$ is not represented exactly in $V_h$, the interpolant $I_h(u)$}  has a discontinuity proportional to $h^{p+1}$, that goes to zero as the space is refined, along the edges of the Voronoi regions. This means that the cardinal basis functions also have discontinuities between Voronoi regions, see~\cref{fig:method:cardinal_1}. 
\begin{figure}[htb!]
	\centering
	\begin{tabular}{cc}
	\includegraphics[width=0.35\linewidth]{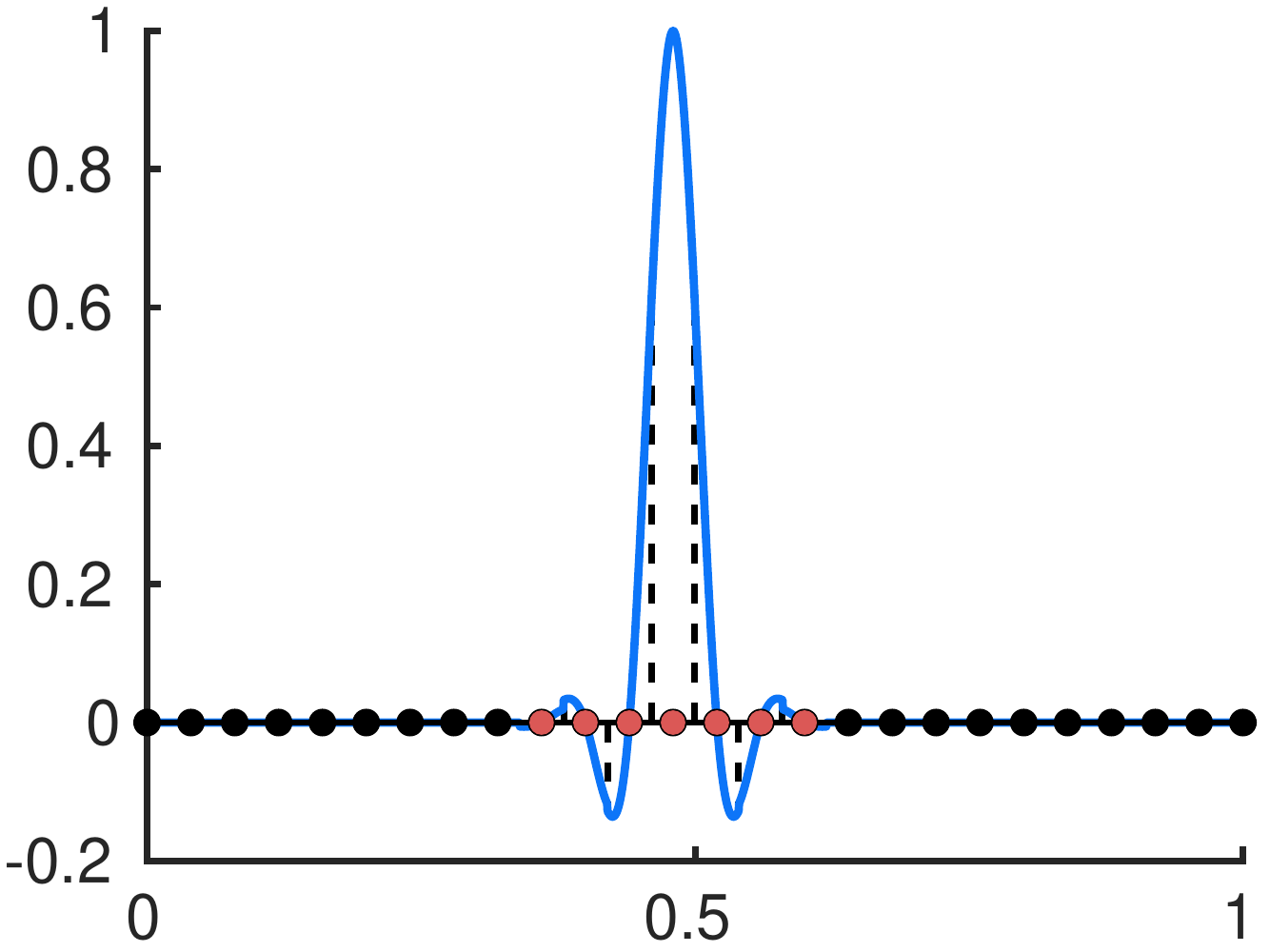}
	\includegraphics[width=0.35\linewidth]{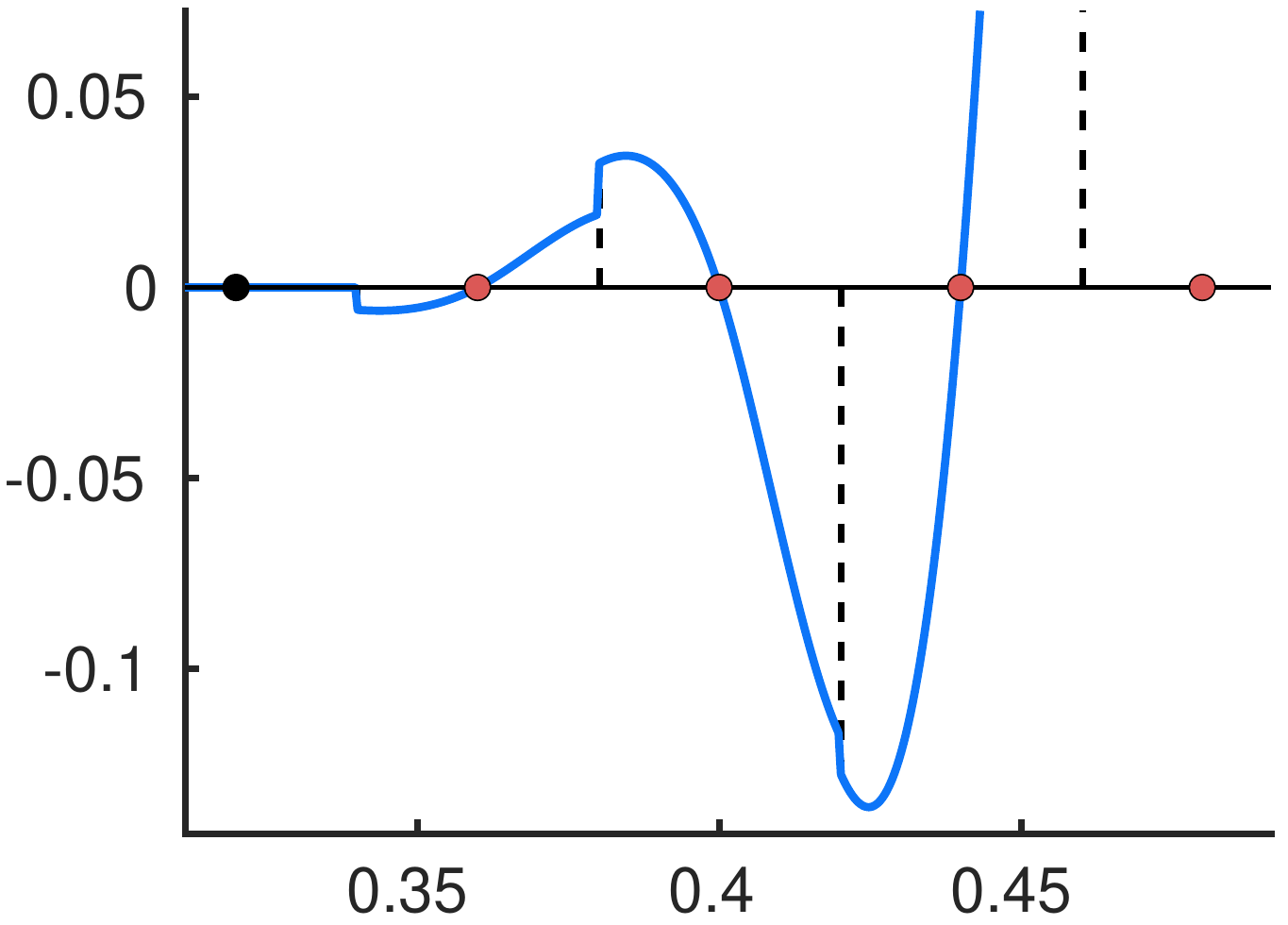}

	\end{tabular}
	\caption{Left: A cardinal function generated with RBF-FD on a uniform node set in one dimension with stencil size $n=7$. 
	Right: A close up to illustrate the discontinuities at the intersections of the Voronoi regions.}
	\label{fig:method:cardinal_1}
\end{figure}

\subsection{\new{Derivatives and norms in the trial space}} We want to reuse the results that we derive here also for the smoothed trial space defined in~\cref{sec:smooth}. Therefore, we define the local stencil-based functions $v_h^k(y)$ that together form $v_h$ (cf.~\cref{eq:loceval}) over the domains $\mathcal{V}_k^\delta\supset \mathcal{V}_k$, which represent an extension of the Voronoi regions $\mathcal{V}_k$ with at most distance $\delta h/2$ in any direction, where $\delta<1$. We have 
\begin{equation}
  v_h^k(y)= \sum_{i=1}^n\psi_i^k(y)v_h(x^k_i), \quad y\in \mathcal{V}_k^\delta,
  \label{eq:locv}
\end{equation}
where $x_i^k$ is the $i$th element of $X_k$. When discussing bounds, we also use the following explicit form derived in~\cite{Bayona19}:
\begin{equation}
  v_h^k(y)= v_h(X_k)^T(I-WP)A^{-1}\phi(y)+v_h(X_k)^T Wp(y),
  \label{eq:Bayona}
\end{equation}
where $W=A^{-1}P(P^T A^{-1} P)^{-1}$, the PHS vector $\phi(y)=(\phi^k_1(y),\ldots,\phi^k_n(y))^T$, and the polynomial vector $p(y)=(p_1(y),\ldots,p_m(y))^T$.
%
%
\begin{theorem}
  \label{theor:Dv}
  Assume that we use cubic splines, that~\cref{theor:v0} holds and that the node quality $c_q$ of $X$ has a lower bound $c_{\min}$, within each stencil node set, for a sequence of discretizations with different fill distances $h$. We define $\tilde{y}=y/h$ and $\tilde{v}_h^k=\sum_{i=1}^n\tilde{\psi}_i^k(\tilde{y})v_h(x^k_i)$, where $ \tilde{\psi}_i^k(\tilde{y}) =\psi_i^k(y)$, we define a scaled Voronoi region $\tilde{\mathcal{V}}_k$ such that $\tilde{y}\in \tilde{\mathcal{V}}_k$ when $y\in \mathcal{V}_k$. Let $\mathcal{V}_{k,i}=\mathcal{V}_k\cap\Omega_i\subset\mathbb{R}^{d_i}$ and let $\tilde{\mathcal{V}}_{k,i}$ be the scaled counterpart. Then
  \begin{eqnarray}
    (D^s_yv_h^k,D^t_yv_h^k)_{L_2(\mathcal{V}_{k,i})}&=&h^{d_i-|s|-|t|}(D^s_{\tilde{y}}\tilde{v}_h^k,D^t_{\tilde{y}}\tilde{v}_h^k)_{L_2(\tilde{\mathcal{V}}_{k,i})}\\
    &=& h^{d_i-|s|-|t|}v_h(X_k)^T\tilde{\Psi}^{s,t}_{k,i} v_h(X_k),\quad |s|,|t|\leq 3\nonumber
  \end{eqnarray}
  where $\tilde{\Psi}^{s,t}_{k,i}(q,r)=(D^s_{\tilde{y}}\tilde{\psi}_q^k,D^t_{\tilde{y}}\tilde{\psi}_r^k)_{L_2(\tilde{\mathcal{V}}_{k,i})}$. Furthermore for $|s|,|t|,|q|\leq 3$,
  \begin{equation}
    \max_{\begin{array}{c}\scalebox{0.8}{$\|v_h(X_k)\|_2>0$}\\\scalebox{0.8}{$\|D^qv_h^k\|_{L_2(\mathcal{V}_k)>0}$}\end{array}} \frac{(D^sv_h^k,D^tv_h^k)_{L_2(\mathcal{V}_{k,i})}}{\|D^qv_h^k\|_{L_2(\mathcal{V}_k)}^2}\leq h^{d_i-d+2q-|s|-|t|}K^{s,t,i}_q,\quad k=1,\ldots,N.
    \label{eq:ratbound}
  \end{equation}  
  where $K_q^{s,t,i}$ is an upper bound for the largest eigenvalue $K_{\max}$ of the generalized eigenproblem $\tilde{\Psi}^{s,t}_{k,i}v=K\tilde{\Psi}^{q,q}_{k,0}v$ in the subspace of eigenvectors $v\in \mathbb{R}^n$ that are in the range of $\tilde{\Psi}^{q,q}_{k,0}$. Apart from $s$, $t$ and $q$, the bound $K_q^{s,t,i}$ only depends on the lower bound on the node quality $c_{\min}$, the dimension $d$, the stencil size $n$, the polynomial degree $p$, and the dimension of the polynomial space $m$.
\end{theorem}  
\begin{proof}
  For the scaling of the derivatives, we have $D^s_y\psi_i^k(y)=D_y^s\tilde{\psi}_i^k(y/h) = h^{-s}D^s_{\tilde{y}}\tilde{\psi}_i^k(\tilde{y})$. The integration relation is $\int_{\mathcal{V}_k}v(y)dy=\int_{\tilde{\mathcal{V}}_k}v(\tilde{y}) h^{-d}\,d\tilde{y}$. For integration along a boundary intersecting $\mathcal{V}_k$, we instead get $h^{d-1}$. To understand which factors influence the bounds, we introduce the matrices $A_1$ and $P_1$, which correspond to a scaling of the nodes such that $h=1$, and note that $A=h^3A_1$ and $P=P_1H$, where $H$ is a diagonal matrix with elements $h_{jj}=h^{|t_j|}$, where $p_j(x)=x^{t_j}$. Combining~\cref{eq:Bayona} and the scaled matrices leads to
\begin{equation}
  h^{|s|}D^s v_h^k(y)= v_h(X_k)^T(I-W_1P_1^T)A_1^{-1}D^s_{\tilde{y}}\phi(\tilde{y})+v_h(X_k)^T W_1D^s_{\tilde{y}}p(\tilde{y}).
  \label{eq:matform}
\end{equation}
We note the right hand side of~\cref{eq:matform} is scale invariant. The smallest eigenvalue of the symmetric matrix $A_1$ can be bounded in terms of the separation distance $q$ of the stencil node set~\cite[corollary 12.7]{Wend05}. When the nodes are scaled such that $h=1$, we have $q=c_q$. That is, by choosing  $q=c_{\min}$ for the bound, it holds for all discretizations in the sequence. All other matrices and vectors can be bounded in terms of the stencil size $n$, the dimension of the polynomial basis $m$, the polynomial degree $p$, and the derivative $s$. The bounds are larger for skewed stencils and grow with $p$.
\nnew{An upper bound for the numerator of~\cref{eq:ratbound} can be obtained by  directly bounding the matrices and vectors in~\cref{eq:matform}. The denominator is a mass matrix/stiffness matrix for the local scaled cardinal functions. The lower bound is attained for the (non-zero differentiated) cardinal function that has the smallest amount of mass in $\mathcal{V}_k$. The mass is smaller for skewed stencils as well as small Voronoi regions. The same parameters, $c_{\min}$, $n$, $m$, $p$ determine the behaviour.}
The local trial space functions are twice continuously differentiable, while the third derivatives exist at all points, but are piecewise continuous.
\end{proof}

To investigate the relations between the discrete and continuous norms (cf.~\cref{sec:discretenorm}), and the smoothed and discontinuous trial spaces (cf.~\cref{sec:smooth}), we introduce a semi-discrete bilinear form $a^*(v,v)$ using the norm
\begin{equation}
  \|v\|_{L_2^*(\Omega_i)}^2=(v,v)_{L_2^*(\Omega_i)}=\sum_{k=1}^N \int_{\mathcal{V}_k\cap\Omega_i}v(y)^2\,dy.
\end{equation}  
For a function $v\in W_2^2$ $\|\mathcal{L}_i v\|_{L_2^*(\Omega)}=\|\mathcal{L}_i v\|_{L_2(\Omega)}$, $i\leq 2$, while for the trial space functions, the semi-discrete norm eliminates the derivatives of the jumps.

\subsection{\new{The smoothed trial space}}
\label{sec:smooth}
We introduce a smoothing operator $S_2^0:V_h^0\mapsto V^0$, where
$V^0=\{v\in W_\infty^2(\Omega)\,|\, v(y)|_{\partial\Omega_0}=0\}$, and let $v=S_2^0(v_h)$. We never construct the smoothed trial space function $v$ in practice, but we use it as a tool in the error analysis (cf.~\cref{sec:theor}). First we define a set of overlapping patches by extending each Voronoi region $\mathcal{V}_k$ a distance $\delta h/2$, $\delta<1$ in the normal direction at each interior edge/face. To handle the Dirichlet boundary, we imagine a mirrored set of Voronoi regions outside the boundary and let these extend $\delta h$ into the domain. The Voronoi regions along the inside of the Dirichlet boundary are assumed to conform to the domain boundary where it intersects the region. As in the previous subsection, we denote the extended Voronoi regions by $\mathcal{V}_k^\delta$.

Then we construct a set of non-negative, compactly supported, partition of unity weight functions $\{\omega_k\}_{k=1}^{N+N_0}$, where the extra $N_0$ weight functions belong to the regions outside the Dirichlet boundary. We let $\omega_k\in W_\infty^2(\Omega)$ be supported on $\mathcal{V}_k^\delta$. That is, the weight functions overlap with a distance $\delta h$ at all shared edges. Let $\mathcal{V}_k^{-\delta}$ denote the interior part of the Voronoi region, and let $\Gamma_k^\delta=\mathcal{V}_k\setminus \mathcal{V}_k^{-\delta}$. For the weight functions it holds that $\sum_{k=1}^{N+N_0}\omega_k=1$, $y\in\Omega$, $\omega_k(y)=1$, $y \in \mathcal{V}_k^{-\delta}$, $D^s\omega_k(y)=0$, $y\in \mathcal{V}_k^{-\delta}$, and
\begin{equation}
  |D^s \omega_k| \leq \frac{G_{|s|}}{ h^{|s|}},\quad |s|\leq 3,\quad y\in \mathcal{V}_k^\delta\setminus \mathcal{V}_k^{-\delta},
  \label{eq:wderivs}
\end{equation}
where $G_0=1$. The second derivatives of the weight functions are continuous, while the third derivatives exist, but are only piecewise continuous. For further details on the construction of weight functions and their properties, see~\cite{Wend02,LaShcher17}. 
%

We connect the exterior weight functions with zero functions, i.e., $v_h^k\equiv 0$, $k>N$. Then we combine the local interpolants and weight functions to get
%
\begin{eqnarray}
  v&=&S_2^0(v_h)=\sum_{j=1}^N \omega_j(y)v_h^j(y)=\sum_{j\in J_k}^N\omega_j(y)v_h^j(y),  \quad y\in \mathcal{V}_k,\nonumber
\end{eqnarray}
where $J_k$ contains the indices of the patches that overlap with $\mathcal{V}_k$. 
\begin{theorem}
  \label{theor:S}
  Assume that~\cref{theor:Dv} holds, that $\delta<1$ is fixed for all discretizations and that $v=S_2^0(v_h)$ for $v_h\in V_h^0$.
Then the following relations hold:
  \begin{equation}
  \|v_h\|_{L_2(\Omega)}^2\leq (1+\eta_0)\|v\|_{L_2(\nnew{\Omega})}^2,
  \label{eq:eta0}
\end{equation}
\begin{equation}
  a(v,v)\leq (1+\eta_a)a^*(v_h,v_h)
  \label{eq:eta}
\end{equation} 
where the bounds $\eta_0$ and $\eta_a$ depend on $\delta$, the node quality $c_{\min}$, the dimension $d$, the stencil size $n$, the polynomial degree $p$, and the dimension of the polynomial space $m$.
\end{theorem}
\begin{proof}
  If $\|v_h\|_{L_2(\mathcal{V}_k)}>0$, then $\|v\|_{L_2(\mathcal{V}_k)}>0$ and $\|v\|_{L_2(\mathcal{V}_k^{-\delta})}>0$, and we use~\cref{theor:Dv} to bound the ratio
\begin{equation}
  \frac{\|v_h\|_{L_2(\Gamma_k^\delta)}^2}{\|v\|_{L_2(\mathcal{V}_k^{-\delta})}^2}=\frac{\|v_h\|_{L_2(\Gamma_k^\delta)}^2}{\|v_h\|_{L_2(\mathcal{V}_k^{-\delta})}^2}=\frac{v_h(X_k)^T\tilde{\Psi}_{k,\delta}^{I,I} v_h(X_k)}{v_h(X_k)^T\tilde{\Psi}_{k,-\delta}^{I,I} v_h(X_k)}\leq \eta_{k,0},
\end{equation}
where the subscript $k,\pm\delta$ denotes integration over \nnew{$\Gamma_{k}^\delta$ and $\mathcal{V}_k^{-\delta}$}.
Due to the size ratio of the domains, $\eta_{k,0}$ is approximately proportional to $\delta/(1-\delta)$. The other dependencies come from the basis functions, (cf.~\cref{theor:Dv}). We sum the local results to get 
\begin{eqnarray}
  \|v_h\|^{\nnew{2}}_{L_2(\Omega)}&=&\sum_{\nnew{k=1}}^{\nnew{N}}\|v_h\|^{\nnew{2}}_{L_2(\mathcal{V}_k^{-\delta})} + \|v_h\|^{\nnew{2}}_{L_2(\Gamma_k^{\delta})}\leq \sum_{k=1}^N(1+\eta_{k,0})\|v\|^{\nnew{2}}_{L_2(\mathcal{V}_k^{-\delta})}\nonumber\\
  &\leq& \max_{\nnew{k}} (1+\eta_{k,0})\|v\|^{\nnew{2}}_{L_2(\Omega)}.
\end{eqnarray}
Setting $\eta_0=\max_{k}(1+\eta_{k,0})$ gives the result~\cref{eq:eta0}.
\nnew{For the bilinear form, we start from the restriction to a Voronoi region. We have
\begin{eqnarray}
  a_k(v,v)&=&a(v,v)|_{\mathcal{V}_k}=a_k(v,v)|_{\mathcal{V}_k^{-\delta}}+a_k(v,v)|_{\Gamma_k^{\delta}}\label{eq:415}\\
  &=&a_k(v_h^k,v_h^k)|_{\mathcal{V}_k^{-\delta}} +a_k(\sum_{j\in J_k}\omega_{\nnew{j}}v_h^j,\sum_{j\in J_k}\omega_{\nnew{j}}v_h^j)|_{\Gamma_k^{\delta}}\nonumber\\
  &=&a_k(v_h^k,v_h^k)|_{\mathcal{V}_k^{-\delta}} + \sum_{i\in J_k}\sum_{j\in J_k}a_k(\omega_{\nnew{i}}v_h^i,\omega_{\nnew{j}}v_h^j)|_{\Gamma_k^{\delta}\cap\mathcal{V}_i^\delta\cap\mathcal{V}_j^\delta}\nonumber\\
  &\leq&a_k(v_h^k,v_h^k) + Q\sum_{j\in J_k}a_k(\omega_{\nnew{j}}v_h^j,\omega_{\nnew{j}}v_h^j)|_{\Gamma_k^{\delta}\cap\mathcal{V}_j^\delta},\nonumber
\end{eqnarray}
where $Q\approx 2^d$ is the largest number of extended Voronoi regions that overlap at any given point. We note that $\omega_k|_{\mathcal{V}_{k}^\delta\cap\partial\Omega_0}=0$, we introduce the notation $\Gamma_{k,j}^\delta=\Gamma_k^{\delta}\cap\mathcal{V}_j^\delta$ and $\Gamma_{k,j,1}^\delta=\Gamma_k^{\delta}\cap\mathcal{V}_j^\delta\cap\partial\Omega_1$,  and use~\cref{eq:wderivs} to estimate one term in the sum in~\eqref{eq:415}as
\begin{eqnarray}
  a_k(\omega_{\nnew{j}}v_h^j,\omega_{\nnew{j}}v_h^j)|_{\Gamma_{k,j}^{\delta}}&=&
  \|\Delta(\omega_jv_h^j)\|^2_{L_2(\Gamma_{k,j}^{\delta})} +
  \|\partial(\omega_jv_h^j)/\partial n\|^2_{L_2(\Gamma_{k,j,1}^{\delta})}\nonumber\\
  &\leq& 2\left(\|\Delta v_h^j\|^2_{L_2(\Gamma_{k,j}^{\delta})}
  +2G_1^2h^{-2}\|\nabla v_h^j \|^2_{L_2(\Gamma_{k,j}^{\delta})}
  + G_2^2h^{-4}\|v_h^j\|^2_{L_2(\Gamma_{k,j}^{\delta})}
  \right)\nonumber\\
  &+& 2\left(\|\partial v_h^j/\partial n\|^2_{L_2(\Gamma_{k,j,1}^{\delta})} +  G_1^2h^{-2}\|v_h^j\|^2_{L_2(\Gamma_{k,j,1}^{\delta})}\right).
  \label{eq:aknorms}
\end{eqnarray}
%
%
Then, we consider the quotient of the overlap terms in $\mathcal{V}_k$ and the bilinear form evaluated over the stencils surrounding and including $\mathcal{V}_k$. We use~\cref{theor:Dv} to write the norms in~\eqref{eq:aknorms} on scale invariant form:
%
%
\begin{eqnarray}
  \frac{\sum_{j\in J_k}a_k(\omega_{\nnew{j}}v_h^j,\omega_{\nnew{j}}v_h^j)|_{\Gamma_{k,j}^{\delta}}}{\sum_{j\in J_k}a_j(v_h^j,v_h^j)}&\leq&
  \frac{2\sum_{j\in J_k}v_h(X_j)^T\left(\tilde{\Psi}_{k,j}^{\Delta,\Delta}+2G_1^2\tilde{\Psi}_{k,j}^{\nabla,\nabla}+G_2^2\tilde{\Psi}_{k,j}^{I,I}
    \right) v_h(X_j)}{\sum_{j\in J_k}v_h(X_j)^T\left(\tilde{\Psi}_{j}^{\Delta,\Delta}+h\tilde{\Psi}_{j,1}^{\mathcal{L}_1,\mathcal{L}_1} +h^3\tilde{\Psi}_{j,0}^{I,I}\right) v_h(X_j)}\nonumber\\
 & + &  \frac{2h\sum_{j\in J_k}v_h(X_j)^T\left(\tilde{\Psi}_{j,k,1}^{\mathcal{L}_1,\mathcal{L}_1}
    + G_1^2\tilde{\Psi}_{j,k,1}^{I,I}\right) v_h(X_j)}{\sum_{j\in J_k}v_h(X_j)^T\left(\tilde{\Psi}_{j}^{\Delta,\Delta}+h\tilde{\Psi}_{j,1}^{\mathcal{L}_1,\mathcal{L}_1} +h^3\tilde{\Psi}_{j,0}^{I,I}\right) v_h(X_j)},
\label{eq:417}  
\end{eqnarray}
where we extended the notation for $\tilde{\Psi}$ to allow integration over $\Gamma_{j,k}$ as in~\eqref{eq:aknorms} and to include composite scalar operators, when the order is the same in each term. The denominator is zero if the bilinear form is zero over all involved stencils. However, with the PHS and polynomial basis functions that we use, this implies that the data is sampled from a polynomial of degree $\leq p$ in the nullspace of the Laplacian operator. Due to the overlap of the stencils and polynomial unisolvency, the data is consistent across stencils. The polynomial is represented exactly on all stencils and we have  
$a_k(v,v)=a_k(v_h^k,v_h^k)$. If the denominator is positive, similarly as in~\cref{theor:Dv}, we can find an upper bound through the generalized eigenproblem on the extended domain involving a Voronoi region and its neighbours. We denote the specific bound for $\mathcal{V}_k$ by $\eta_{k,a}$ and combine~\eqref{eq:415} and~\eqref{eq:417}, resulting in
%
%
{\small
\begin{equation*}
  a(v,v)=\sum_{k=1}^N a_k(v,v) \leq \sum_{k=1}^N\left( a_k(v_h^k,v_h^k) + 2Q\eta_{k,a}\sum_{j\in J_k}a_j(v_j^h,v_j^h)\right)\leq (1+\eta_a)a^*(v_h,v_h),
\end{equation*}
}
where $\eta_a=2Q\max_k \eta_{k,a}|J_k|$ provides the result~\eqref{eq:eta}.}
\end{proof}

\subsection{Discrete norm errors}\label{sec:discretenorm}
The discrete norm \new{$\|\cdot\|_{\ell_2(\Omega)}$} on the set of nodes $Y=\{y_i\}_{i=1}^M$ is an approximation of \new{the continuous norm $\|\cdot\|_{L_2(\Omega)}$, and for the global error estimate, we need to quantify the difference. We start from a generic integral:} 
\begin{equation}
\mathcal{I}=\int_\Omega f(y)dy = \frac{|\Omega|}{M}\sum_{i=1}^M f(y_i) + \gamma_I(f)=\mathcal{I}_h+\gamma_I(f).
\end{equation}
We want to use this relation for non-trivial domains, which means that we need to consider scattered nodes. Even if the nodes are regular in parts of the domain, they need to be somewhat irregular near the boundary. A very general error estimate for scattered node quadrature is given by 
\begin{equation}
  |\gamma_I(f)|\leq D_M(Y)V(f,\Omega),
  \label{eq:mc}
\end{equation}  
where $D_M(Y)$ is the star discrepancy of the node set and $V(f,\Omega)$ is the Hardy-Krause variation of $f$~\cite{APST17}. This has been shown for general domains and piecewise smooth functions in~\cite{BCGT13a,BCGT13b}. Both of the factors in~\eqref{eq:mc} are hard to quantify in general. However, the standard deviation of the error for an arbitrary node layout (Monte Carlo integration) in practical cases decreases as $\mathcal{O}(1/\sqrt{M})$ and for a low discrepancy (quasi random) node layout, it decreases as $\mathcal{O}((\log M)^d/M)$.
%
Furthermore, for a (piecewise) differentiable function, the total variation can be mesured as $V(f,\Omega)=\int_\Omega|\nabla f|dy$. Assuming that the evaluation node set $Y$ is quasi uniform and that $d$ is small enough for $(\log M)^d$ to be viewed as almost constant, we can estimate the error in the squared norm of a function $v\in W_2^1$  as
\begin{equation}
  \left|\|v\|_{L_2(\Omega_i)}^2 - \|v(Y)\|^2_{\ell_2(\Omega_i)}\right|\leq C_Ih_y^{d_i}\int_{\nnew{\Omega_i}}|\nabla (v^2)|dy,
  \label{eq:quaderr}
\end{equation}  
where $d_i$ is the dimensionality of $\Omega_i$. \new{To apply this estimate to the trial space function it is easiest to apply it locally to each Voronoi region. The number of points in each region is then small, but we get a statistical averaging through the sum.}
\begin{theorem}
  \label{theor:hy}
  %
  %
If~\cref{theor:Dv} holds and $h_y$ is chosen according to~\cref{eq:hy} \nnew{for a relative integration error tolerance $\tau<1$}, then
   \begin{equation}
     \|v_h(Y)\|_{\ell_2(\Omega)}\leq(1+\tau)\|v_h\|_{L_2(\Omega)},
     \label{eq:tau0}
   \end{equation}
   \begin{equation}
     a^*(v_h,v_h)\leq\frac{1}{(1-\tau)}a_h(v_h,v_h).
     \label{eq:tau2}
   \end{equation}
\end{theorem}  
\begin{proof}
  %
  %
  We use~\cref{theor:Dv} for the gradients
  \begin{equation}
    \int_{\mathcal{V}_{k,i}}|\nabla((\mathcal{L}_iv_h^k)^2)|dy =
    2(|\nabla\mathcal{L}_iv_h^k|,|\mathcal{L}_iv_h^k|)_{\mathcal{V}_{k,i}}
    = h^{d_i-2i-1}v_h(X_k)^T \tilde{\Psi}_{k,i}^{|\nabla\mathcal{L}_i|,|\mathcal{L}_i|}v_h(X_k)
  \end{equation}  
  and then again for the relative errors when $\|v_h\|_{L_2(\mathcal{V}_k)}>0$ and $a_{k,k}^*(v,v)>0$ to get
  \begin{eqnarray}
\frac{\left|\|v_h\|_{L_2(\mathcal{V}_k)}^2-\|v_h(Y)\|^2_{\ell_2(\mathcal{V}_k)}\right|}{\|v_h\|_{L_2(\mathcal{V}_k)}^2}\leq 2C_I\frac{h_y^d}{h} K_I^{|\nabla|,|I|}.
  \label{eq:relnorm}
\end{eqnarray}
%
\begin{eqnarray}
  \lefteqn{\frac{\left|\sum_{i=0}^2\left(\|\mathcal{L}_iv_h\|_{L_2(\mathcal{V}_{k,i})}^2 -\|\mathcal{L}_iv_h(Y)\|_{\ell_2(\mathcal{V}_{k,i})}^2\right)\right|}{\sum_{i=0}^2\|\mathcal{L}_iv_h\|_{L_2(\mathcal{V}_k,i)}^2}}\\
  &\leq&2C_I\frac{h_y^d}{h}\frac{v_h(X_k)^T(\tilde{\Psi}_k^{|\nabla\Delta|,|\Delta|}+h_y^{-1}h\tilde{\Psi}_{k,1}^{|\nabla\mathcal{L}_1|,|\mathcal{L}_1|}+h_y^{-1}h^3\tilde{\Psi}_{k,0}^{|\nabla|,|I|}) v_h(X_k)}{v_h(X_k)^T(\tilde{\Psi}_k^{\Delta,\Delta}+h\tilde{\Psi}_{k,1}^{\mathcal{L}_1,\mathcal{L}_1}+h^3\tilde{\Psi}_{k,0}^{I,I}) v_h(X_k)} \nonumber\\
   &\leq& 2C_I\left(\frac{h_y^d}{h}K_\Delta^{|\nabla\Delta|,|\Delta|} + h_y^{d-1}K_\Delta^{|\nabla\mathcal{L}_1|,|\mathcal{L}_1|,1}+ h_y^{d-1}h^2K_\Delta^{|\nabla|,|I|,0} \right). \nonumber
  \label{eq:relbilinear}
\end{eqnarray}
The first coefficient is the critical one, so if we choose
\begin{equation}
  h_y\leq \left(\frac{\tau h}{2C_I\max\left(K_I^{|\nabla|,|I|},K_\Delta^{|\nabla\Delta|,|\Delta|}\right) }\right)^\frac{1}{d},
\label{eq:hy}  
\end{equation}  
then the relative error is in both cases bounded by $\tau$ locally and globally.
\end{proof}

\section{Convergence and error estimates for RBF-FD-LS}
\label{sec:theor}


In this section, we derive stability, convergence and error estimates for the RBF-FD-LS method. When solving the least squares problem numerically in the form~\eqref{eq:ls}, we have not experienced practical problems with well-posedness. However, from the theoretical perspective it simplifies the analysis to have the Dirichlet boundary conditions imposed strongly as in~\eqref{eq:nonhomres}, which means that the error is zero at the node set $X_{\partial\Omega_0}$. When performing the analysis, we assume that the weights $\beta_i=1$, $i=0,\,1,\,2$ in~\eqref{eq:method:scaling}. Scaling is discussed separately in \cref{sec:errors:scaling}. \new{Before stating the global error estimate, we prove coercivity for the continuous bilinear form in a homogeneous space, and then relate the discrete bilinear form to the continuous bilinear form.}


\subsection{Coercivity of the continuous bilinear form in a homogeneous space}
\new{We investigate the coercivity of the continuous bilinear form~\eqref{eq:bilinearcont} for functions $v\in V^0=\{v\in W_2^2(\Omega)\,|\, v(y)|_{\partial\Omega_0=0\}}$.} 
Given the smoothness assumptions on the domain $\Omega$ and the function $v$, a Poin\-car\'e-Friedrich inequality holds with the boundary data given on some part of the boundary~\cite{BouCha07}. This can be seen if the inequality is shown using integration along paths from points on the Dirichlet boundary to points in the domain.
\begin{equation}
  \|v\|_\Omega^2\leq C_P^2(\|v\|_{\partial\Omega_0}^2+\|\nabla v\|_\Omega^2)=C_P^2\|\nabla v\|_\Omega^2,\quad \new{v\in W_2^1(\Omega)}.
  \label{eq:pc}
\end{equation}
%
We also need a trace inequality that relates the solution on (any part of) the boundary to the solution in the interior. \new{The following inequality~\cite[Theorem 1.6.6]{BreSco08}, \cite[Theorem A.4]{LaTho03} holds for domains with Lipschitz or smooth boundary:
\begin{equation}
  \|v\|_{\partial\Omega_i}^2\leq C_T^2\|v\|_\Omega(\|v\|_\Omega^2+\|\nabla v\|_\Omega^2)^\frac{1}{2}\leq C_T^2C_P\sqrt{C_P^2+1}\|\nabla v\|_\Omega^2,\quad v\in W_2^1(\Omega),
  \label{eq:trace}
\end{equation}
where~\cref{eq:pc} was used for the second inequality.}
%
Then we have Green's first identity that can be derived from the divergence theorem
\begin{equation}
  \int_\Omega \nabla u\cdot \nabla v = \int_{\partial\Omega}u\frac{\partial v}{\partial n}-\int_\Omega u\Delta v,\quad \new{u\in W_2^1(\Omega), \quad v\in W_2^2(\Omega)},
\end{equation}  
leading to
\begin{eqnarray}
  \|\nabla v\|_\Omega^2&\leq& \|v\|_{\partial\Omega_0}\|\partial v/\partial n\|_{\partial\Omega_0} + \|v\|_{\partial\Omega_1}\|\partial v/\partial n\|_{\partial\Omega_1}+ \|v\|_\Omega\|\Delta v\|_\Omega\nonumber\\
  &=& \|v\|_{\partial\Omega_1}\|\partial v/\partial n\|_{\partial\Omega_1}+ \|v\|_\Omega\|\Delta v\|_\Omega,
  \label{eq:gs}
\end{eqnarray}  
where we separated the boundary integral into two parts due to the structure of our specific problem, and used that functions in $V^0$ vanish on $\partial\Omega_0$.

\new{To show coercivity, we start from~\eqref{eq:gs}, then use the trace inequality~\eqref{eq:trace} on the first term, and use the Poincar\'e inequality~\eqref{eq:pc} on the second term.
  \begin{eqnarray}
  \|\nabla v\|_\Omega^2&\leq& \|v\|_{\partial\Omega_1}\|\partial v/\partial n\|_{\partial\Omega_1}+ \|v\|_\Omega\|\Delta v\|_\Omega\nonumber \\  
  & \leq & C_T\sqrt{C_P}\sqrt[4]{C_P^2+1}\|\nabla v\|_\Omega\|\partial v/\partial n\|_{\partial\Omega_1} + C_P\|\nabla v\|_\Omega\|\Delta v\|_\Omega.
    \end{eqnarray}
Dividing through by the gradient norm, squaring the result, and using $(a+b)^2\leq 2a^2+2b^2$ leads to:
\begin{equation}
    \|\nabla v\|^2_\Omega\leq 2C_T^2C_P\sqrt{C_P^2+1}\|\partial v/\partial n\|_{\partial\Omega_1}^2 + 2C_P^2\|\Delta v\|_\Omega^2.
\end{equation}  
Let $C_1^2=2C_T^2C_P^2\sqrt{1+\frac{1}{C_P^2}}$ and $C_2^2 = 2C_P^2$.} Using~\eqref{eq:pc} one more time, we have:
\begin{equation}
  \|v\|^2_\Omega \leq C_P^2\|\nabla v\|^2_\Omega \leq C_P^2 \max(C_1^2,C_2^2)(\|\partial v/\partial n\|_{\partial\Omega_1}^2+ \|\Delta v\|_\Omega^2).
\end{equation}
If we finally let $C^2= C_P^2 \max(C_1^2,C_2^2)$, we have the coercivity result
\begin{equation}
  C^2a(v,v)\geq \|v\|^2_\Omega, \quad v\in V^0.
  \label{eq:stab0}
\end{equation}

\subsection{\new{Coercivity of the discrete bilinear form in a homogeneous space}}
%
\begin{theorem}
  \label{theor:stabh}
  If~\cref{theor:Dv} and~\cref{theor:S} hold, $h_y$ is chosen according to~\cref{theor:hy}, and the continuous bilinear form is coercive, then the discrete bilinear form is coercive for $v_h\in V_h^0$ such that
  \begin{equation}
    \|v_h\|_{\ell_2(\Omega)}\leq C^2\frac{(1+\tau)(1+\eta_0)(1+\eta_a)}{(1-\tau)}a_h(v_h,v_h)\equiv C_ha_h(v_h,v_h).
  \end{equation}  
\end{theorem}
\begin{proof}
  Since $v=S_2^0\in W_2^2(\Omega)$ the continuous coercivity property holds for $v$, showing that $a(v,v)>0$ and hence $a^*(v_h,v_h)>0$, and we can use~\cref{eq:tau0}, \cref{eq:eta0}, \cref{eq:stab0}, \cref{eq:eta}, and~\cref{eq:tau2} in sequence.
\end{proof}


\subsection{The global error estimate}


\new{\begin{theorem}
    Consider the least squares problem~\eqref{eq:nonhomres} with properties~\eqref{eq:ah1} and~\eqref{eq:ah2}. If~\cref{theor:Dv} and~\cref{theor:S} hold, $h_y$ is chosen according to~\cref{theor:hy}, and~\cref{theor:stabh} holds for the trial space error $e_h=u_h+u_h^0-I_h(u) \in V_h^0$, then the error $e=u_h+u_h^0-u$ satisfies  
\begin{equation}
  \|e\|_{\ell_2(\Omega)}\leq C\left(\frac{(1+\tau)(1+\eta_0)(1+\eta_a)}{(1-\tau)}\right)^\frac{1}{2}\sqrt{2a_h(e_I,e_I)} + \|e_I\|_{\ell_2(\Omega)},
  \label{eq:errfinal}
\end{equation}
where $e_I=I_h(u)-u$ is the interpolation error.
\label{theor:global}
\end{theorem}}  
\begin{proof}
  
  The error $e=u_h+u_h^0-u$ does not lie in the trial space unless $u$ lies in the trial space. Also, $u_h+u_h^0$ does not in general interpolate $u$.
  We use the interpolant $I_h(u)\in V_h$ as an auxiliary function to write $e=u_h+u_h^0-u=(u_h+u_h^0-I_h(u))+(I_h(u)-u)=e_h+e_I$. The first term $e_h$ has nodal values $e(X)$ and $e_h\in V_h^0$, since the Dirichlet condition is imposed strongly. The second term $e_I$ is the interpolation error, which is not in the trial space.
  We split the error to get
\begin{equation}
  \|e\|_{\ell_2(\Omega)}=\|e_h+e_I\|_{\ell_2(\Omega)}\leq \|e_h\|_{\ell_2(\Omega)} +\|e_I\|_{\ell_2(\Omega)}.
  \label{eq:err1}
\end{equation}
For the trial space error, we start from~\cref{theor:stabh}, and then use~\eqref{eq:ah2}, together with $|a_h(u,v)|\leq \frac{1}{2}a_h(u,u) + \frac{1}{2}a_h(v,v)$, and~\eqref{eq:ah1}
\begin{eqnarray}
  \frac{1}{C_h^2}\|e_h\|^2_{\ell_2(\Omega)}&\leq& a_h(e_h,e_h)= a_h(e-e_I,e_h) =a_h(-e_I,e_h)=a_h(-e_I,e-e_I)\nonumber\\
  &=& a_h(-e_I,e) + a_h(e_I,e_I)\leq 0.5a_h(e,e) + 1.5a_h(e_I,e_I)\nonumber\\
  &\leq&2a_h(e_I,e_I),
  \label{eq:errah}
\end{eqnarray}
leading to
\begin{equation}
  \|e_h\|_{\ell_2(\Omega)}\leq C_h\sqrt{2a_h(e_I,e_I)}.
  \label{eq:err2}
\end{equation}
Combining~\eqref{eq:err1} and~\eqref{eq:err2} provides the final result~\eqref{eq:errfinal}.%
\end{proof}

\subsection{The details of the global error estimate including scaling}\label{sec:errors:scaling}
The components of the global error estimate are now in place, and we can discuss their properties as well as the question about scaling of the different terms in the bilinear form.
Equation~\eqref{eq:eint} for the interpolation error $e_I=I_h(u)-u$ yields
\begin{equation}
  \|e_I\|_{\ell_2\new{(\Omega)}}\leq \alpha_0h^{p+1}|u|_{W_\infty^{p+1}(\Omega)},
\end{equation}  
and for the bilinear form applied to the interpolation error we have
\begin{equation}
  a_h(e_I,e_I)\leq \left((\alpha_0h^{p+1})^2+(\alpha_1h^{p})^2+(\alpha_2h^{p-1})^2\right)|u|_{W_\infty^{p+1}(\Omega)}^2.
\end{equation}
%
Noting that $a^2+b^2+\cdots\leq(a+b+\cdots)^2$ for positive numbers, we insert all terms in the global error estimate~\eqref{eq:errfinal}, to get
\begin{equation}
  \|e\|_{\ell_2\new{(\Omega)}}\leq \sqrt{2}C\left(\frac{(1+\tau)(1+\eta_0)(1+\eta_a)}{(1-\tau)}\right)^\frac{1}{2}\left(\tilde{\alpha}_0h^{p+1} + \alpha_1h^{p} +\alpha_2h^{p-1}\right)|u|_{W_\infty^{p+1}(\Omega)},
  \label{eq:finalerror}
\end{equation}
where $\tilde{\alpha}_0=\alpha_0(1+\frac{1}{\sqrt{2}C_h})$. 
%
%
The error estimate tells us that $h$ should be chosen to resolve the solution function, while $h_y\leq h$ should be chosen to resolve the integrals of the trial space error and its derivatives. 

The error expression~\eqref{eq:finalerror} indicates that the error may be somewhat improved, by adjusting the scale factors $\beta_i$ in~\eqref{eq:method:scaling}. The three terms in the sum change directly with the scale factors, while the stability constant depends on the worst case, leading to 
\new{\begin{equation}
  C_\beta = \sqrt{2}C_P^2\max\left(\frac{C_T\sqrt[4]{1+\frac{1}{C_P^2}}}{\beta_1},\frac{1}{\beta_2}\right)=C_P\max\left(\frac{C_1}{\beta_1},\frac{C_2}{\beta_2}\right).
\end{equation}
Choosing the scaling such that $C_1/\beta_1=C_2/\beta_2$, can in principle reduce the overall scaled error estimate. However, since we do not know the constants a priori, we use $\beta_1=\beta_2=1$ in the numerical experiments.}

\new{The scaling $\beta_0$ of the Dirichlet condition does not affect the stability constant, but it can be beneficial to increase $\beta_0$ to reduce the errors near the boundary.} The largest scaling such that the order of the Dirichlet term does not dominate the order of the Neumann term is $\beta_0=\mathcal{O}(h^{-1})$.
This scaling strategy is evaluated numerically in \cref{sec:numericalstudy} and is shown to perform well.   

\section{Numerical study}\label{sec:numericalstudy}
\new{In this section, we investigate the convergence, stability, and efficiency of least-squares based RBF-FD, compared with collocation-based RBF-FD.}
\new{The two methods are tested in two flavors: with additional ghost points (RBF-FD-LS-Ghost, RBF-FD-C-Ghost) and without ghost points (RBF-FD-LS, RBF-FD-C).} 
We solve the PDE problem~\eqref{eq:methods:Poisson} using the scaling~\eqref{eq:method:scaling} with 
\new{$\beta_0=h^{-1}$, $\beta_1=1$, and $\beta_2=1$} on a domain with an outer boundary defined in polar coordinates as
$r(\theta) = 1 + \frac{1}{10} \left( \sin(7\theta) + \sin(\theta) \right)$, $\theta \in [-\pi, \pi)$,
and with the Dirichlet boundary $\partial \Omega_0$ defined by $r(\theta)$, $\theta \in [-\pi, 0)$, and the Neumann boundary $\partial \Omega_1$ defined by $r(\theta)$, $\theta \in [0, \pi)$. 
The node sets $X$ and $Y$ are generated using \emph{DistMesh}~\cite{distmesh}. However, in order to enforce $X\subset Y$, we modify the $Y$ node set such that for each $x_k\in X$ we find 
the closest point $\tilde{y}_j$ in the initial node set $\tilde{Y}$, and then let $y_j=x_k$ in the final node set $Y$. 
The Dirichlet boundary conditions are enforced exactly at $X_{\partial\Omega_0}$ according to \eqref{eq:nonhomres}.

The domain $\Omega$, an example of the spatial error distribution, the node sets $X$ and $Y$, and the Voronoi diagram corresponding to the node set $X$ are shown in \cref{fig:experiments:Poisson:exactsolution}. 
All experiments were run in MATLAB on a laptop with an Intel i7-7500U processor and 16 GB of RAM. The code that was used to generate the rectangular and square RBF-FD matrices is available in \cite{Tominec_rbffdcode2021}.
\begin{figure}[!htb]
\centering
\begin{tabular}{cc}
\includegraphics[width=0.45\linewidth]{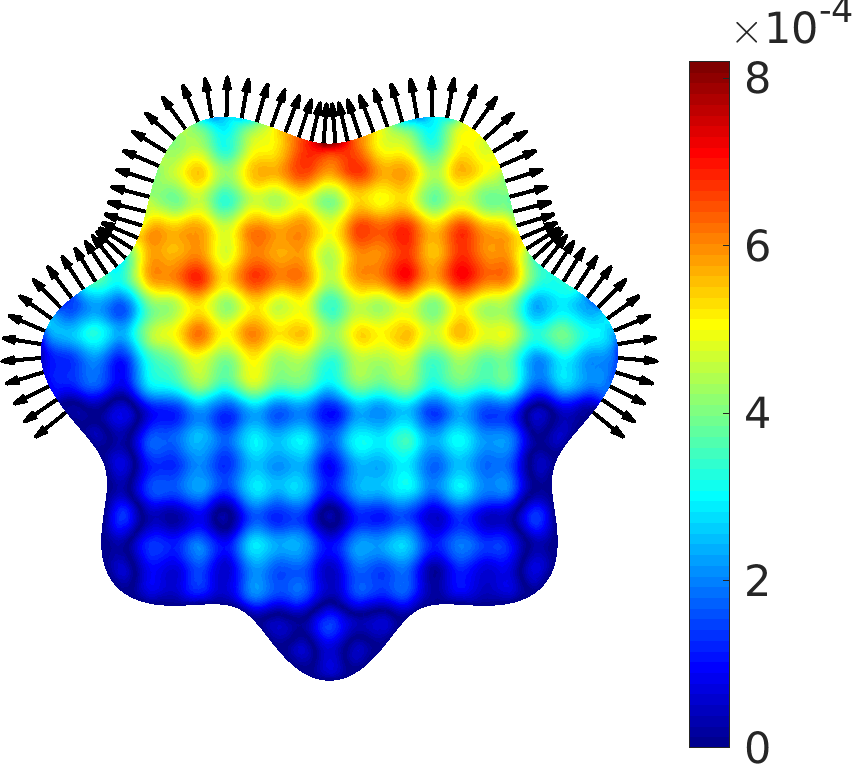} & \includegraphics[width=0.4\linewidth]{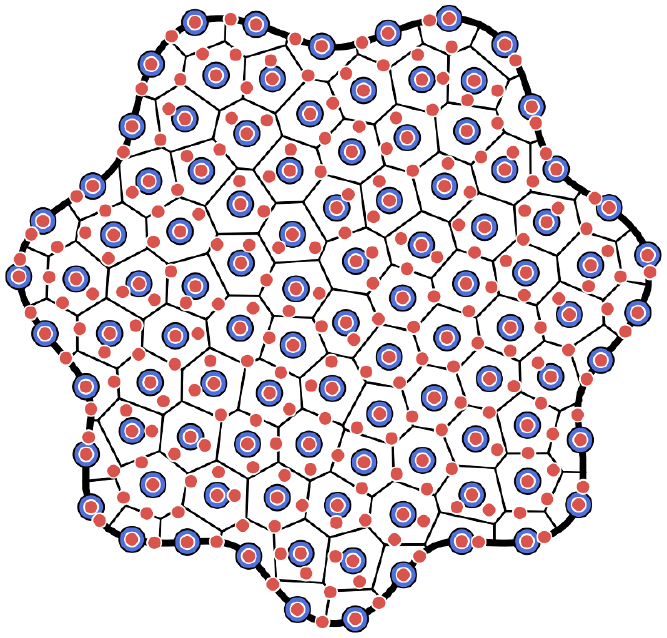} \\
a) & b) \\
\end{tabular}
\caption{a) A contour plot of the absolute error distribution over the domain $\Omega$ for the truncated Non-analytic solution function \eqref{eq:experiments:nonanalytic} when $h=0.02$, $p=3$, $q=3$. The outward normals indicate the locations where the Neumann condition is enforced. b) The $X$ node set (large blue markers) and the $Y$ node set (small red markers) are shown together with the Voronoi diagram for the $X$ node set. Each Voronoi cell contains on average three $Y$ node points.}
\label{fig:experiments:Poisson:exactsolution}
\end{figure}

In the numerical study, we focus on the three main method parameters: The node distance $h$, the oversampling parameter $q=\frac{M}{N} = (h/h_y)^2$, which determines $h_y$, and the polynomial degree $p$. 
When nothing else is stated, we use the default value $q=3$ for the oversampling. 
\new{When applicable, the ghost points are added to the node set $X$  as an additional layer outside $\Omega$. 
The layer is generated by projecting every boundary point a distance $h$ in the normal direction. The matrix $\bar{D}_h$ and the right-hand-side vector $\bar{F}$ are then modified such that 
the Laplacian is sampled also on $\partial \Omega$. In the RBF-FD-LS case, the size of $\bar{D}_h$ then grows from $M\times N$ to $(M+N_g) \times (N+N_g)$ and in the RBF-FD-C case 
from $N \times N$ to $(N+N_g) \times (N+N_g)$, 
where $N_g$ is the number of ghost points equal to the number of $X$ points at the boundary.}
The stencil size in all experiments is $n=2m$, where $m$ is the dimension of the polynomial space. In the convergence experiments, we measure the relative $\ell_2$-error
\begin{equation}
  \|e(Y)\|_{\new{\ell_2(\Omega)}} = \frac{\|u_h(Y)-u(Y)\|_{\ell_2\new{(\Omega)}}}{\|u(Y)\|_{\ell_2\new{(\Omega)}}}.
\end{equation}
We also investigate the stability norm by defining it as the ratio of the largest singular value of $\bar{E}_h(Y,\tilde{X})$ and the smallest singular value of $\bar{D}_h(Y,\tilde{X})$. The stability norm provides a numerical value for the coercivity constant $C_h$ in the global error estimate~\eqref{eq:errfinal}. Using~\eqref{eq:nonhomres} for the interior solution, we have
    \begin{eqnarray}
        \|u_h(Y)\|_{\ell_2\new{(\Omega)}} &=& \|\bar{E}_h(Y,\tilde{X})\, u(\tilde{X})\|_2 = \|\bar{E}_h\, \bar{D}_h^{+}\, \tilde{F}(Y)\|_2  \nonumber\\
        &\leq& \|\bar{E}_h\|_2\, \|\bar{D}_h^{+}\|_2\, \|\tilde{F}(Y)\|_2 =
        \frac{\sigma_{\max} (\bar{E}_h)}{\sigma_{\min} (\bar{D}_h)}\|\tilde{F}(Y)\|_2.
        \label{eq:stabilitynorm}
    \end{eqnarray}

\subsection{Errors and convergence tests for different functions}
\label{sec:experiments:h-refinement:differentSolutions}
Three solution functions which are different in nature are used to compute the right-hand-side data of 
\eqref{eq:methods:Poisson}. 
The purpose of this test is to show the differences in the error behavior and to 
pick one solution function for which the method is later on tested more extensively. 
Additionally, we solve the PDE problem 
with only Dirichlet boundary data. 
The following functions are considered:
\begin{eqnarray}
    u_1(x,y) &=& \sqrt{x^2 + y^2}, \\
    \label{eq:experiments:nonanalytic}
    u_2(x,y) &=& \sum_{k=0}^5 e^{-\sqrt{2^k}}\left( \cos(2^k x) + \cos(2^k y)\right),\\
    \label{eq:experiments:rationalsine}        
    u_3(x,y) &=& \sin\big(2(x-0.1)^2\big) \cos\big((x-0.3)^2\big) + \frac{\sin\big(2(y-0.5)^2\big)^2}{1 + 2x^2 + y^2}, 
\end{eqnarray}
which are referred to by the following names: Distance, truncated Non-analytic and Rational sine, respectively.
The polynomial degree $p=5$ is used for computing the local interpolation matrices. 
The error under refinement of $h$ is displayed in \cref{fig:experiments:Poisson:h-refinement_differentFunctions:error}.
\begin{figure}[!htb]
    \small
    \centering
    \begin{tabular}{ccc}
           \phantom{xxxx}Distance & \phantom{xxxxj}Non-analytic  & \phantom{xxxxj}Rational sine\\
        \includegraphics[width=0.268\linewidth]{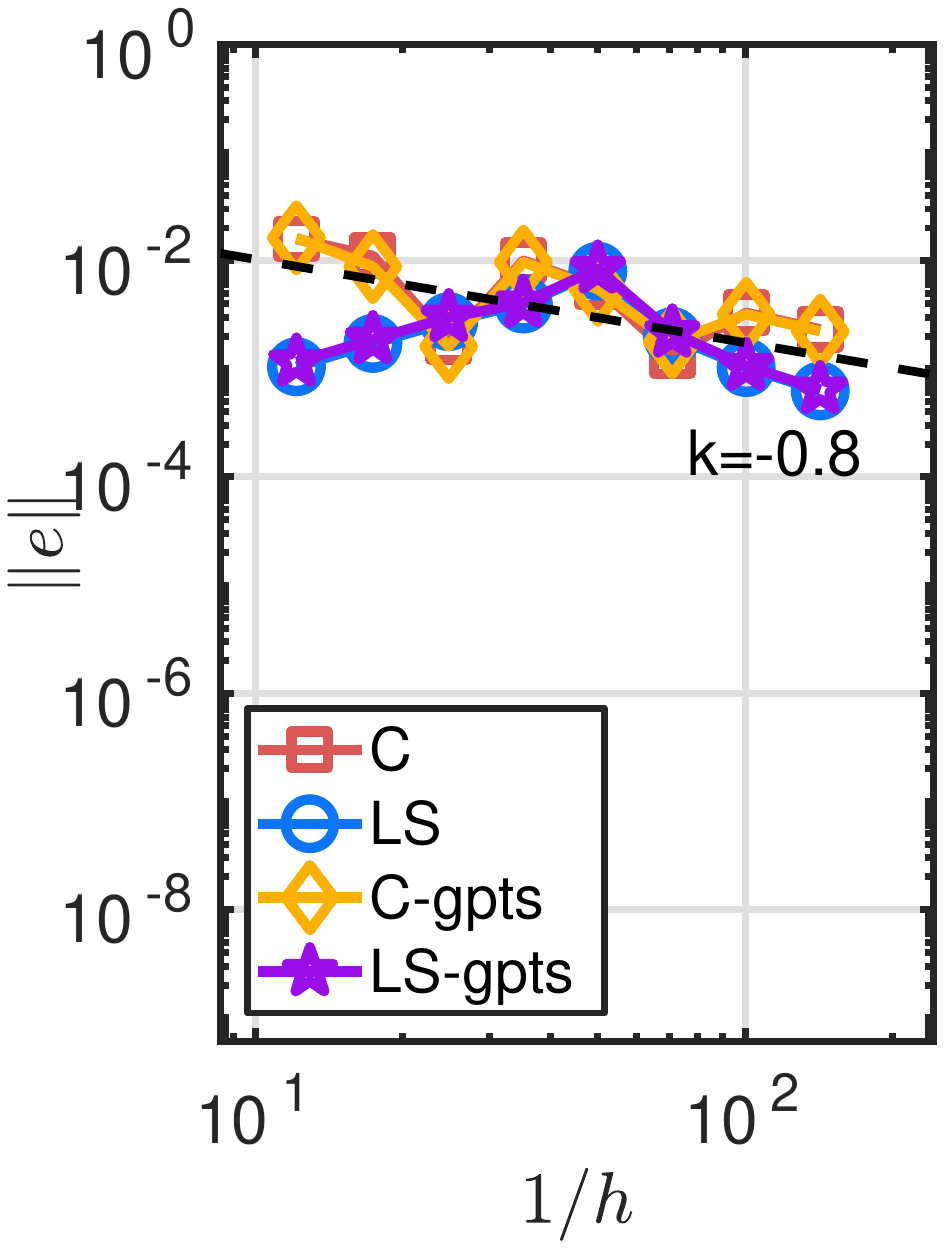} &        
        \includegraphics[width=0.28\linewidth]{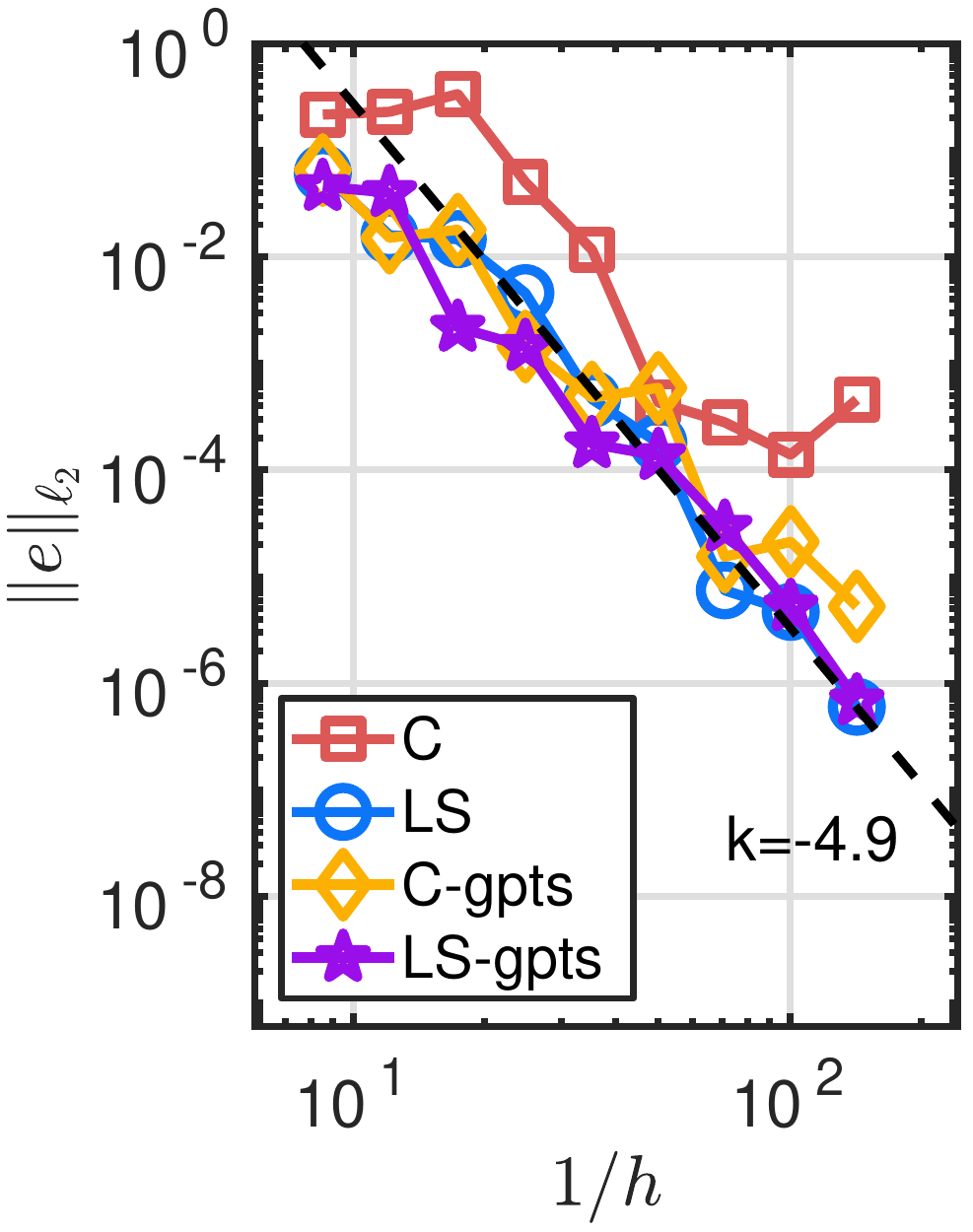} &
        \includegraphics[width=0.28\linewidth]{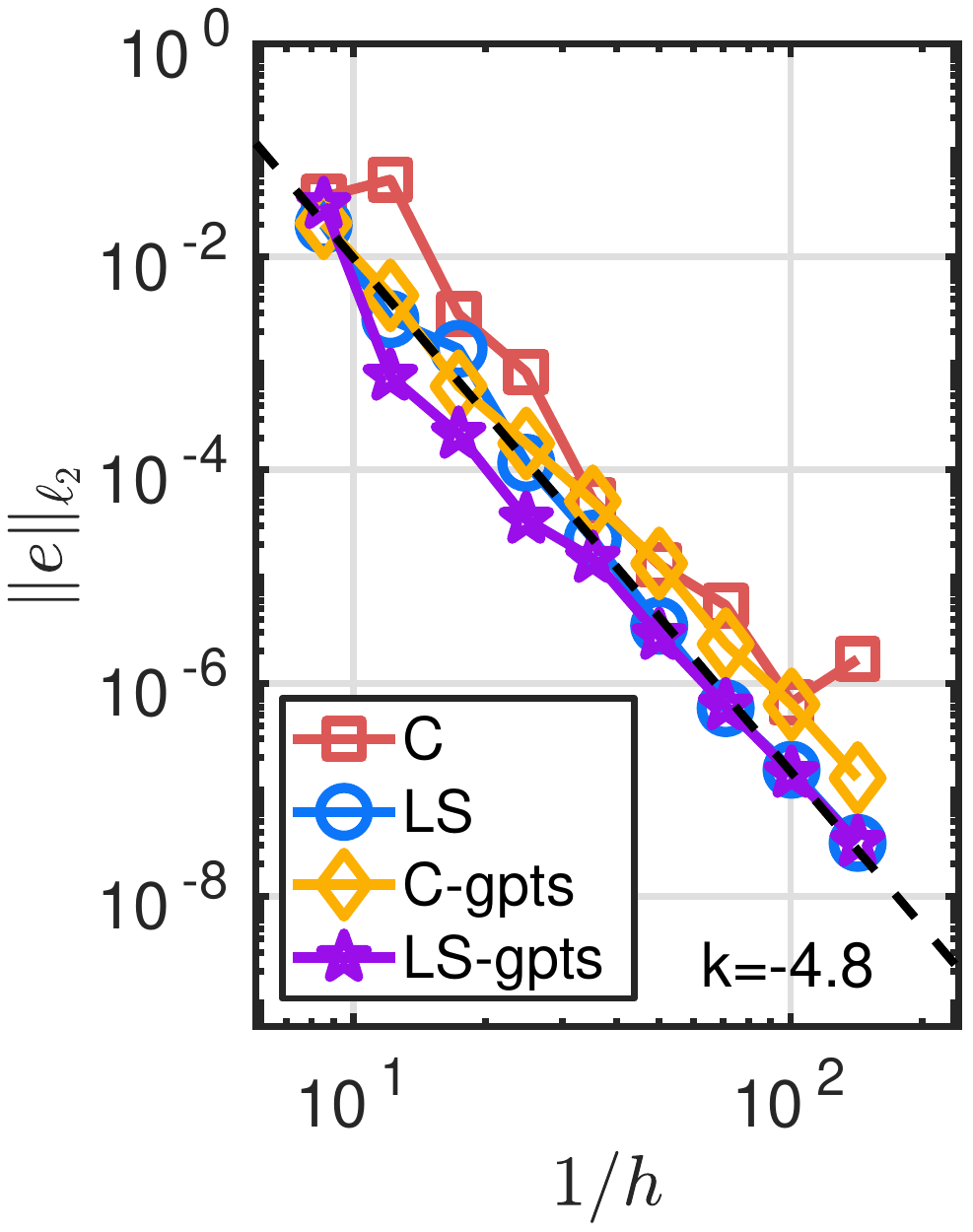}  \\
        \phantom{xxxx}Distance & \phantom{xxxxj}Non-analytic  & \phantom{xxxxj}Rational sine\\
        \includegraphics[width=0.28\linewidth]{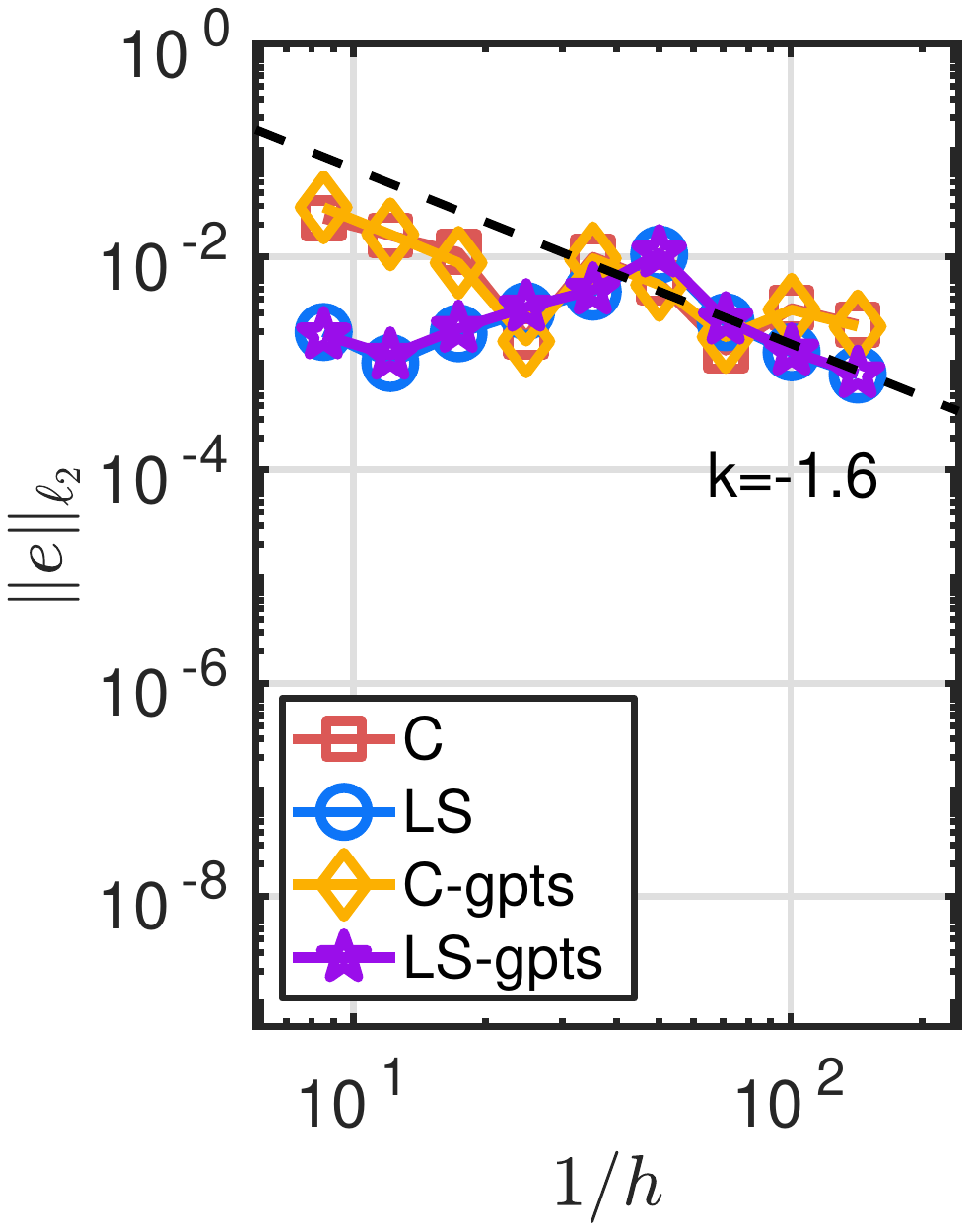} &        
        \includegraphics[width=0.28\linewidth]{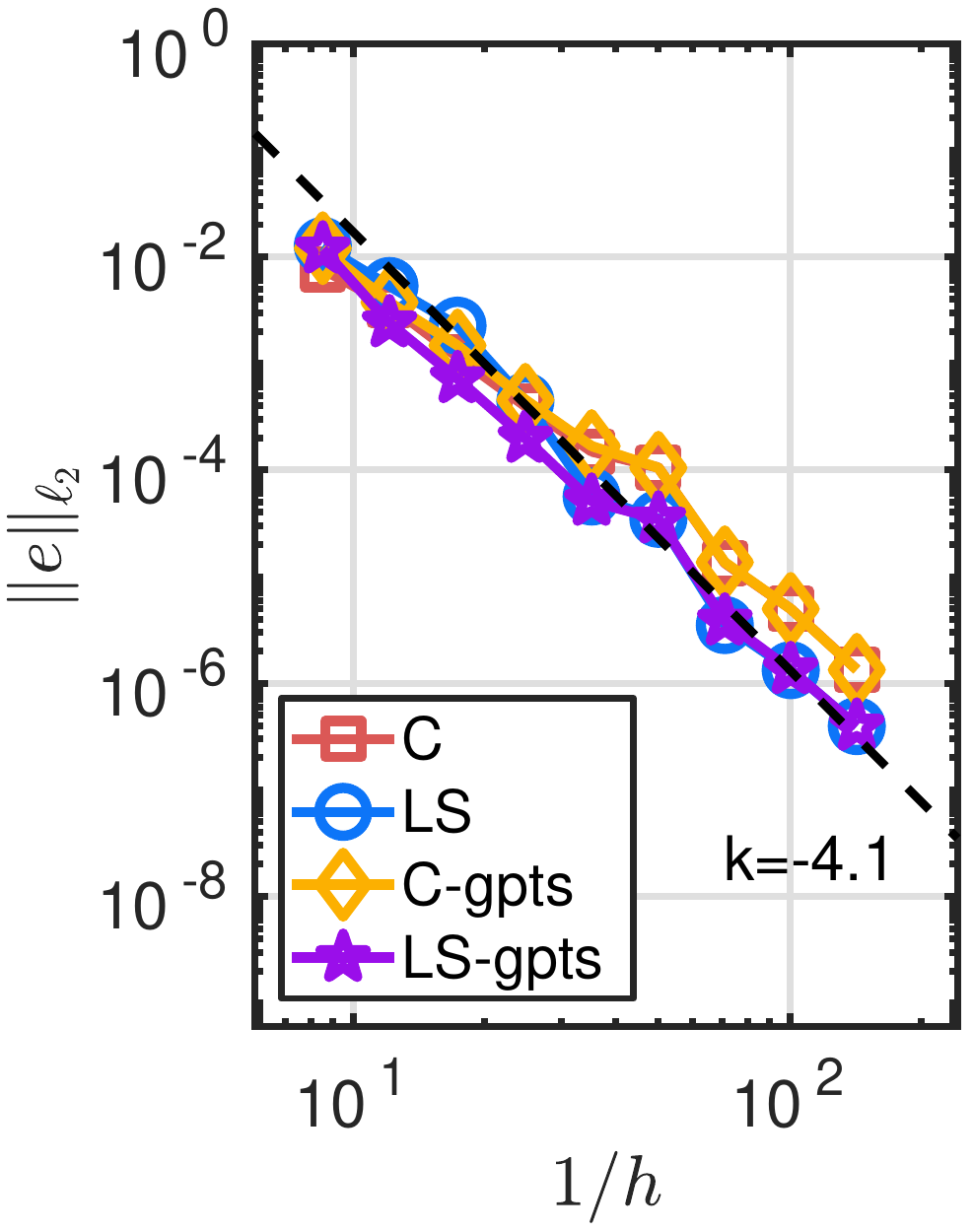} &
        \includegraphics[width=0.28\linewidth]{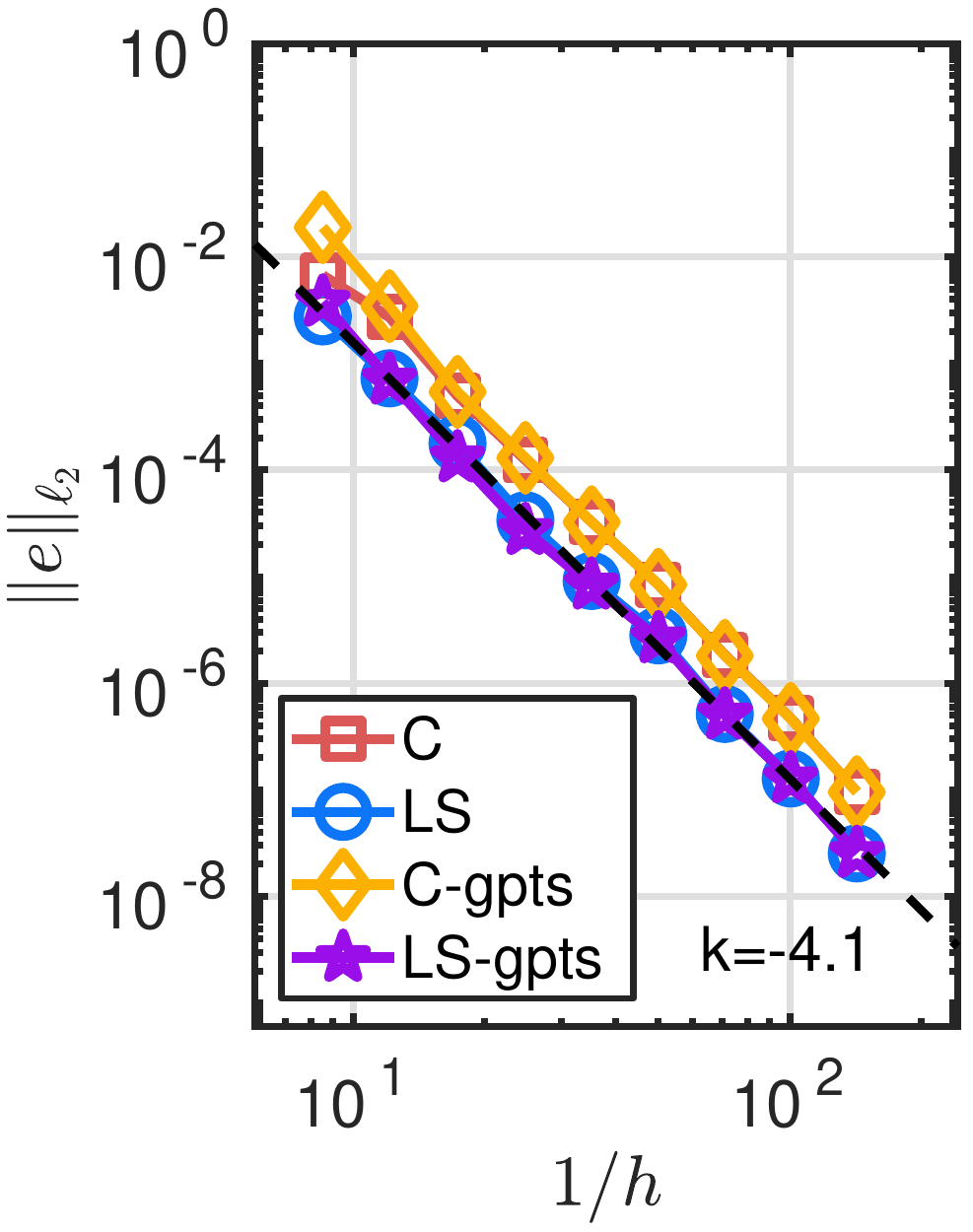}

    \end{tabular}        
    \caption{The errors of RBF-FD-LS and RBF-FD-C as a function 
    of the inverse node distance for the three different solution 
    functions defined in \cref{sec:experiments:h-refinement:differentSolutions}. 
    The first row of plots corresponds to solving \eqref{eq:methods:Poisson} and the second row 
    corresponds to solving the same problem with the Dirichlet condition on the whole boundary.
    The polynomial degree used to construct the interpolation matrices is $p=5$ and the oversampling parameter 
    $q=3$. The number of node points in $X$ ranges from $N=500$ to $N=64000$.}
    \label{fig:experiments:Poisson:h-refinement_differentFunctions:error}
\end{figure}

The accuracy of RBF-FD-LS is better than that of RBF-FD-C for all solution functions when both the Neumann and Dirichlet conditions are present. 
The convergence rates for the truncated Non-analytic and Rational sine functions \new{are $k=4.9$ and $k=4.8$} respectively, which 
agrees with the error estimate \eqref{eq:finalerror} since $k \geq p-1 = 4$. 
\new{When ghost points are utilized, the accuracy is better compared to when no ghost points are utilized. Comparing RBF-FD-LS-Ghost with RBF-FD-C-Ghost we see that the accuracy of the former
is generally better compared with the accuracy of the latter. Next, comparing RBF-FD-C-Ghost to RBF-FD-LS we see that RBF-FD-LS is, in this test, overall more accurate, but the gap is smaller compared with the gap between 
RBF-FD-LS and RBF-FD-C.}
The convergence rate for the Distance function \new{is $k=0.8 < 4$\}}, but that is expected since 
it is a $C^0$ function.

When only the Dirichlet condition is imposed, the accuracy of RBF-FD-LS is better for the Distance and 
Rational sine functions. This is also the case for the truncated Non-analytic function, when $h$ is small enough. 
The difference in error between RBF-FD-LS and RBF-FD-C is not as large as when both the Neumann and Dirichlet conditions are imposed. 
In this case the convergence rates for the truncated Non-analytic function and Rational sine function are $k \geq p-1 = 4$, similarly to the mixed conditions case.
\new{Ghost points in this case do not play a significant role in improving accuracy.}

In the following subsections further experiments are made with the truncated Non-analytic solution function, which, due to its 
fine scale variation, is challenging to approximate.

\subsection{Approximation properties under node refinement}
We refine $h$ (this increases the number of nodes $N$), and measure the approximation properties for different polynomial degrees in the local interpolation matrices \eqref{eq:M}. 
We denote this by $h$-refinement. The convergence as a function of the node 
distance is shown in \cref{fig:experiments:Poisson:h-refinement:error}.
%
\begin{figure}[!htb]
    \centering
    \small
    \begin{tabular}{ccc}
        \hspace{0.075\linewidth}$p=3$ & \hspace{0.075\linewidth}$p=4$  & \hspace{0.075\linewidth}$p=5$\\
        \includegraphics[width=0.28\linewidth]{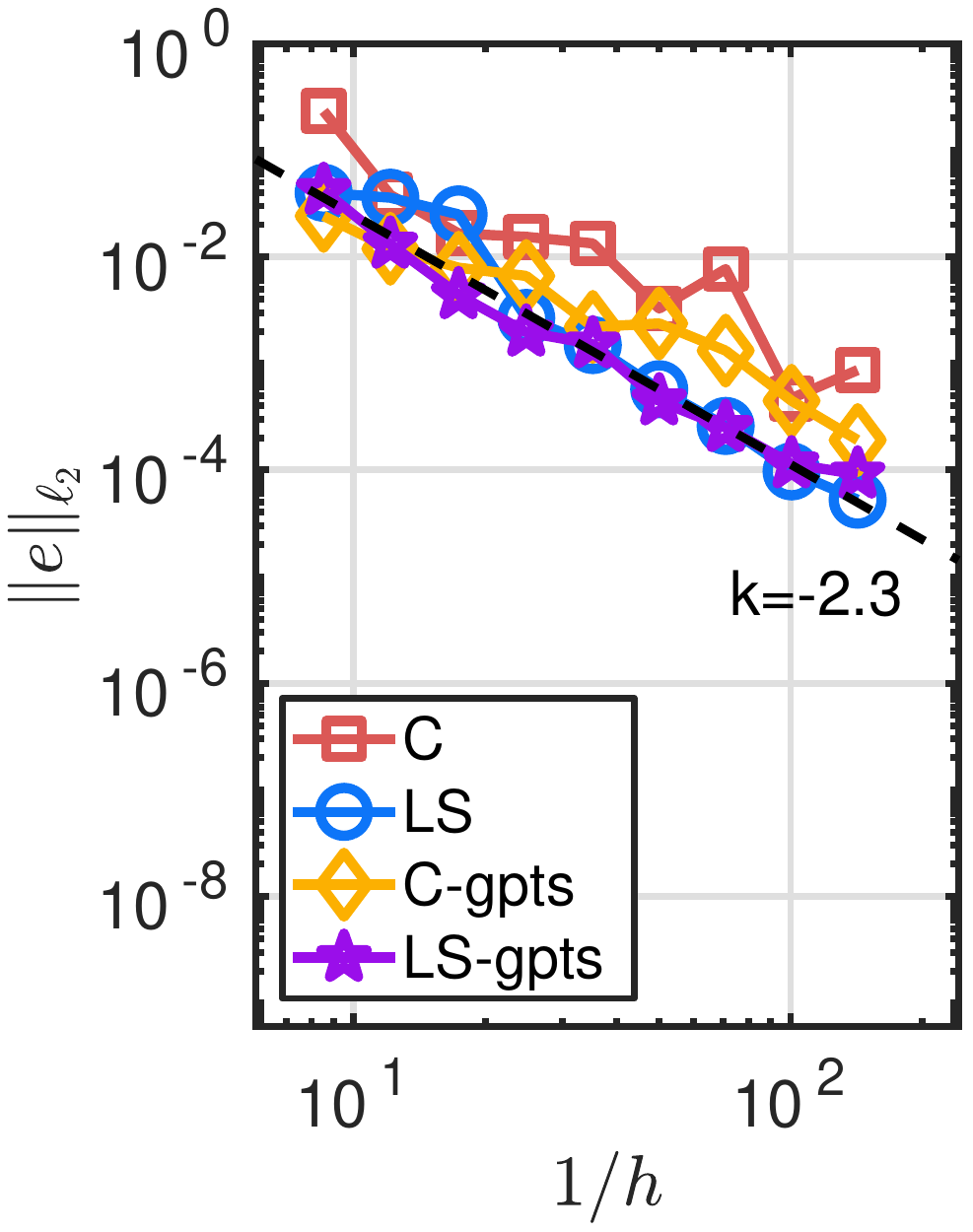} &
        \includegraphics[width=0.28\linewidth]{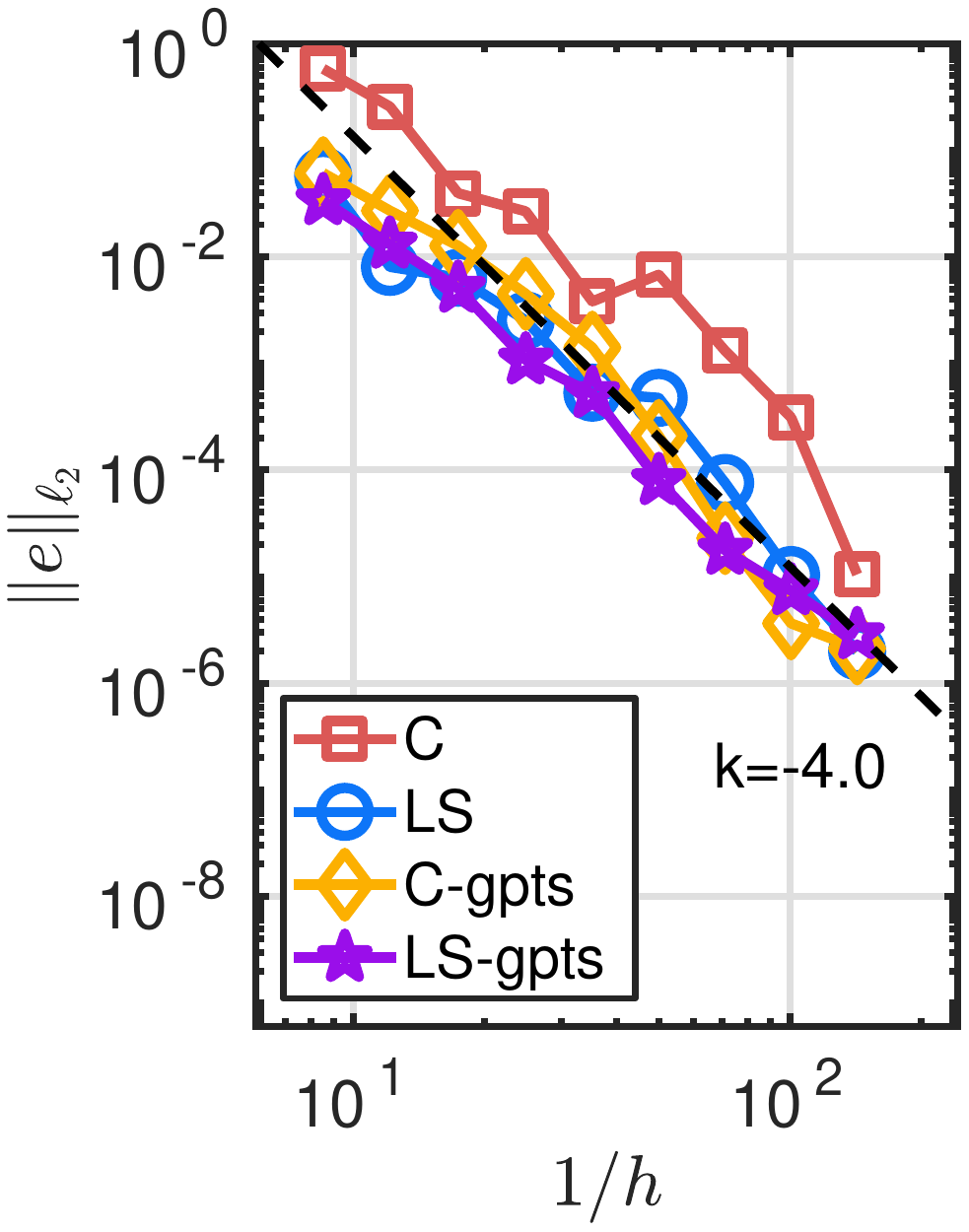} &
        \includegraphics[width=0.28\linewidth]{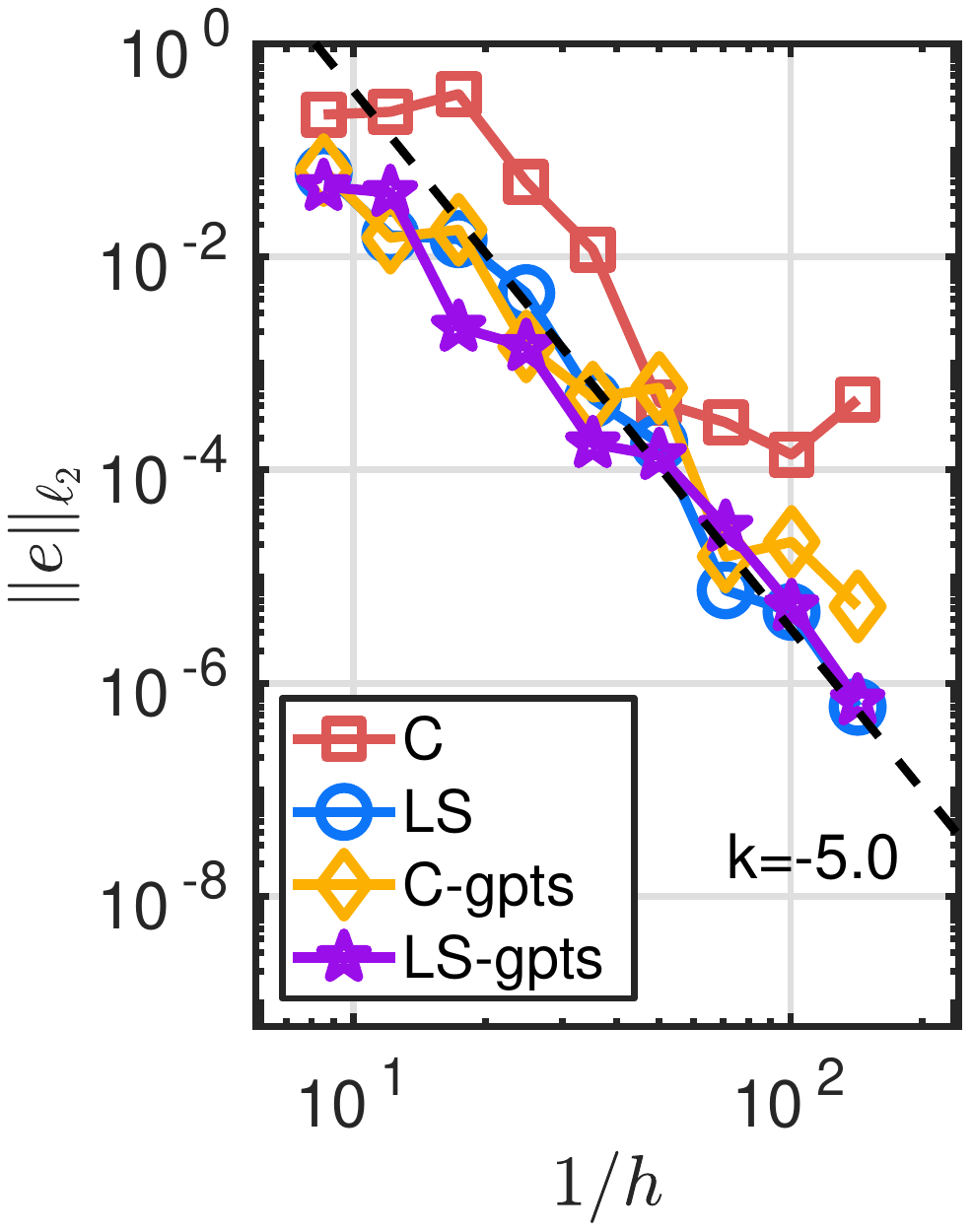}
    \end{tabular}    
    
    \caption{The RBF-FD-LS and RBF-FD-C methods 
    are compared. The relative error as a function 
    of the inverse node distance $1/h$ for a fixed oversampling parameter $q=3$ 
    and different polynomial degrees $p$ used to form the local interpolation matrices is shown.}
    \label{fig:experiments:Poisson:h-refinement:error}
\end{figure}
\noindent
We observe that the accuracy of RBF-FD-LS is better for each tested $p$ compared with RBF-FD-C. 
The overall difference in the errors is larger for $p=5$ compared with when $p=3$ and $4$.
The convergence trend $k$ of RBF-FD-LS is $k\geq p-1$ for every $p$. 
It is hard to evaluate the convergence trend of RBF-FD-C since the error behavior is unpredictable.
\new{The accuracy of RBF-FD-LS-Ghost is overall better for all $p$ compared with RBF-FD-C-Ghost. The gap between RBF-FD-LS and RBF-C-Ghost is 
smaller compared with the gap between RBF-FD-LS and RBF-FD-C, and in some points, RBF-FD-C-Ghost is more accurate than RBF-FD-LS.}

Next, the relation 
between the error and the computational time (runtime) is investigated. 
It is important to note that a method with a smaller error/runtime ratio is more efficient.
\new{The runtime is divided into two steps shown in~\cref{fig:experiments:Poisson:h-refinement:errorVsRuntime_matrices} and~\cref{fig:experiments:Poisson:h-refinement:errorVsRuntime_solving}:
\begin{itemize}
\item[$R_1:$] The closest neighbor search, forming and inverting the local interpolation matrices \eqref{eq:M}, and
forming the evaluation and differentiation weights \eqref{eq:locdiff}.
\item[$R_2:$] Assembly of the PDE operator \eqref{eq:discretePDE}, and solution of the system of equations using \emph{mldivide()} in MATLAB.
\end{itemize}
}
\noindent
The node generation is considered as a preprocessing step and is therefore not included in the measurement. 
\begin{figure}[!htb]
  \centering
  \small
    \begin{tabular}{ccc}
        \phantom{xxxx}$p=3$ & \phantom{xxxx}$p=4$  & \phantom{xxxx}$p=5$\\
        \includegraphics[width=0.28\linewidth]{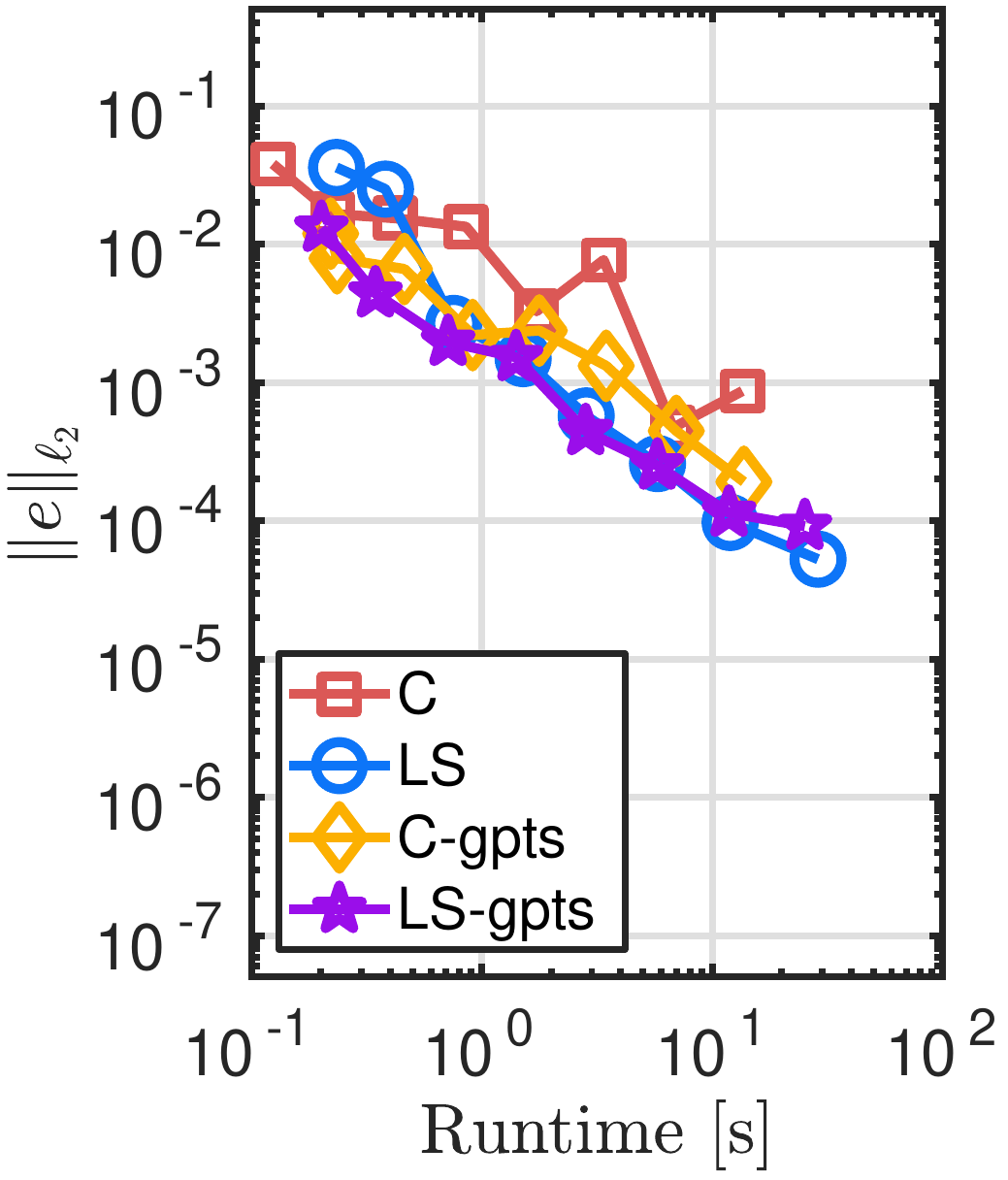} &
        \includegraphics[width=0.28\linewidth]{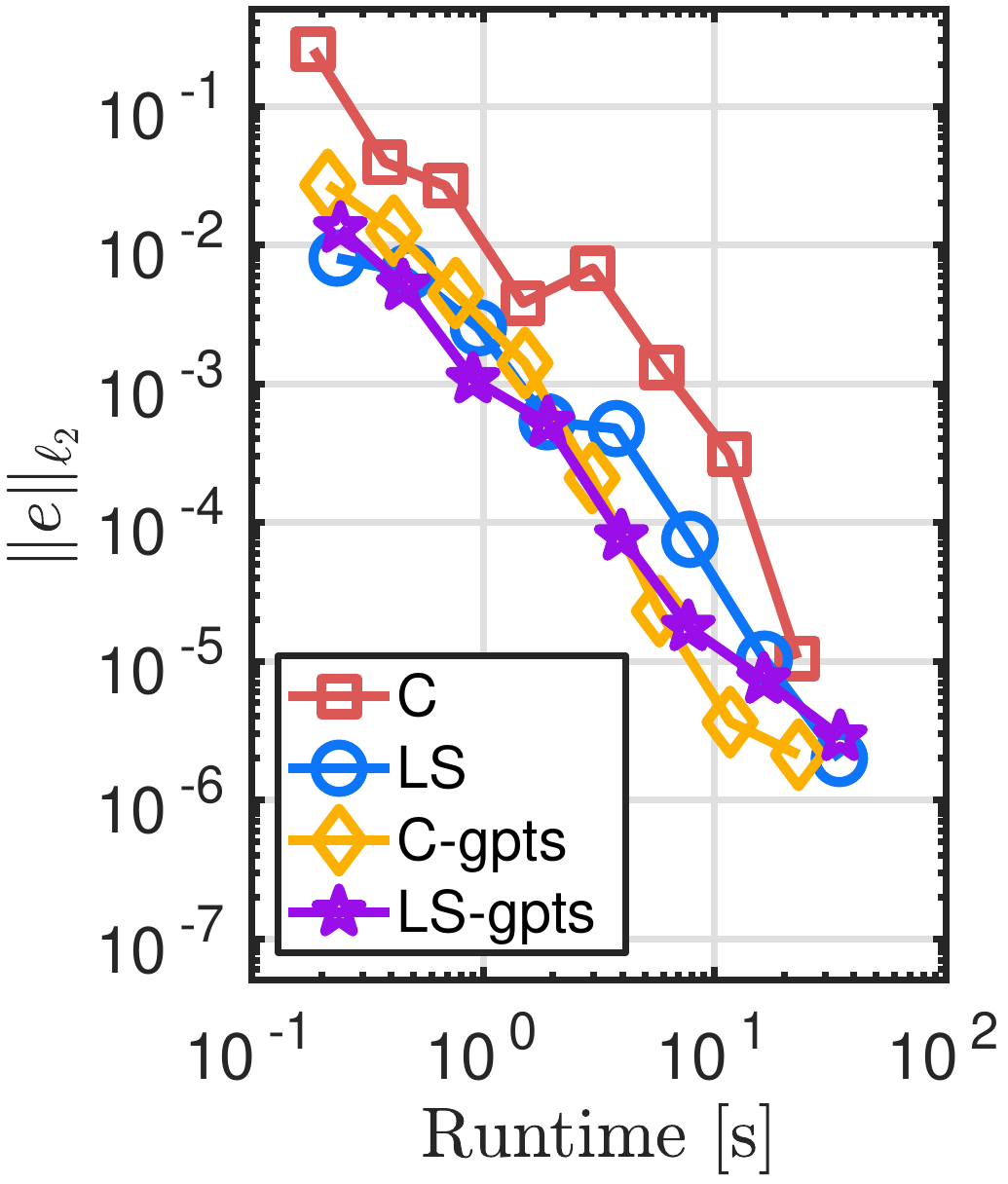} &
        \includegraphics[width=0.28\linewidth]{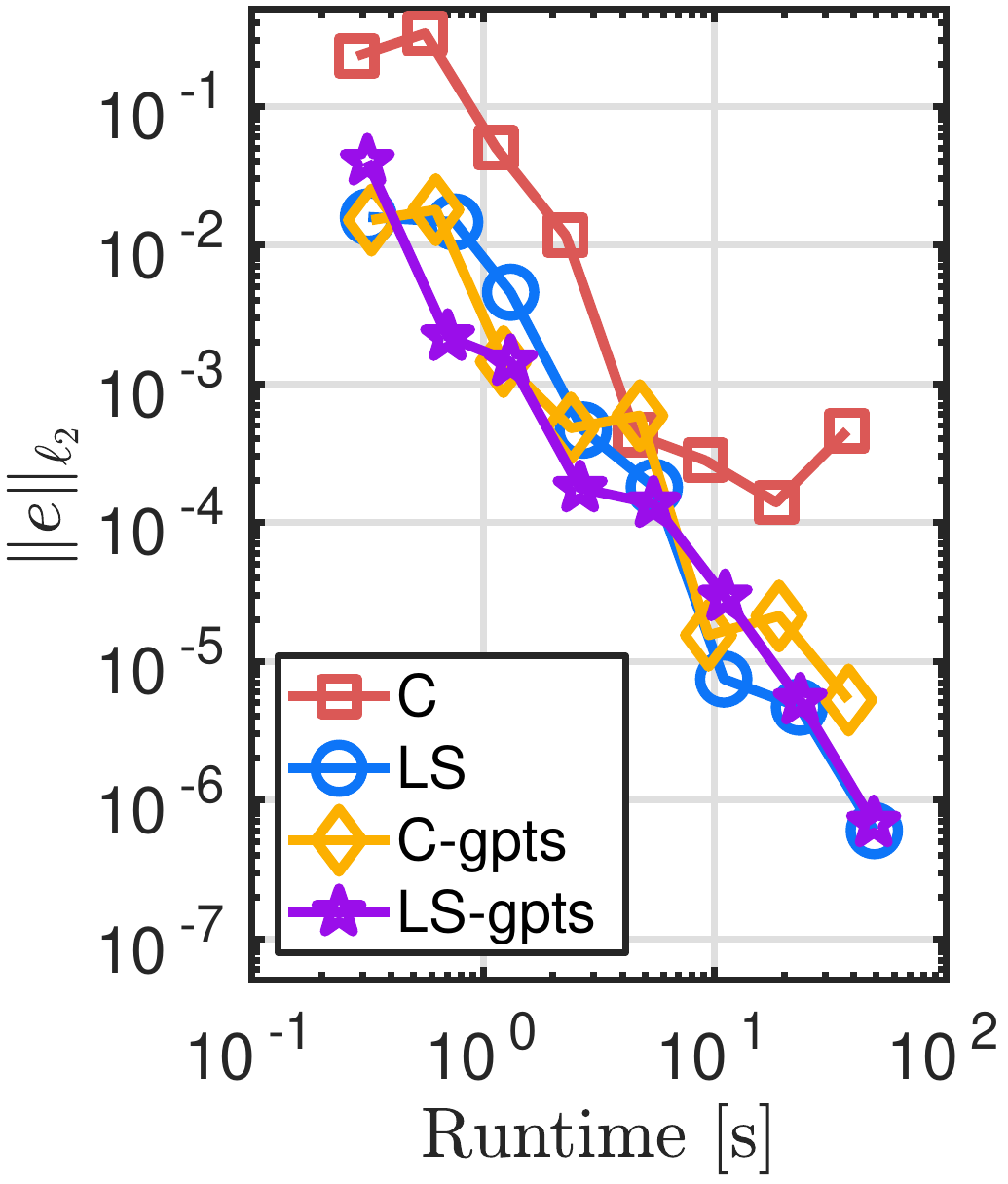}        
    \end{tabular}
    
    \caption{The RBF-FD-LS and RBF-FD-C methods are compared. The relative error as a function of the initialization runtime ($R_1$) measured in seconds for a fixed oversampling parameter $q=3$ 
    and different polynomial degrees $p$ used to form the local interpolation matrices. }
    \label{fig:experiments:Poisson:h-refinement:errorVsRuntime_matrices}
\end{figure}
\begin{figure}[!htb]
  \centering
  \small
  \begin{tabular}{ccc}
    \phantom{xxxx}$p=3$ & \phantom{xxxx}$p=4$  & \phantom{xxxx}$p=5$\\
        \includegraphics[width=0.28\linewidth]{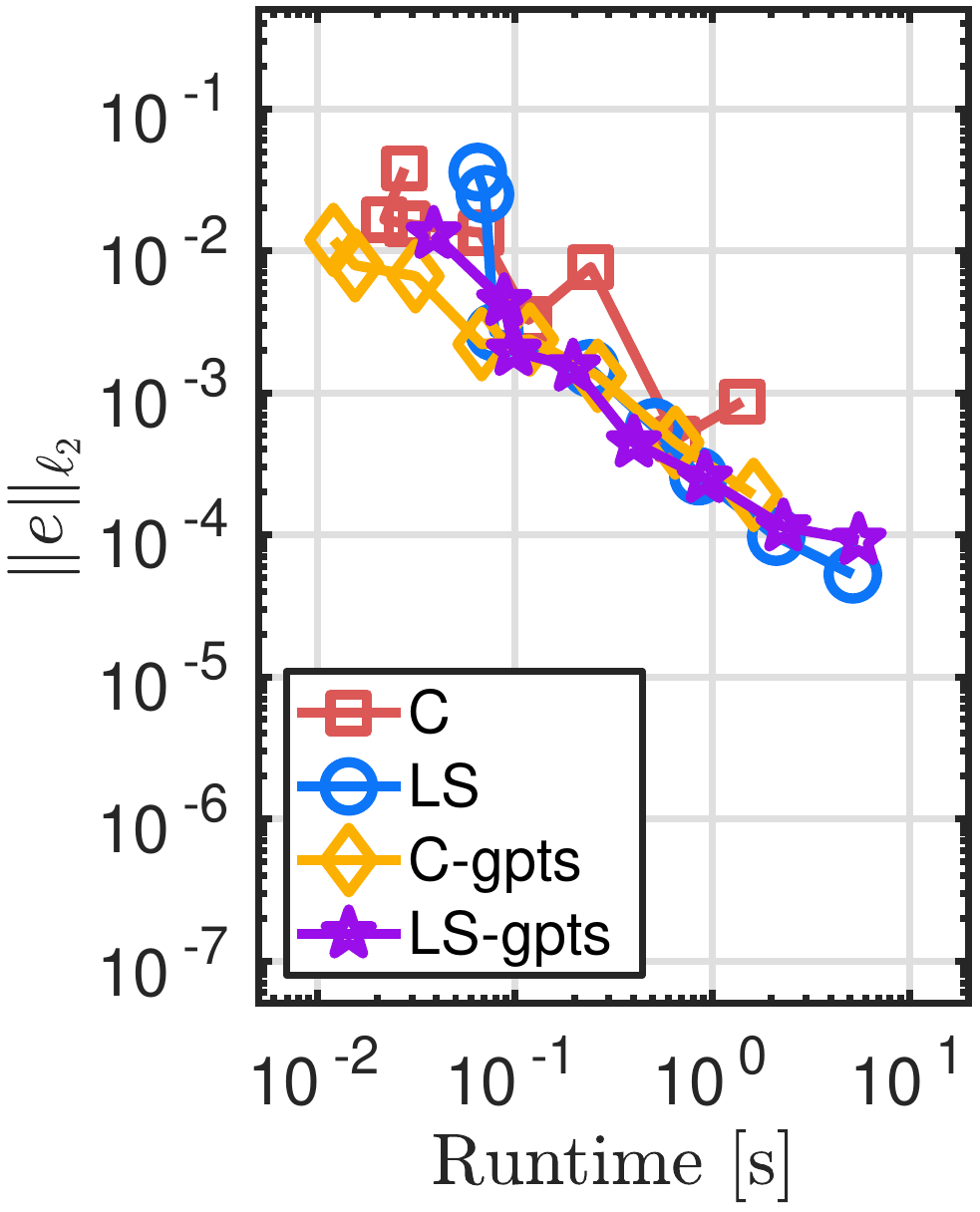} &
        \includegraphics[width=0.28\linewidth]{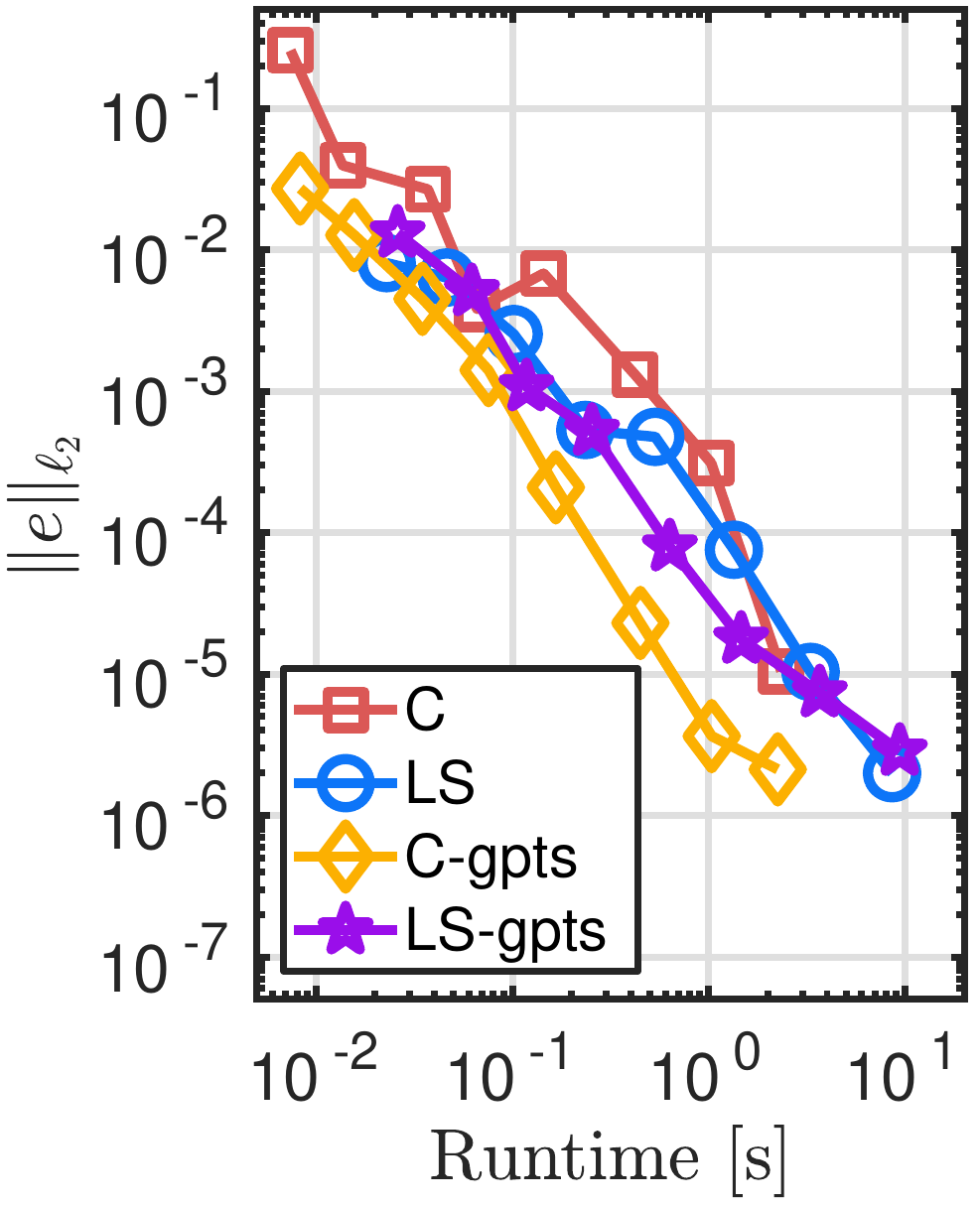} &
        \includegraphics[width=0.28\linewidth]{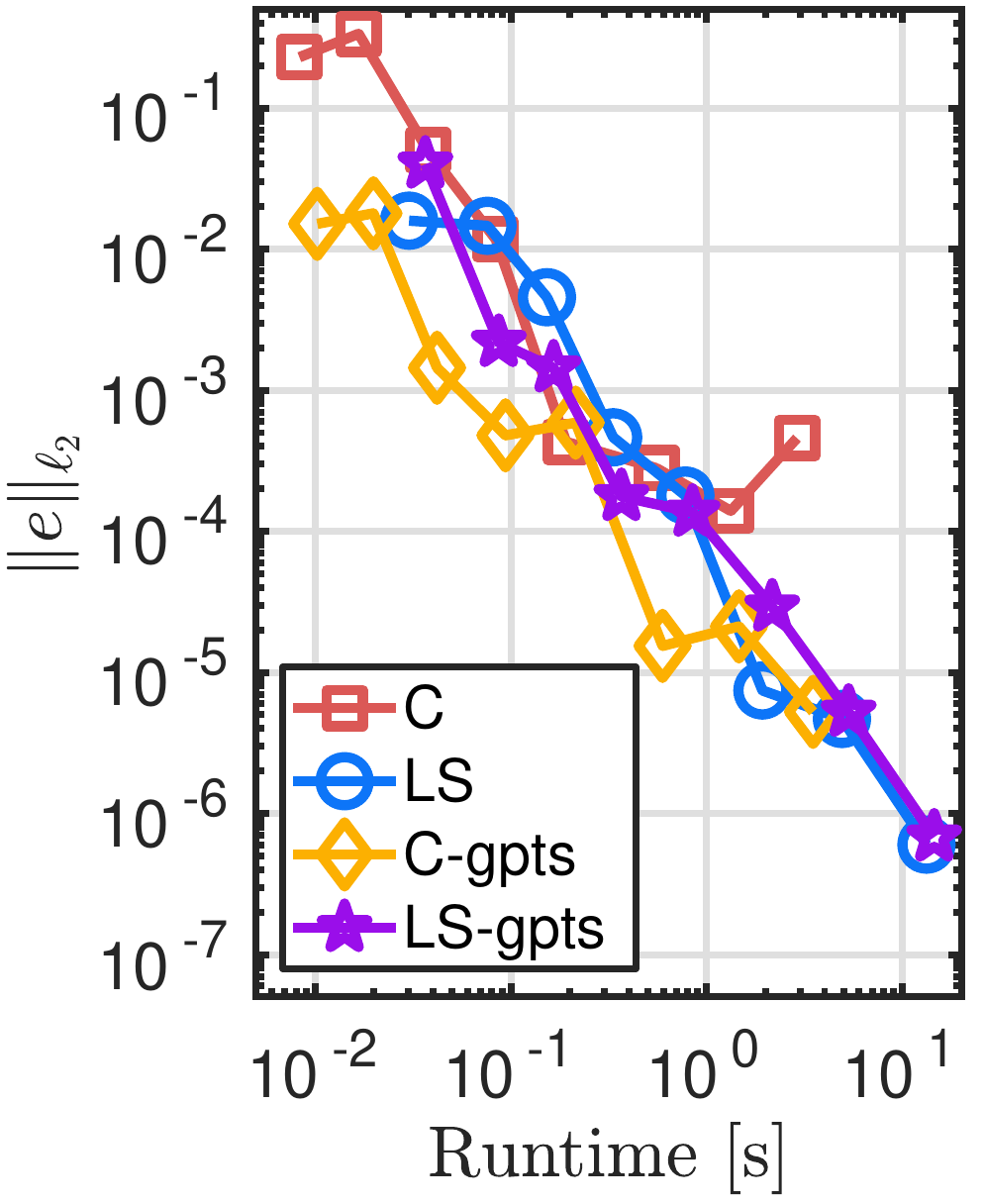}
    \end{tabular}
    
    \caption{The RBF-FD-LS and RBF-FD-C methods are compared. The relative error as a function of the solving runtime ($R_2$) measured in seconds for a fixed oversampling parameter $q=3$ 
    and different polynomial degrees $p$ used to form the local interpolation matrices. }
    \label{fig:experiments:Poisson:h-refinement:errorVsRuntime_solving}
\end{figure}

We observe that the efficiency of RBF-FD-LS is better than that of RBF-FD-C for all considered $p$ and both efficiency measurements: $R_1$ and $R_2$. 
When $p=5$, the difference in the efficiency is larger. 
The oversampling parameter $q$ does not have a decisive role when it comes to the efficiency. We expect the run-time to be dominated by the solution of the overdetermined linear system. For a dense matrix, the cost grows linearly with $M=qN$ for a fixed $N$. We expect a similar behavior for our sparse system.
The added cost is compensated for by the improved accuracy. \cref{fig:experiments:Poisson:q-refinement:error} shows the error improvement with $q$.

\new{RBF-FD-C-Ghost behaves similarly to RBF-FD-LS and RBF-FD-LS-Ghost concerning the efficiency $R_1$. In the $R_2$ case, RBF-FD-C-Ghost outperforms all other methods; however, 
the magnitude of the runtime in $R_2$ is at least one order smaller compared with $R_1$. By observing the efficiency as a function of to the total runtime $R_1+R_2$, the result of $R_1$ is dominating. 
Thus, the three 
methods overall behave similarly in terms of efficiency.}

The stability norm~\eqref{eq:stabilitynorm} as a function of $1/h$ is studied in \cref{fig:experiments:Poisson:h-refinement:stabilityNorm}.
\begin{figure}[!htb]
    \centering
    \small
    \begin{tabular}{@{\hspace{2mm}}c @{\hspace{2mm}}c @{\hspace{2mm}}c}
        \hspace{0.075\linewidth}$p=3$ & \hspace{0.075\linewidth}$p=4$  & \hspace{0.075\linewidth}$p=5$\\
        \includegraphics[width=0.28\linewidth]{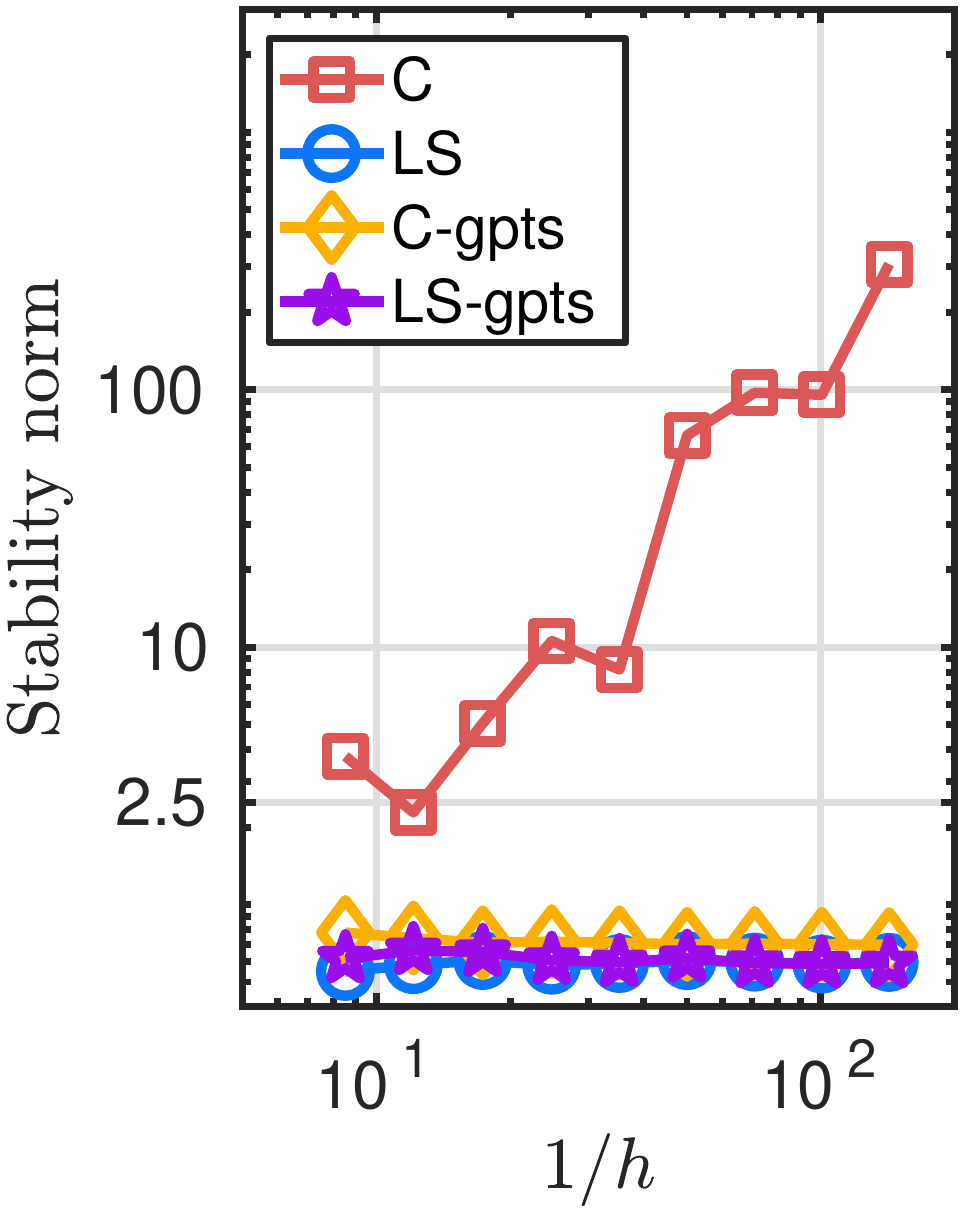} &
        \includegraphics[width=0.28\linewidth]{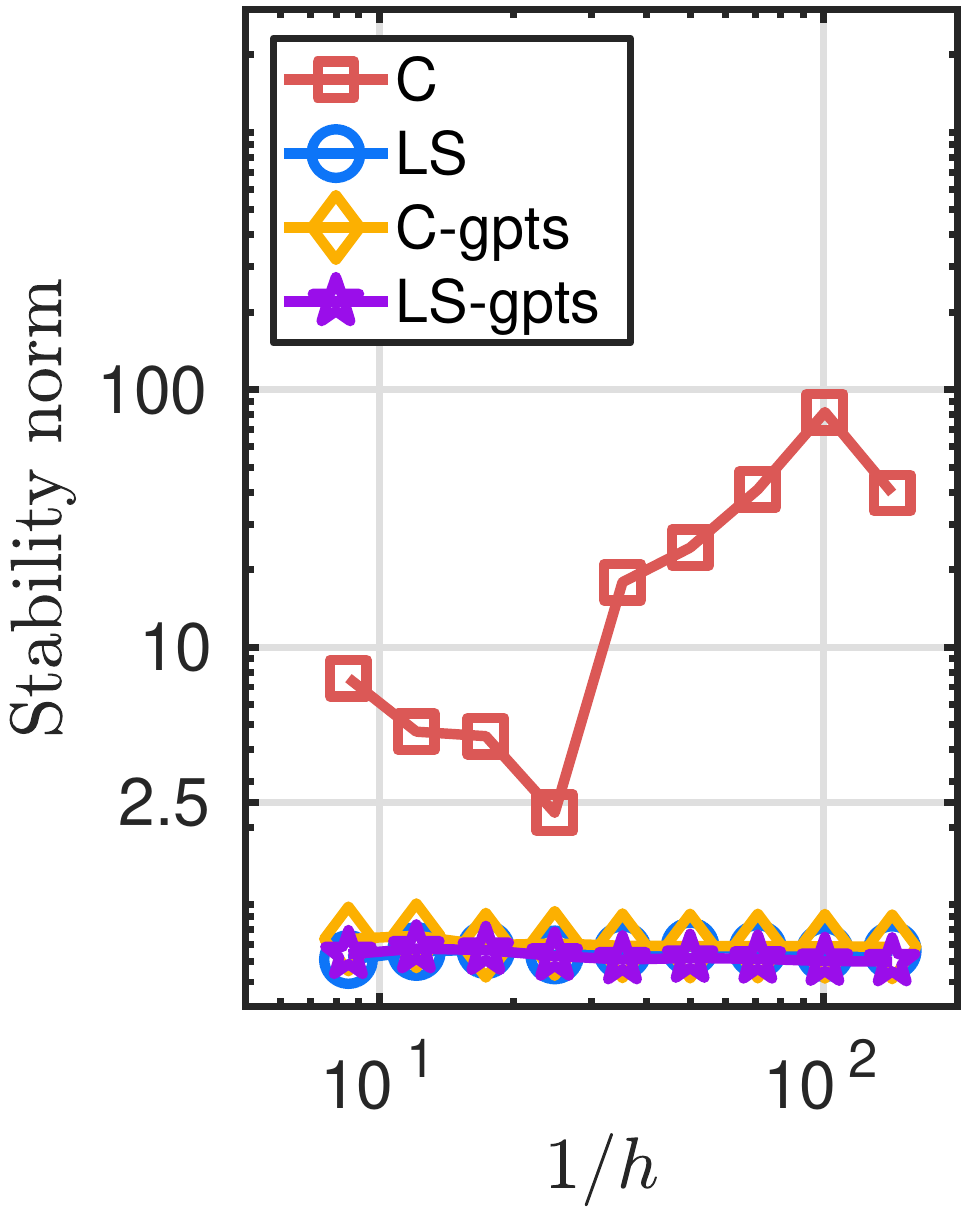} &
        \includegraphics[width=0.28\linewidth]{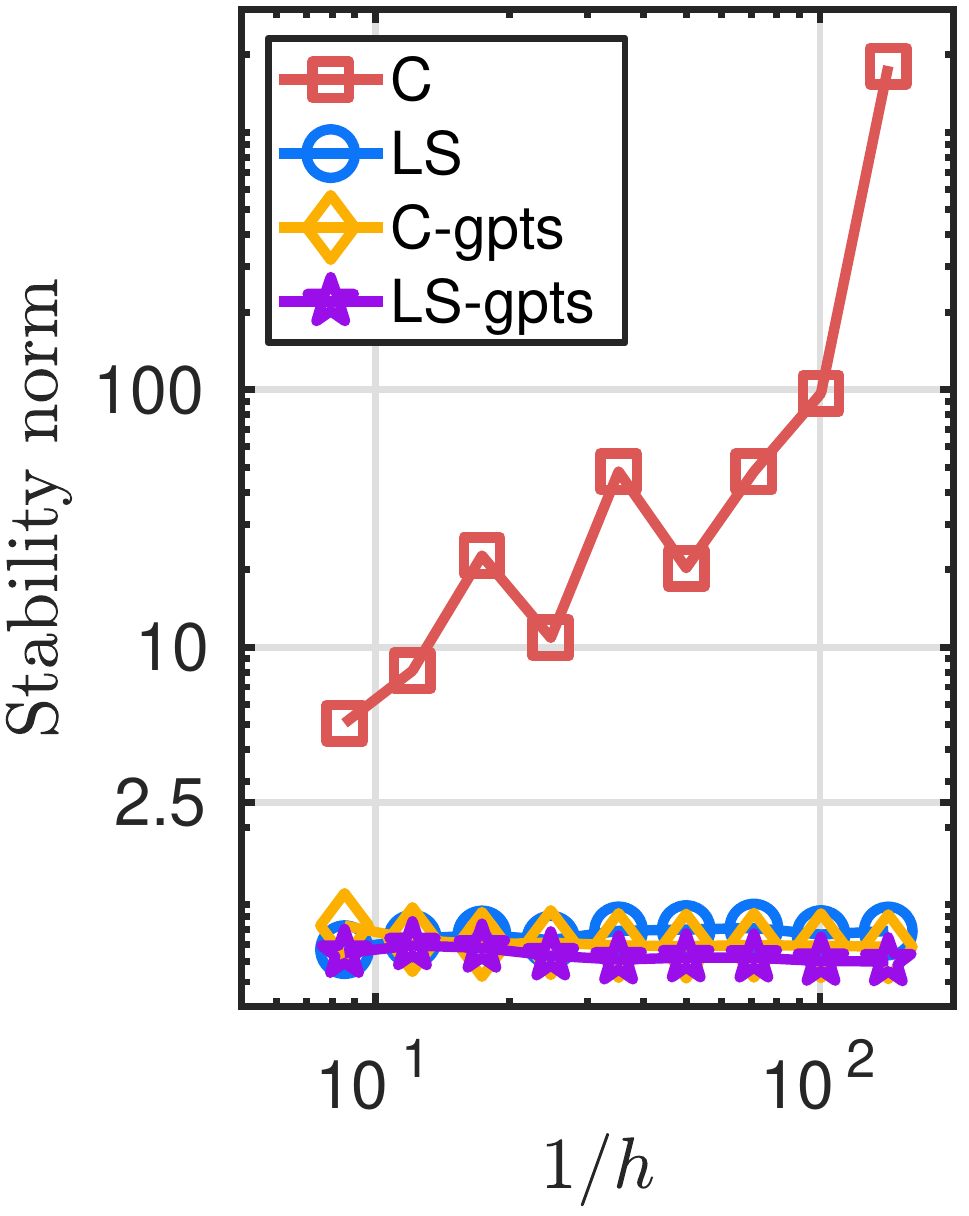}
    \end{tabular}    
    
    \caption{The stability norm \eqref{eq:stabilitynorm} as a function of $1/h$ when the oversampling parameter 
    is $q=3$ for different polynomial degrees $p$.}
    \label{fig:experiments:Poisson:h-refinement:stabilityNorm}
\end{figure}
\noindent
We observe that the stability norm of RBF-FD-LS is almost constant for all polynomial degrees $p$ which we considered.
This corresponds with the error estimate \eqref{eq:finalerror}. When $h_y=h/\sqrt{q}$, the integration error $\tau$ goes to zero as $h/q$, and the factor $\frac{\sqrt{1+\tau}}{\sqrt{1-\tau}}$ in the stability constant $C_h$ approaches 1.
The stability norm of the RBF-FD collocation method does not follow a pattern for the given PDE, parameters and node sets.
Here we emphasize that this behavior is not caused by the RBF-FD trial space, but rather by the collocation formulation in which the 
PDE is solved. \new{An interesting behavior is observed in the RBF-FD-C-Ghost case, where the stability norm is constant. 
While it is possible 
to attribute that to the imposition of ghost points, we argue that this occurs only due to the 
imposition of the extra Laplacian condition on the boundary points, which in turn 
introduces a stronger control over $\|\Delta u_h\|_{\ell_2}$: a key factor when it comes to the invertibility of $\bar D_h$ and thus the stability norm. 
To confirm this claim, we made a side experiment. The extra Laplacian condition was enforced at the boundary points, but no ghost points were used and no oversampling was employed. The resulting 
system of equations was rectangular only due to the imposition of the extra Laplacian conditions. 
The stability norm remained constant.}

The condition number of a rectangular or square matrix $A$ is defined by
$\kappa(A) = \|A\|_2\, \|A^+\|_2 = \sigma_{max}(A)/\sigma_{min}(A)$.
In \cref{fig:experiments:Poisson:h-refinement:conditionNumbers} we show the condition numbers for the two matrices involved in RBF-FD-LS: $\bar{D}_h$ 
and $E_h$.
\begin{figure}[!htb]
  \centering
  \small
    \begin{tabular}{cc}
        \hspace{0.06\linewidth}$\kappa(\bar{D}_h)$ & \hspace{0.06\linewidth}$\kappa(E_h)$ \\
        \includegraphics[width=0.35\linewidth]{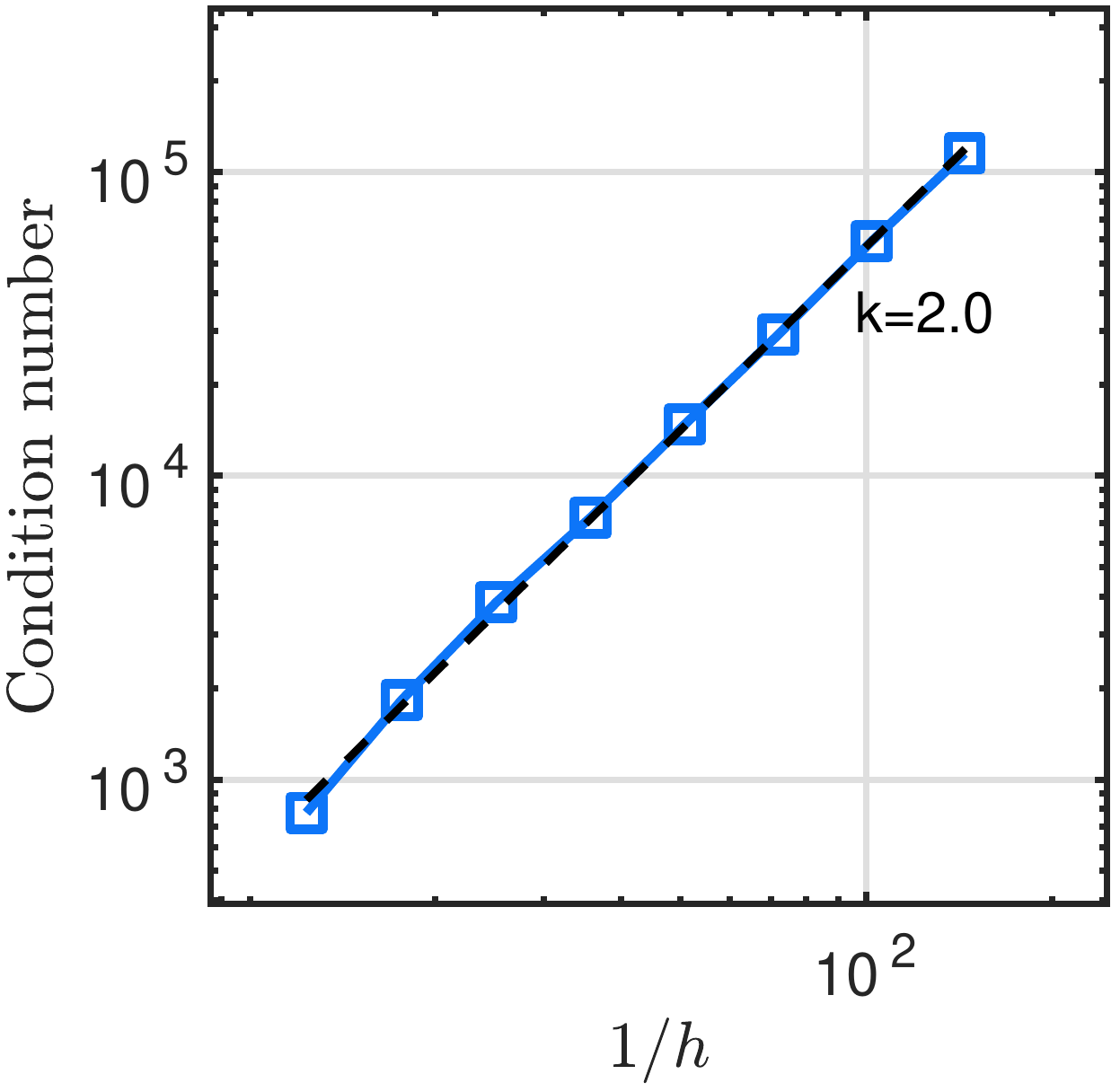} &
        \includegraphics[width=0.35\linewidth]{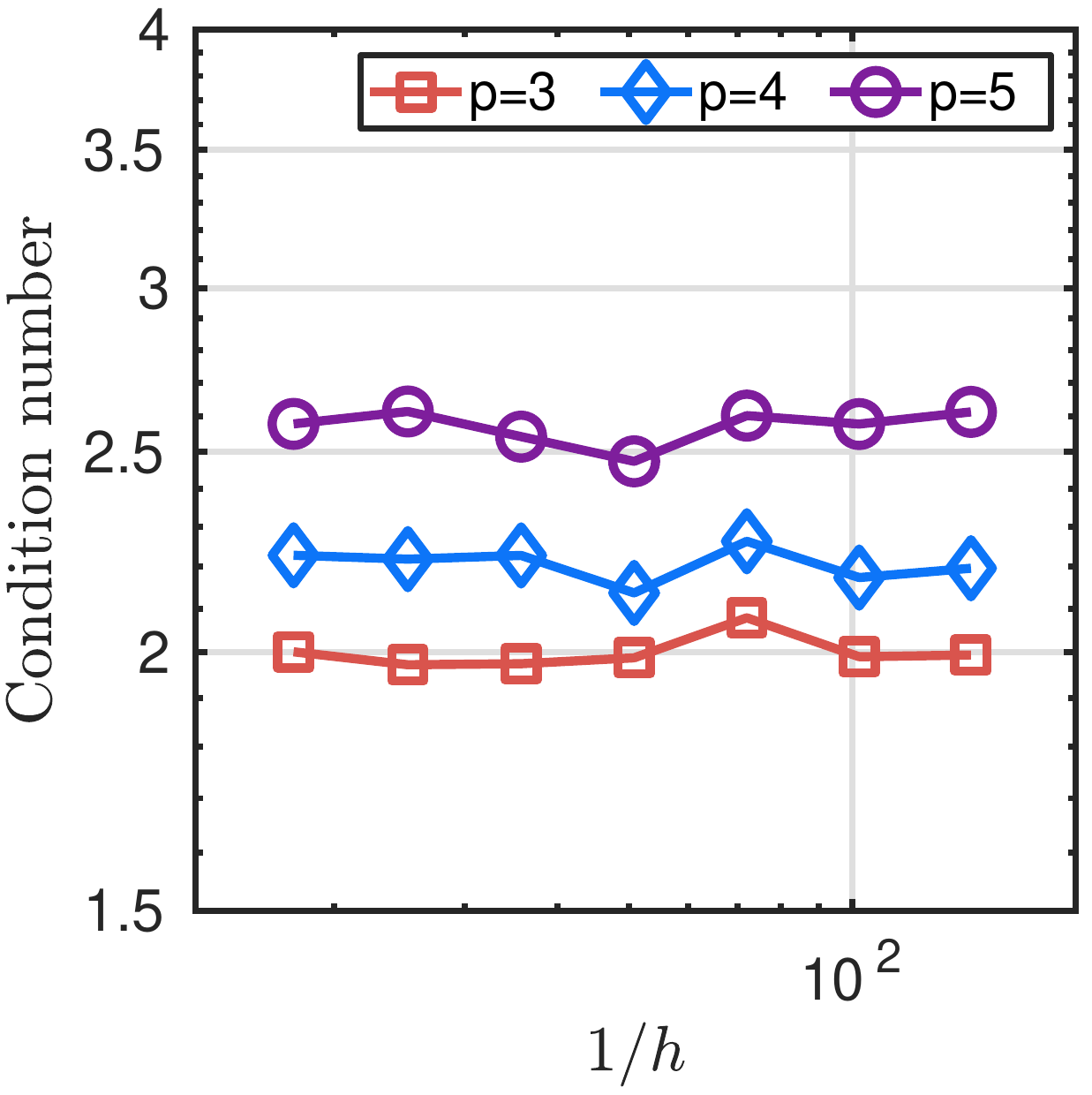} 
    \end{tabular}
    
    \caption{Left: The condition number of the PDE matrix $\bar{D}_h$ as a function 
    of the inverse node distance when the polynomial degree used for representing the trial space is $p=3$. 
    Right: The condition number of the evaluation matrix $E_h$ when the polynomial degrees used for representing the trial space are $p=3,\,4$ and $5$.}
    \label{fig:experiments:Poisson:h-refinement:conditionNumbers}
\end{figure}
We observe that $\kappa(\bar{D}_h)$ grows with $\frac{1}{h^2}$ for $p=3$. The results are almost identical for $p=4$ and $5$. 
This is an expected (optimal) growth, since $\bar{D}_h$ is a numerical second-order differentiation operator, 
which has an inverse quadratic dependence on $h$ when the stencil size is kept constant.
On the other hand $\kappa(E_h)$ is constant with respect to $h$ for all $p$, which is also an expected result, 
since $E_h$ is a numerical interpolation operator, which does not by itself yield a dependence on $h$ when the stencil 
size is kept constant.

\subsection{Approximation properties as the oversampling is increased}
\label{sec:Experiments:hy-refinement}
In this section we build understanding of the error and stability behavior for different choices of $h_y=h/\sqrt{q}$, when $h$ is fixed at $h=0.08$ (under-resolved case) and at $h=0.02$ (well-resolved case).
Three polynomial degrees $p=3,\,4$ and $5$ are used for the local interpolation matrices \eqref{eq:M}. 
The exact solution is chosen to be the truncated Non-analytic function \eqref{eq:experiments:nonanalytic}. 
%
The convergence study \new{for RBF-FD-LS} is displayed in \cref{fig:experiments:Poisson:q-refinement:error}. 
For both the under-resolved and well-resolved cases, the error \new{decays} and then levels out as $h_y$ becomes small enough. 
This behavior matches the error estimate \eqref{eq:finalerror} for the case when $h$ is fixed. As $h_y \to 0$ the term
        $\big(\frac{1+\tau}{1 - \tau}\big)^{\frac{1}{2}}\to 1$ from a larger value, and therefore $\|e\|_{\ell_2\new{(\Omega)}}$ levels out.
\begin{figure}[!htb]
  \centering
  \small
    \begin{tabular}{cc}
        \hspace{0.06\linewidth}$h=0.08$ & \hspace{0.06\linewidth}$h=0.02$ \\
        \includegraphics[height=0.35\linewidth]{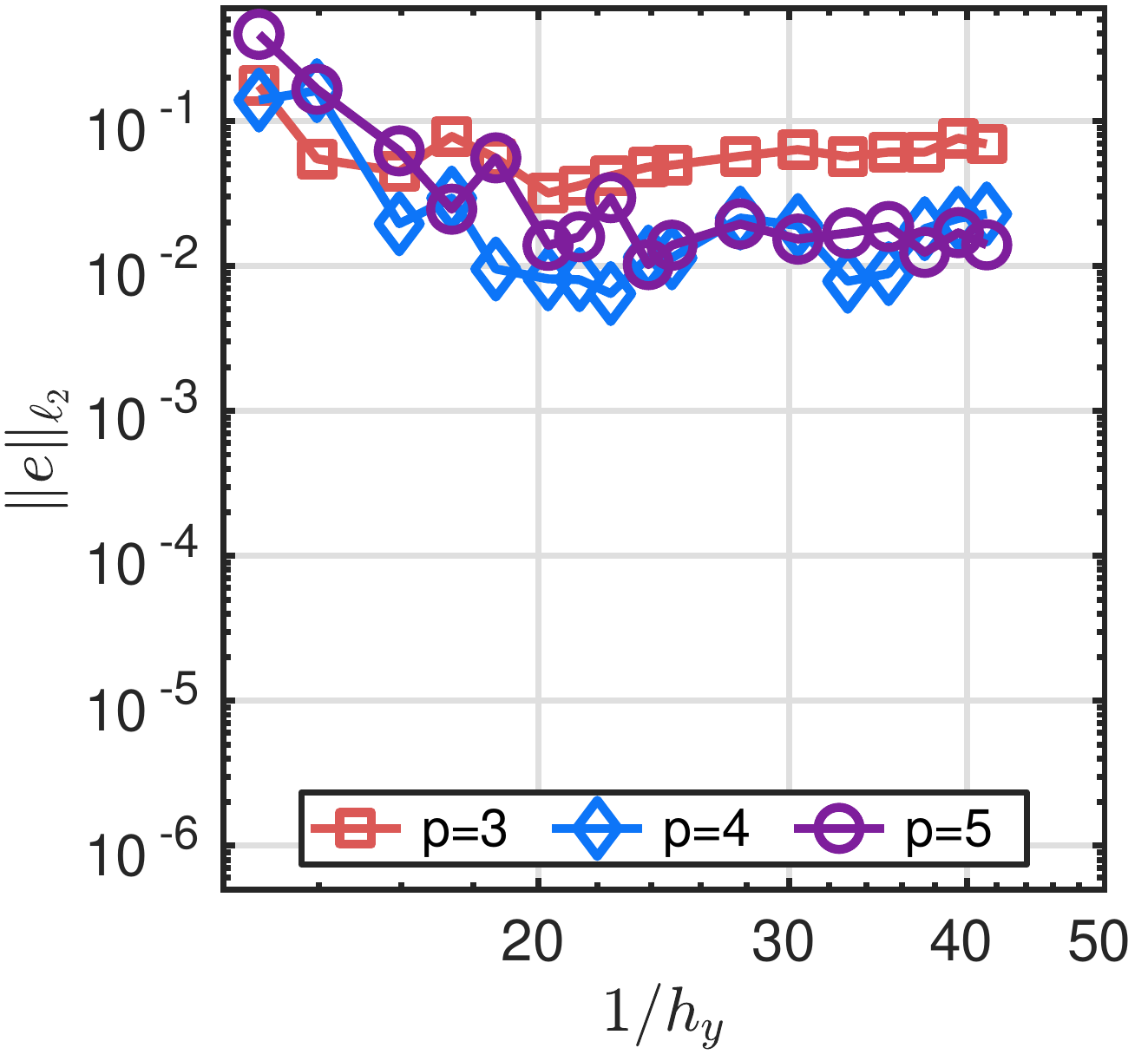} &
        \includegraphics[height=0.35\linewidth]{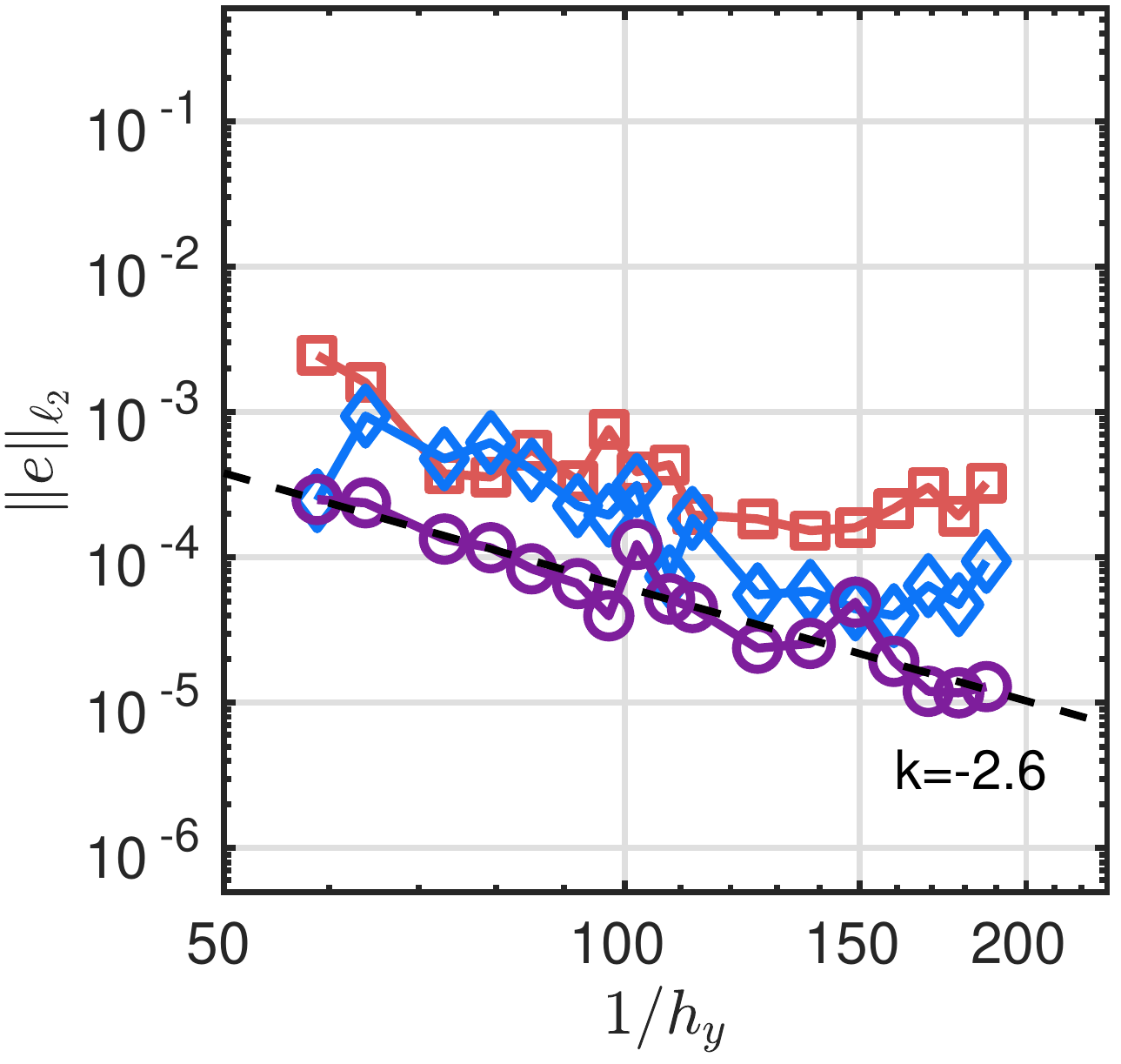} 
    \end{tabular}

    \caption{The error against $1/h_y$ which is the average distance in the $Y$ node set 
    for different choices of the polynomial degree $p$ used to form the trial space.
    The average internodal distance in the node set $X$, was for the left plot fixed at $h=0.08$ and for the right plot fixed and $h=0.02$.
    The values of $h_y$ were computed from 
    $q=(1.1$, $1.3,$ $1.7,$ $2,$ $2.3,$ $2.7,$ $3,$ $3.3,$ $3.7,$ $4,$ $5,$ $6,$ $7,$ $8,$ $9,$ $10,$ $11)$.}
    \label{fig:experiments:Poisson:q-refinement:error}
\end{figure}
%
The stability norm behavior is shown in \cref{fig:experiments:Poisson:q-refinement:stability}, from which we observe that 
in both the well-resolved and under-resolved cases, the norm first rapidly decays and then flattens out when $h_y$ is small enough.
The approximate point when the stability norm starts to flatten out is at $1/h_y \approx 28$ (corresponding to $q=3.7$) 
for the under-resolved case and at $1/h_y \approx 70$ ($q=3.3$) for the well-resolved case. 

\begin{figure}[!htb]
  \centering
  \small
    \begin{tabular}{cc} 
        \hspace{0.03\linewidth}$h=0.08$ & \hspace{0.03\linewidth}$h=0.02$ \\
        \includegraphics[height=0.36\linewidth]{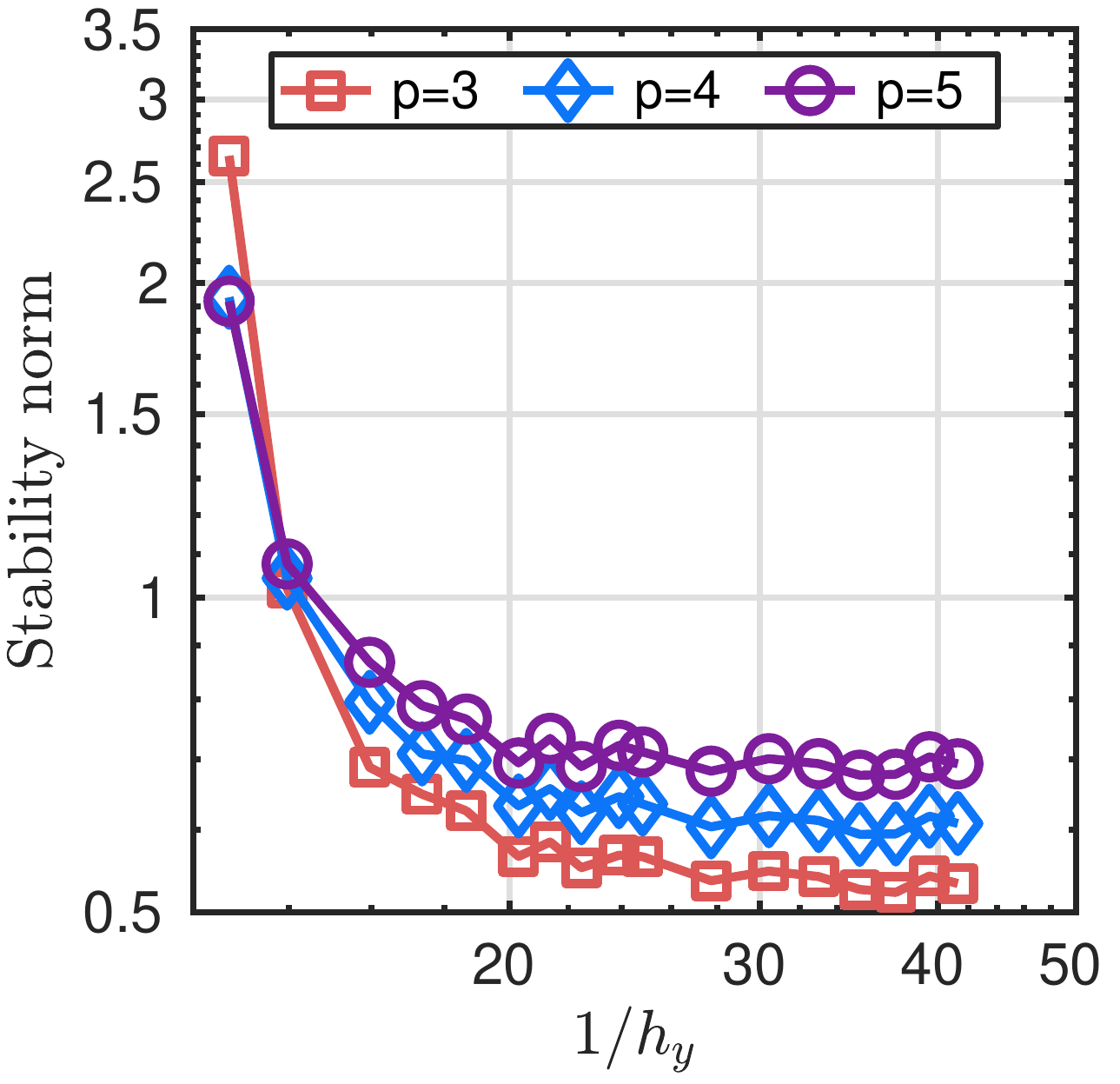} &
        \includegraphics[height=0.36\linewidth]{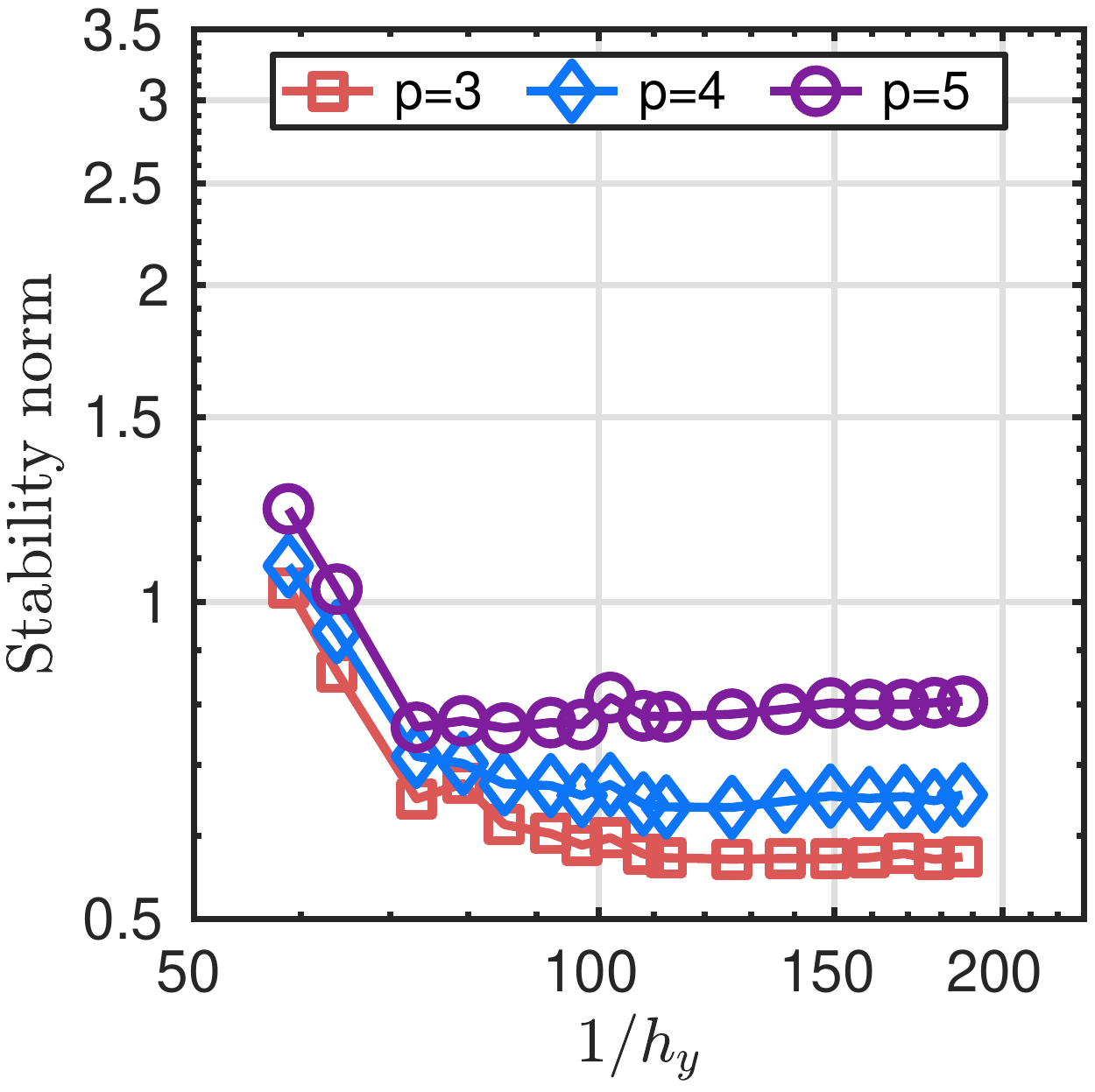} 
    \end{tabular}
        
    \caption{Both plots show the stability norm \eqref{eq:stabilitynorm}
    as a function of $1/h_y$, the average inverse node distance in the point set $Y$. 
    Different choices of the polynomial degree $p$ were used to form the trial space.
    The average internodal distance $h$ in the node set $X$, was for the left plot fixed at $h=0.08$ and for the right plot fixed and $h=0.02$, 
    illustrating the under-resolved and the well-resolved case. The values of $h_y$ were computed from
    $q=(1.1$, $1.3,$ $1.7,$ $2,$ $2.3,$ $2.7,$ $3,$ $3.3,$ $3.7,$ $4,$ $5,$ $6,$ $7,$ $8,$ $9,$ $10,$ $11)$.}
    \label{fig:experiments:Poisson:q-refinement:stability}
\end{figure}
\subsection{Eigenvalue spectrum as the oversampling is increased}


\new{In the case of the least-squares matrices, we use the relation $u(Y) = E_h u(X)$ to arrive at $E_h^+ u(Y) = u(X)$, which is then used in the discretized PDE 
$\bar D_h u(X) = \bar F(Y)$ to arrive at:
$$\bar D_h E_h^+ u(Y) = \bar F(Y).$$
We then investigate the eigenvalues of the square $M\times M$ matrix $\bar D_h E_h^+$. 
Both matrices $\bar D_h$ and $E_h$ are rectangular of size $M \times N$, thus the rank of each at most $N$. 
The rank of $\bar D_h E_h^+$ can then not be larger than $N$, implying that there will always exist a nullspace of size $M-N$. This is not a problem when we use the method for solving PDEs as we always solve for the unique $N$-dimensional least-squares solution.
}

\new{In Figure \ref{fig:experiments:Poisson:q-refinement:eigenvalues_Laplacian_EDinv}, we display the eigenvalue spectrum of the PDE matrix for $N=1000$ and an increasing oversampling parameter $q$. 
We observe that the real part of the spectra in the least-squares case shrink as $q$ is increased. 
The spectra of the collocation matrices have a larger negative real part compared to the spectra of the least-squares matrices.}
\begin{figure}[!htb]
    \centering
        \includegraphics[width=0.24\linewidth,height=0.3\linewidth]{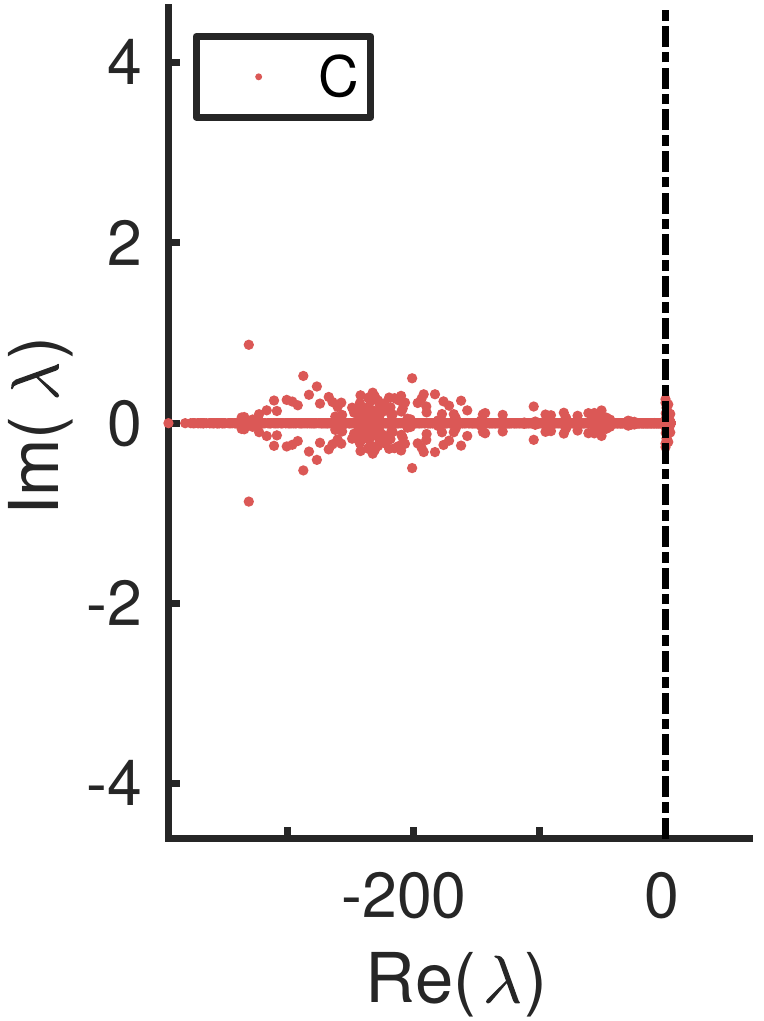}
        \includegraphics[width=0.24\linewidth,height=0.3\linewidth]{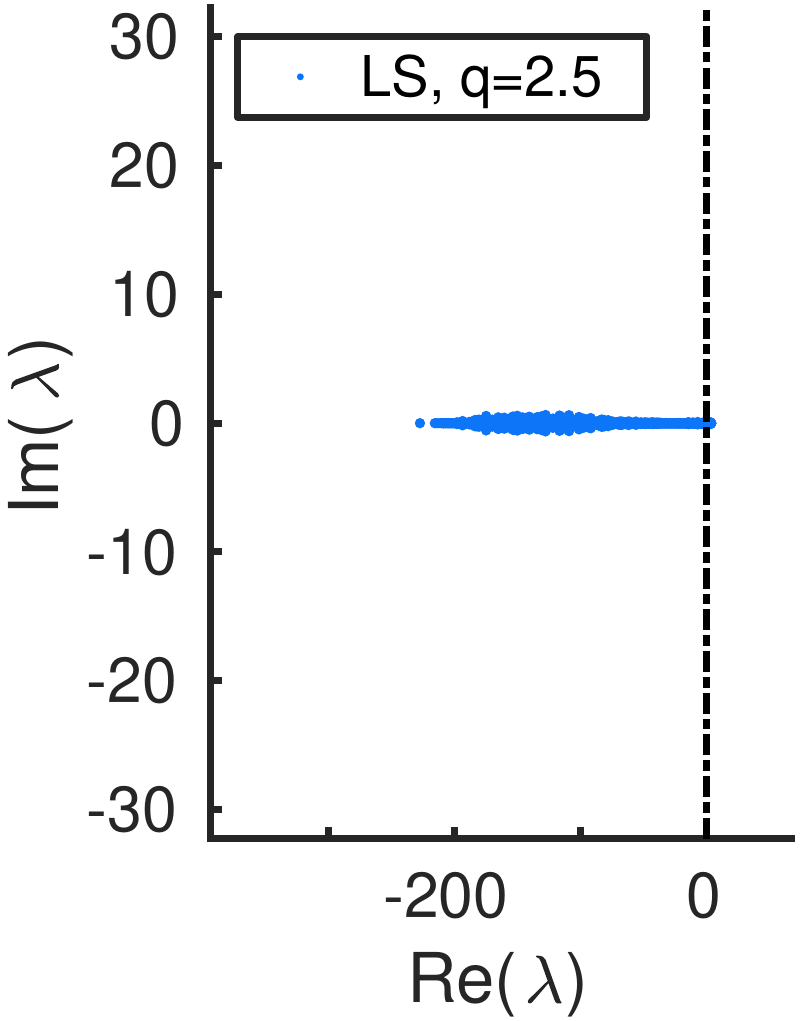}
        \includegraphics[width=0.24\linewidth,height=0.3\linewidth]{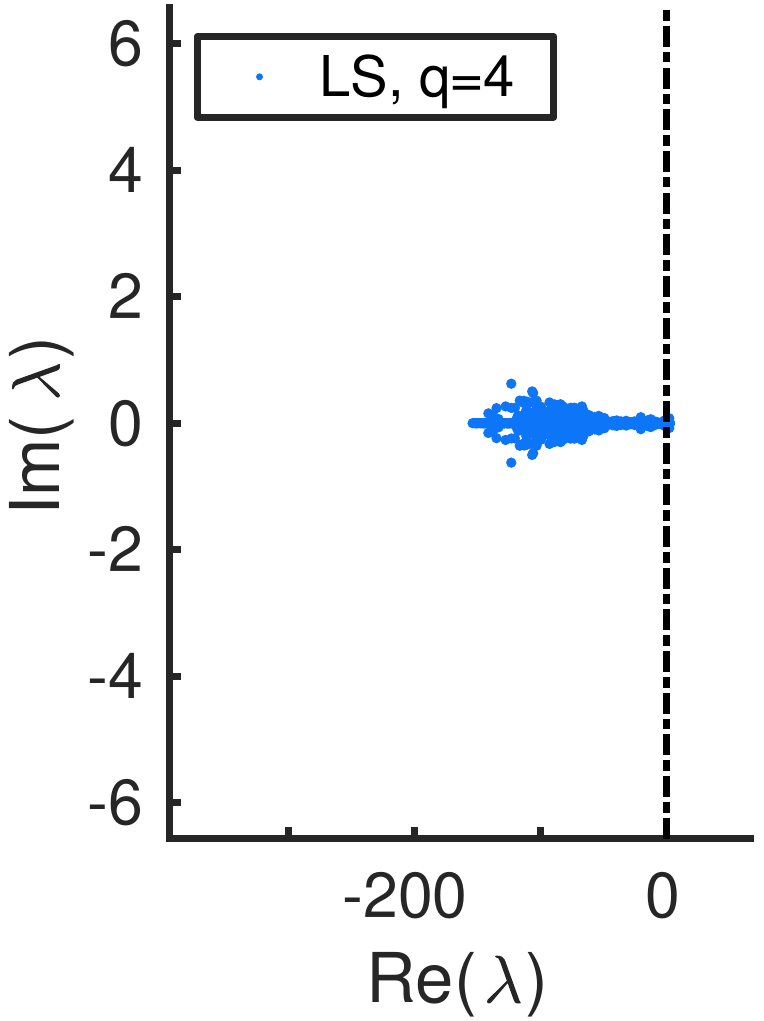}
        \includegraphics[width=0.24\linewidth,height=0.3\linewidth]{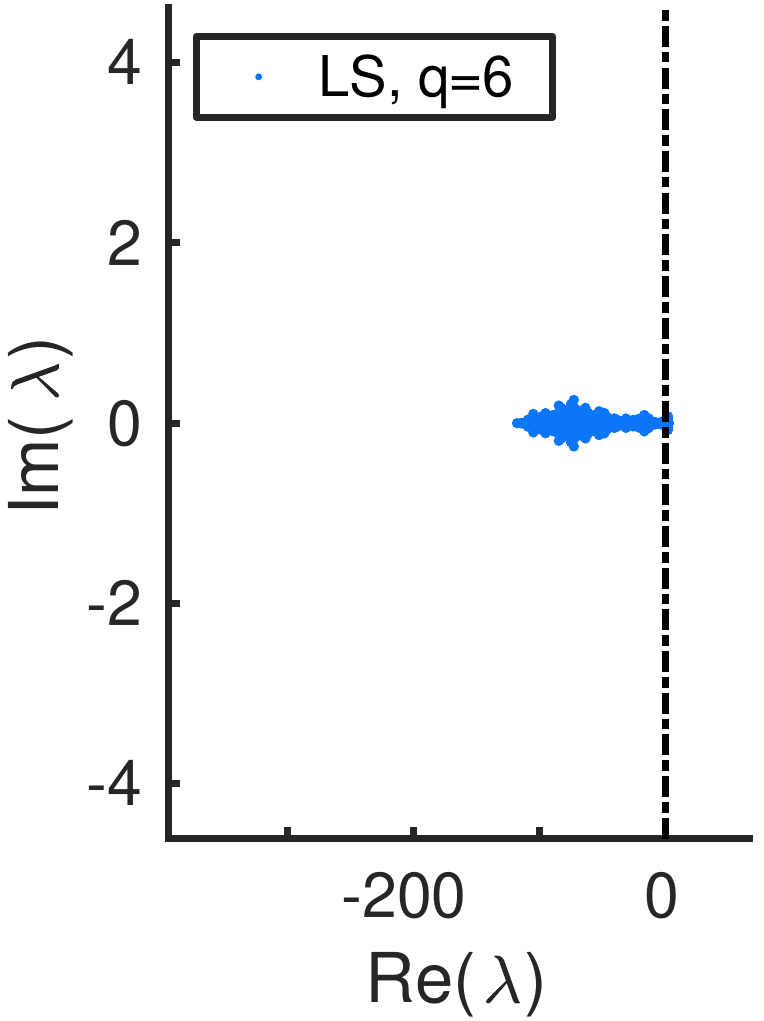}        
        \includegraphics[width=0.24\linewidth,height=0.3\linewidth]{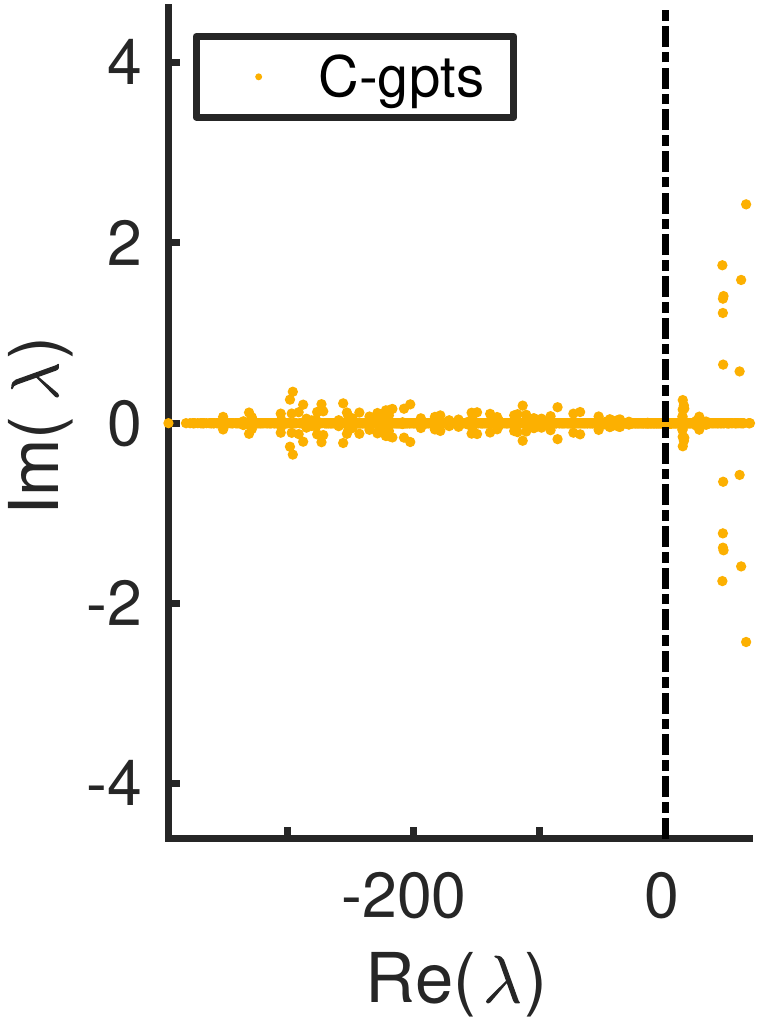}
        \includegraphics[width=0.24\linewidth,height=0.3\linewidth]{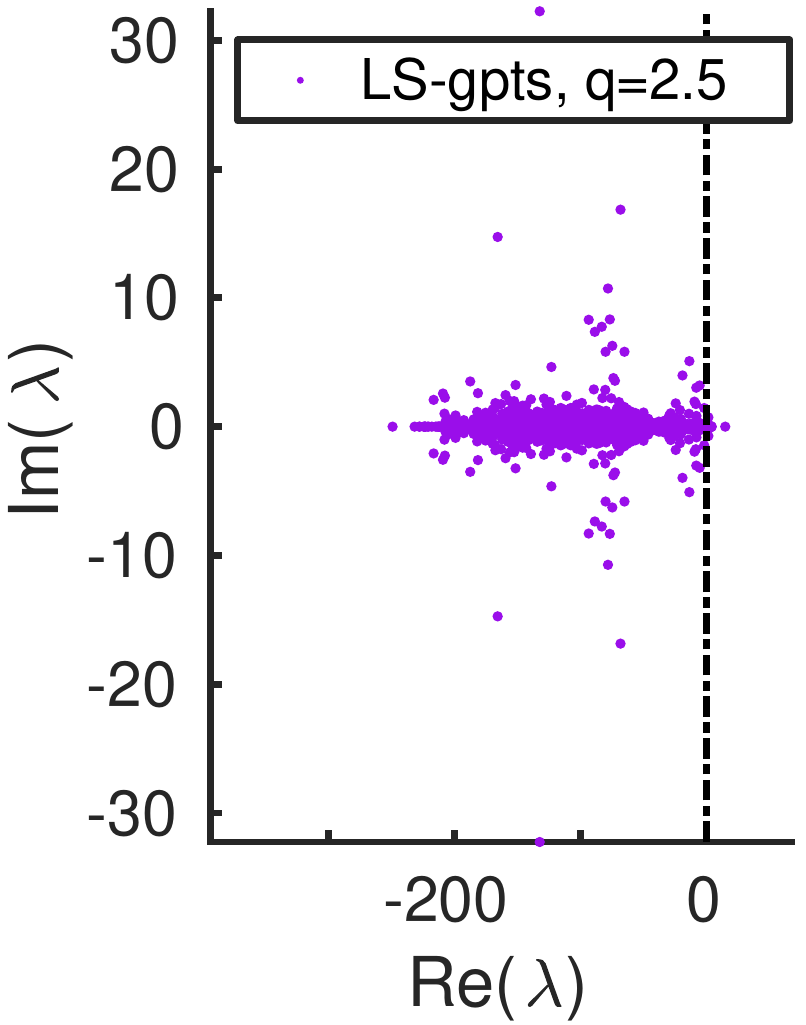}
        \includegraphics[width=0.24\linewidth,height=0.3\linewidth]{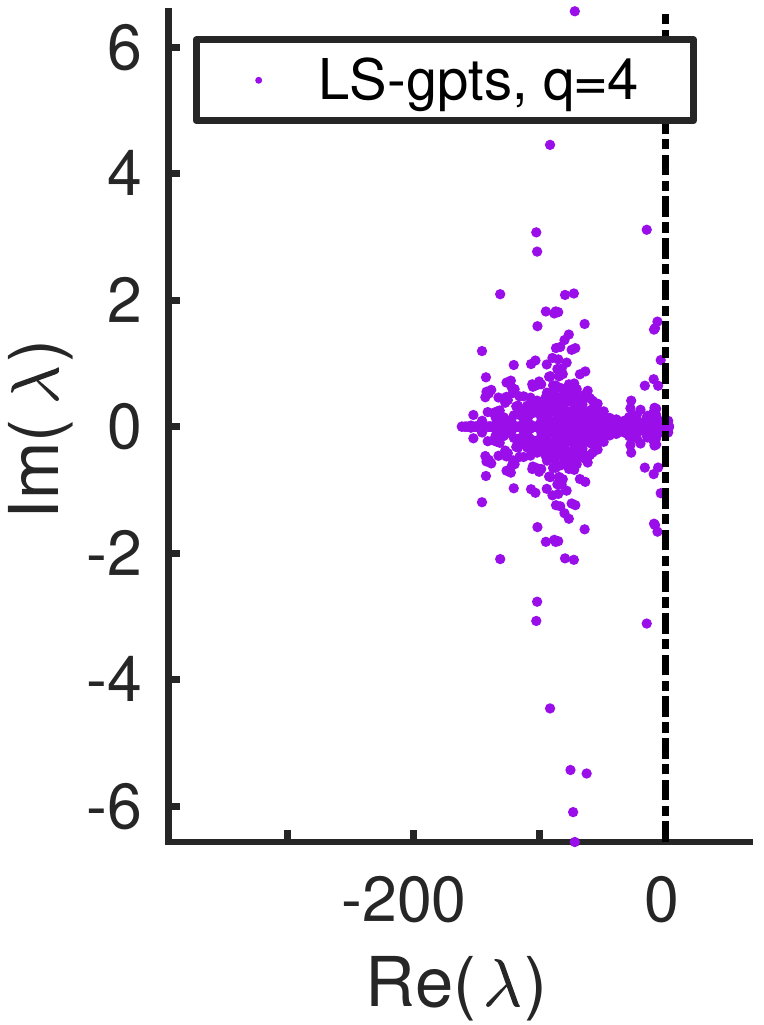}
        \includegraphics[width=0.24\linewidth,height=0.3\linewidth]{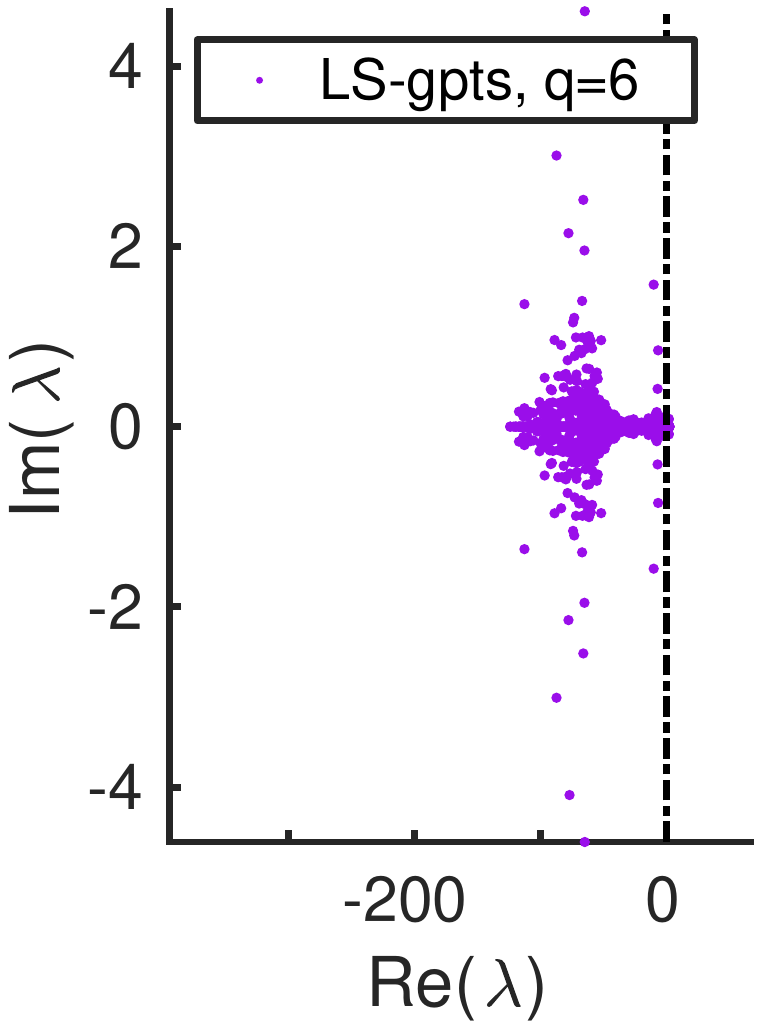}        
    \caption{Eigenvalues of $E\bar D^+$ as the oversampling parameter $q$ is increased, where $\bar D$ discretizes the Poisson equation with Dirichlet and Neumann boundary conditions imposed on 
    two disjoint parts of the boundary of the domain. In this test the number of nodes is $N=1000$ and the polynomial degree used to construct the local approximations is $p=4$.}
    \label{fig:experiments:Poisson:q-refinement:eigenvalues_Laplacian_EDinv}
\end{figure}
\new{In addition, we also display the eigenvalue spectra of the discretized first order operator $-g \cdot \nabla u$, where $g = [0, 1]$.
We only impose the Dirichlet boundary condition at the location of the polar angle $\theta \in [0, \pi]$. In this case, 
the least-squares cases have a less distinct behavior. However, we can observe that the negative and positive real parts of the spectrum
are moving slightly towards the imaginary axis as the oversampling is increased.}
\begin{figure}[!htb]
    \centering
        \includegraphics[width=0.24\linewidth,height=0.3\linewidth]{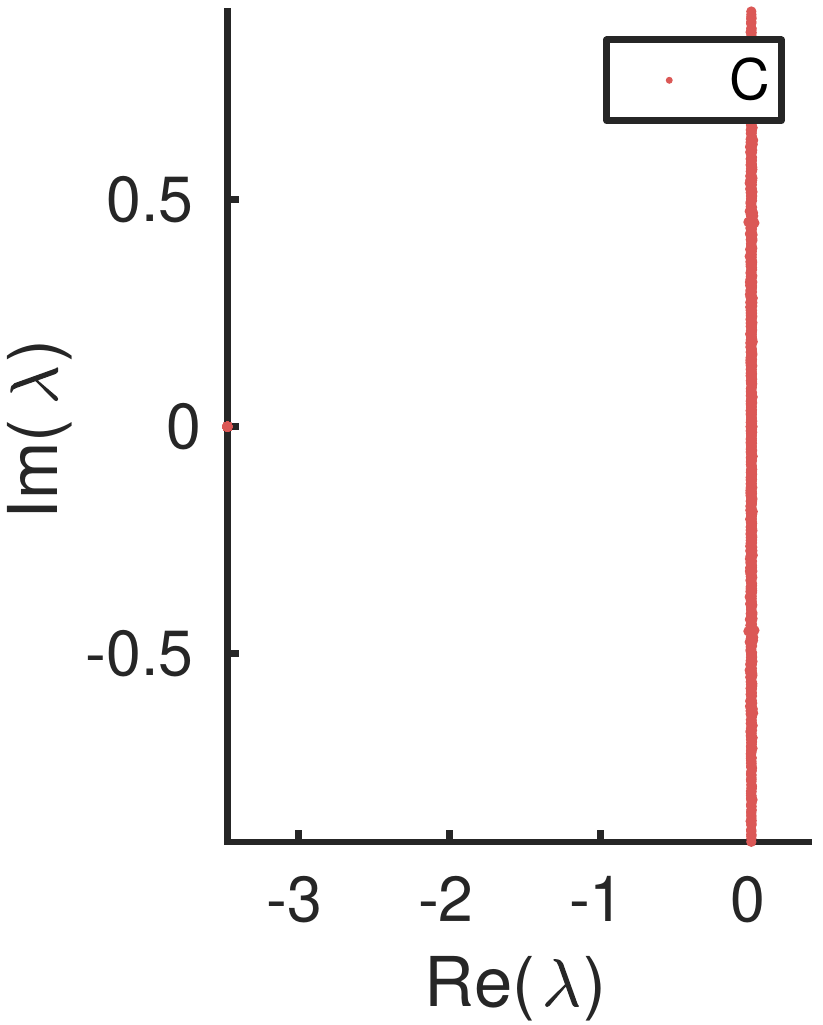}
        \includegraphics[width=0.24\linewidth,height=0.3\linewidth]{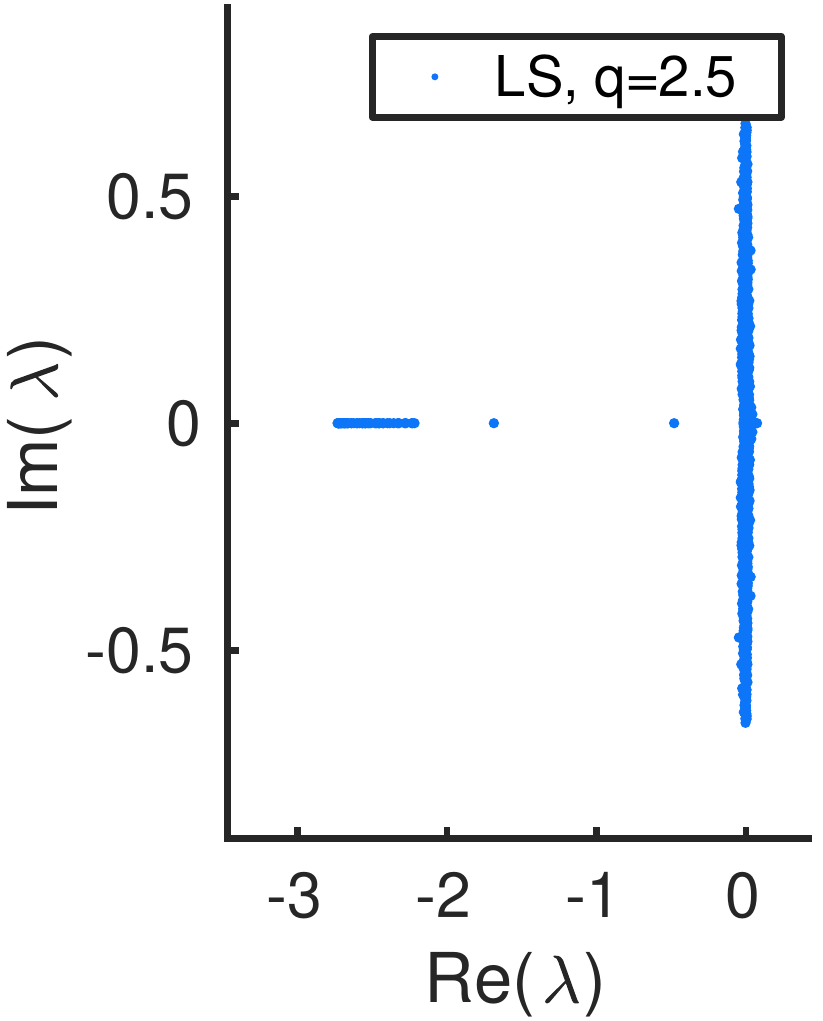}
        \includegraphics[width=0.24\linewidth,height=0.3\linewidth]{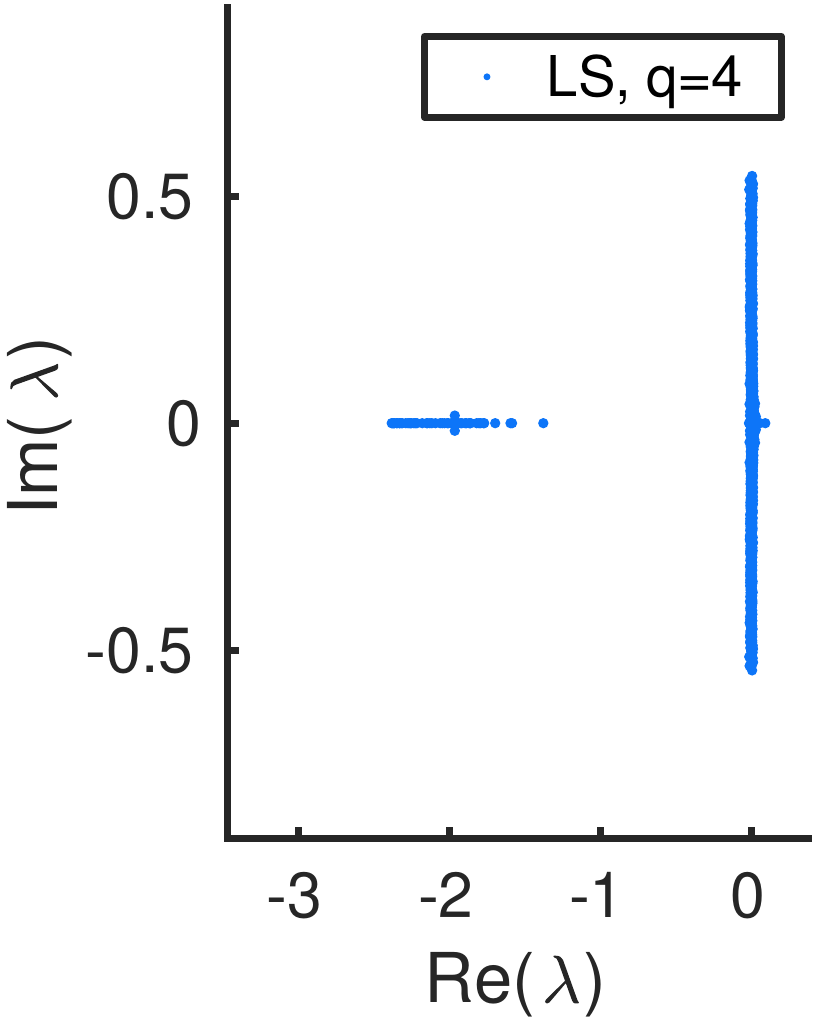}
        \includegraphics[width=0.24\linewidth,height=0.3\linewidth]{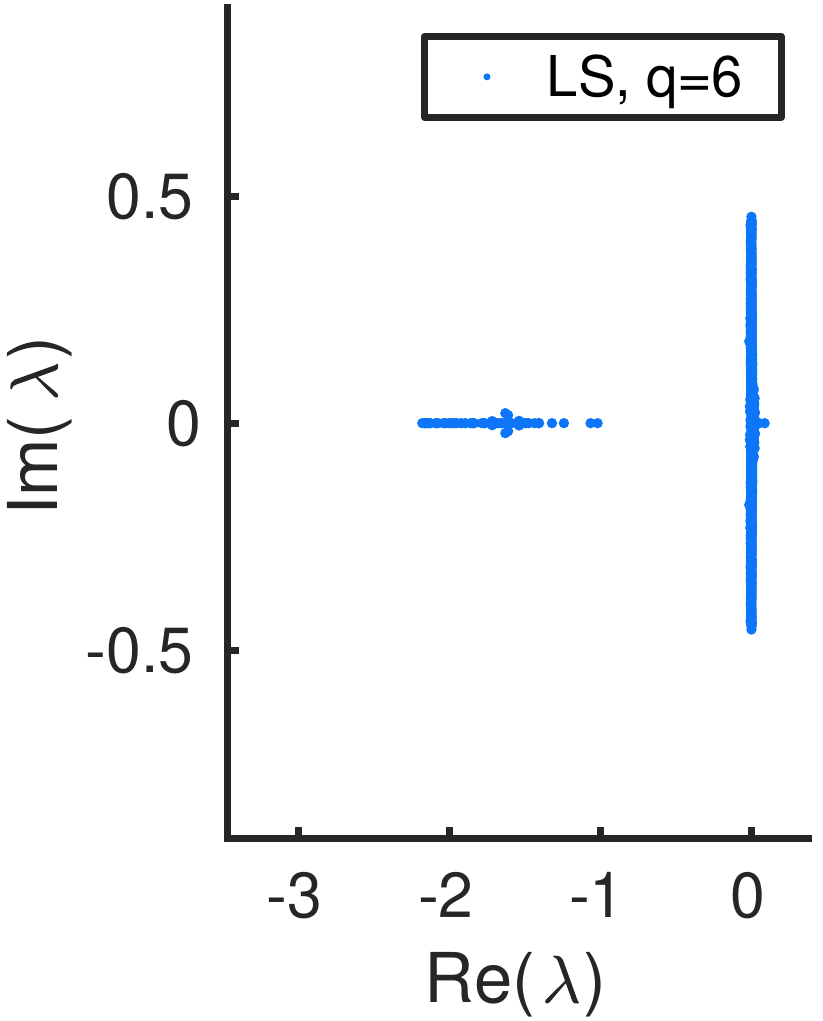}        
        \includegraphics[width=0.24\linewidth,height=0.3\linewidth]{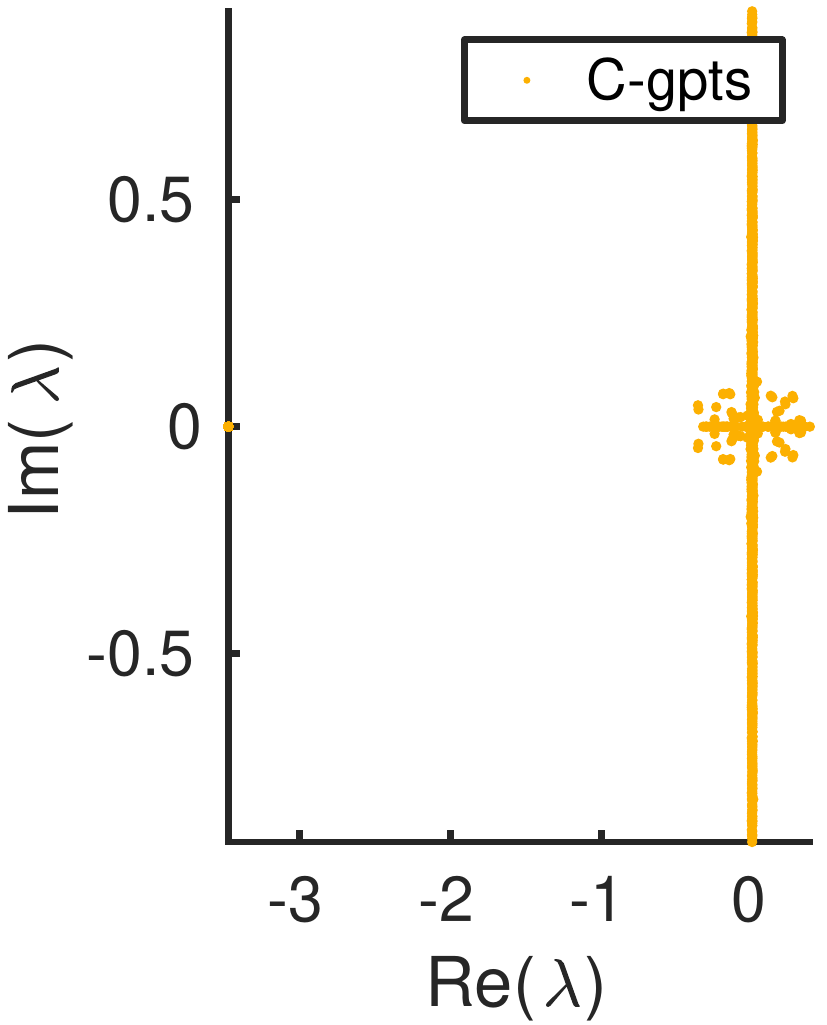}
        \includegraphics[width=0.24\linewidth,height=0.3\linewidth]{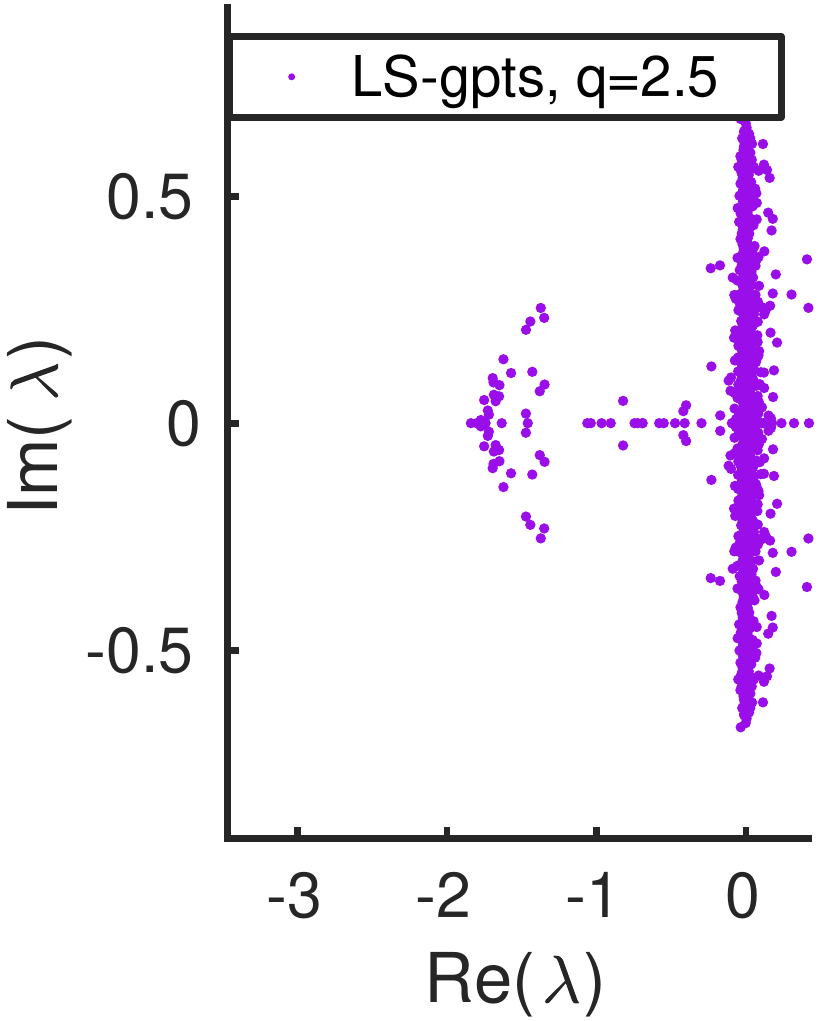}
        \includegraphics[width=0.24\linewidth,height=0.3\linewidth]{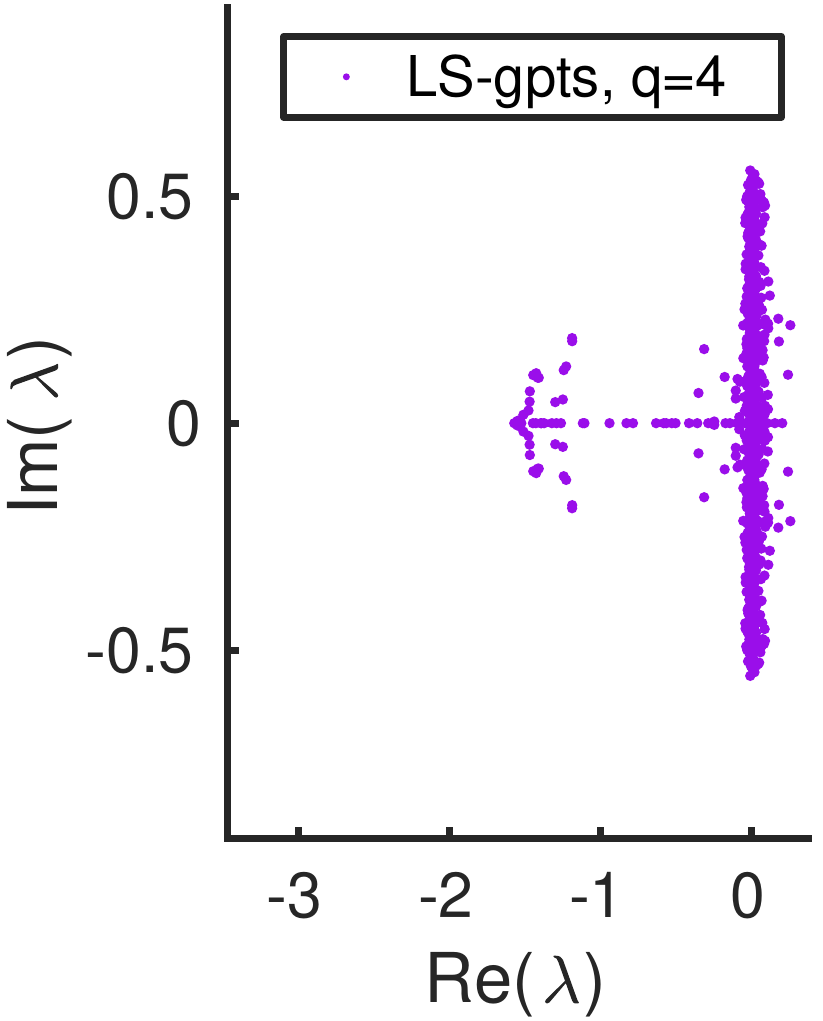}
        \includegraphics[width=0.24\linewidth,height=0.3\linewidth]{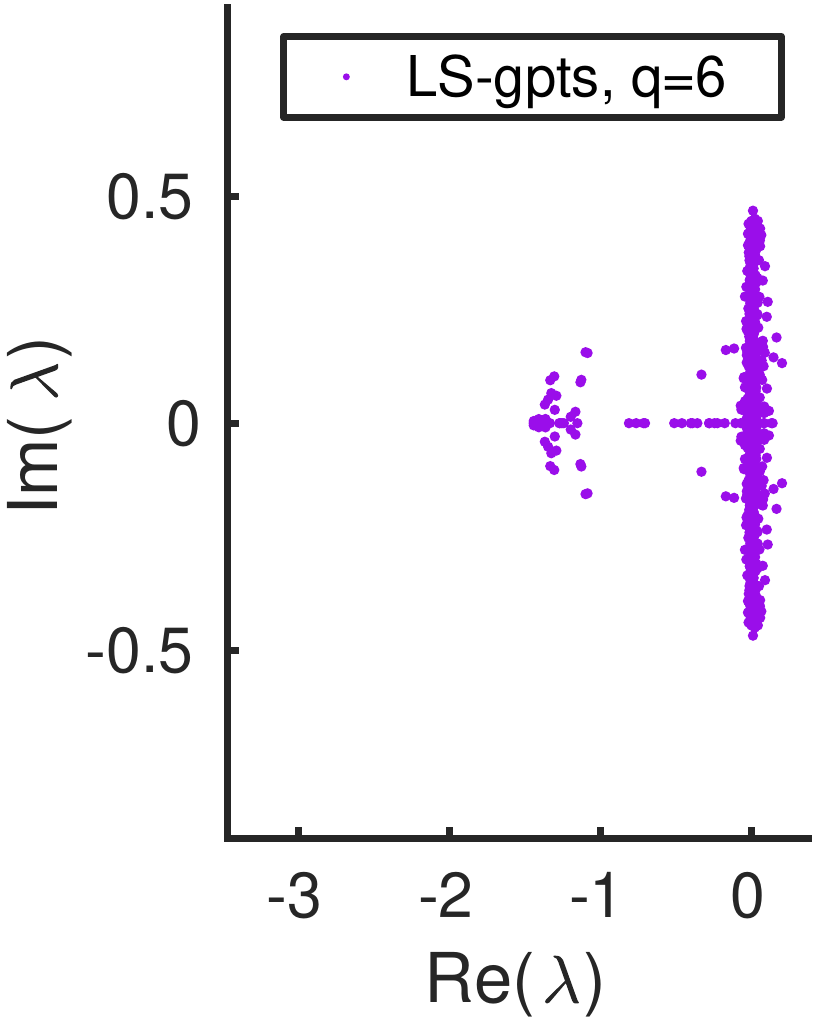}        
    \caption{Eigenvalues of $E\bar D^+$ as the oversampling parameter $q$ is increased, where $\bar D$ discretizes the advection equation with Dirichlet boundary condition at the inflow boundary
    of a domain $\Omega$. In this test the number of nodes is $N=1000$ and the polynomial degree used to construct the local approximations is $p=4$.}
    \label{fig:experiments:Poisson:q-refinement:eigenvalues_Advection_EDinv}
\end{figure}
\new{The increase in oversampling improves the eigenvalue spectra of the rectangular differentiation matrices 
in both eigenvalue examples.}

\subsection{Approximation properties as the polynomial degree is increased}
Here we increase the number of points per stencil, together with increasing the polynomial degree $p$ used to form 
the stencil-based interpolant \eqref{eq:M}, while the distance $h$ between the stencil points is kept the same. 
This is denoted by $p$-refinement. 
We consider polynomial degrees 
up to $p=12$ in order to test the limits of the method. Two different solution functions 
are considered, the truncated Non-analytic function \eqref{eq:experiments:nonanalytic} and the Rational sine 
function \eqref{eq:experiments:rationalsine}. The methods 
RBF-FD-LS, RBF-FD-LS-Ghost, RBF-FD-C, RBF-FD-C-Ghost are compared when $h$ is fixed at $h=0.08$ (under-resolved case) and 
at $h=0.02$ (well-resolved case).
\begin{figure}[!htb]
  \centering
  \small
    \begin{tabular}{ccc}
                   & \hspace{0.075\linewidth}Non-analytic & \hspace{0.075\linewidth}Rational sine \\
    \rotatebox[origin=c]{90}{\hspace{0.35\linewidth} $h = 0.08$} & \includegraphics[width=0.35\linewidth]{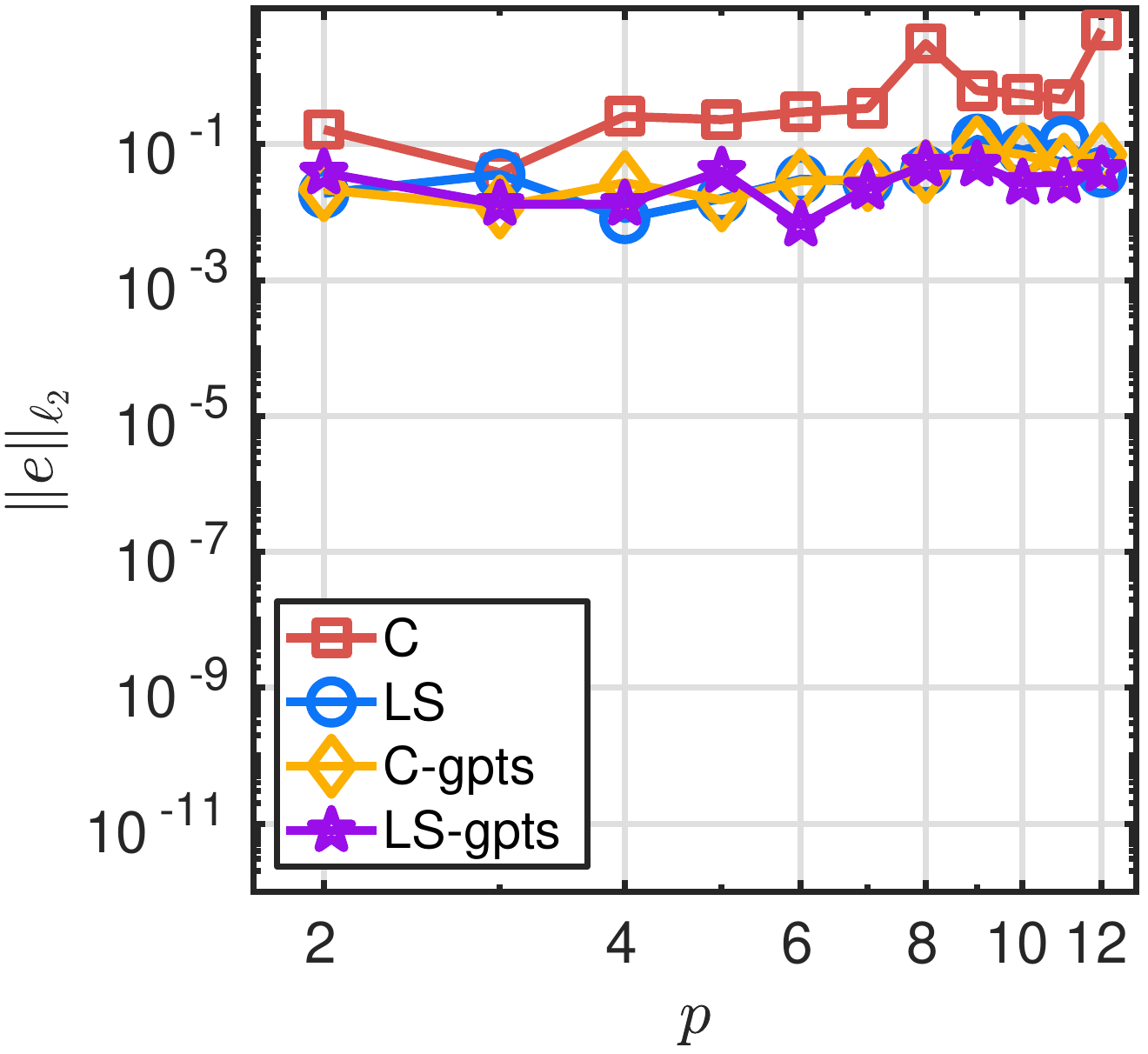} & \includegraphics[width=0.35\linewidth]{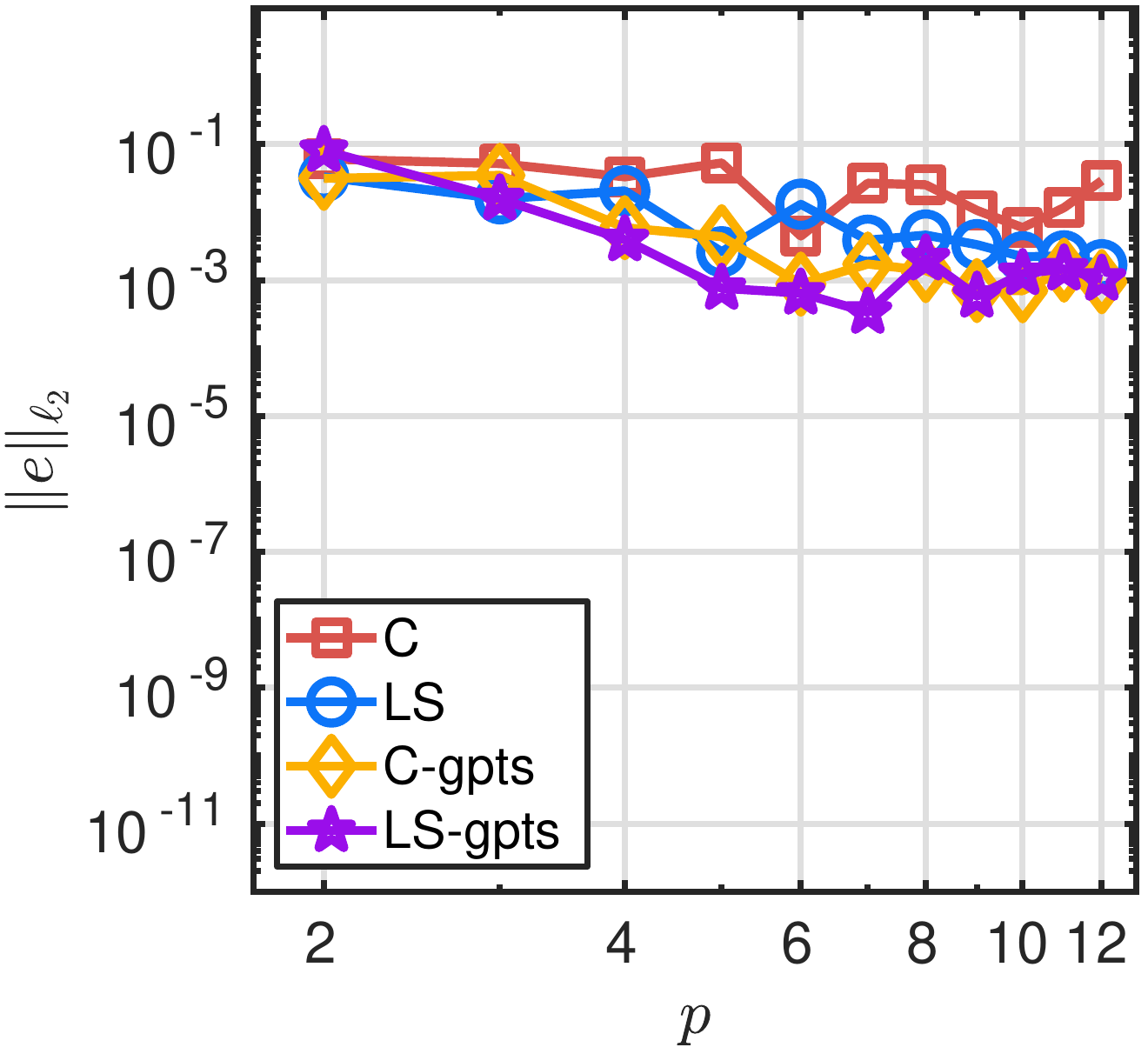} \\
    \rotatebox[origin=c]{90}{\hspace{0.35\linewidth} $h = 0.02$} & \includegraphics[width=0.35\linewidth]{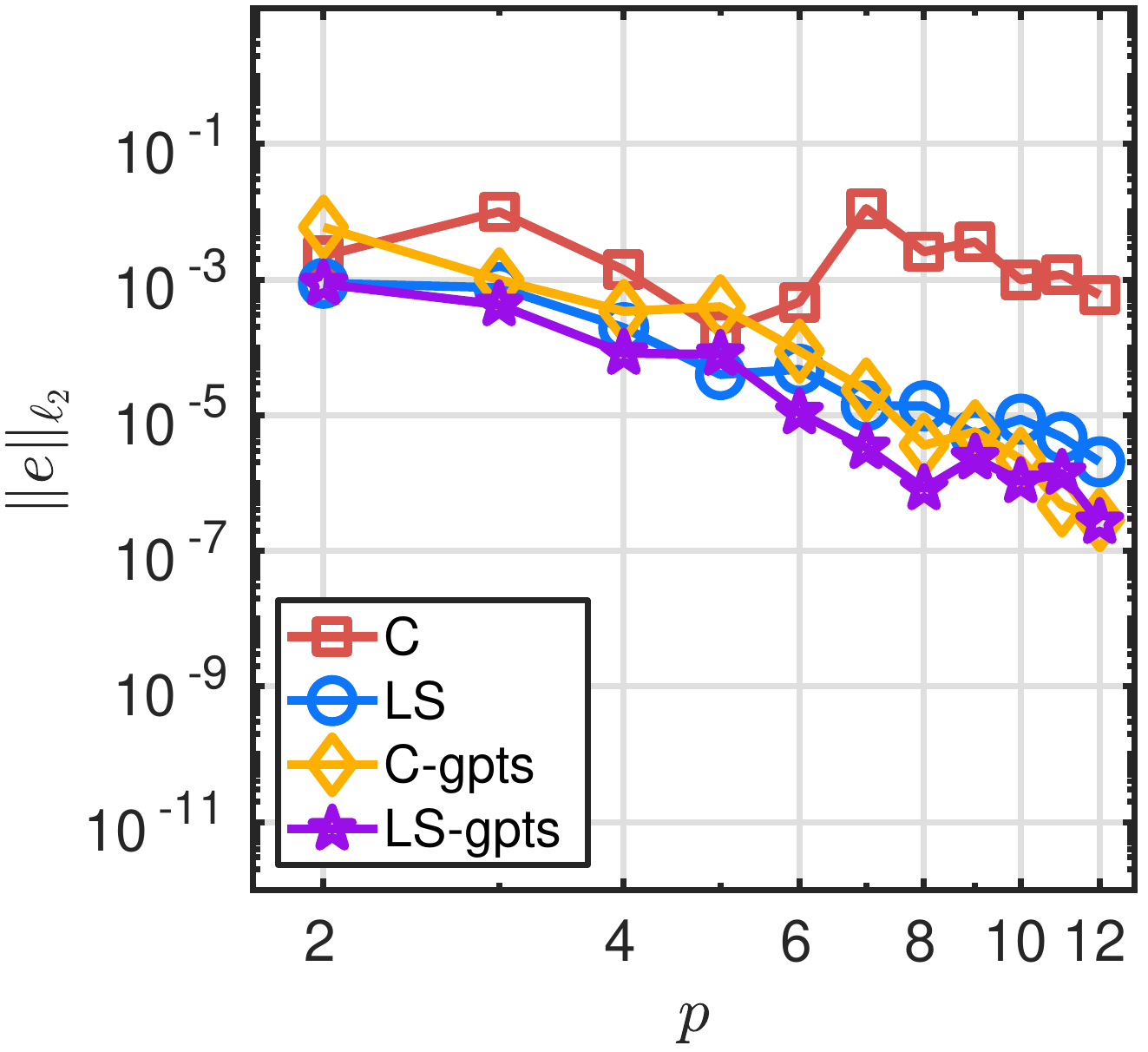}  & \includegraphics[width=0.35\linewidth]{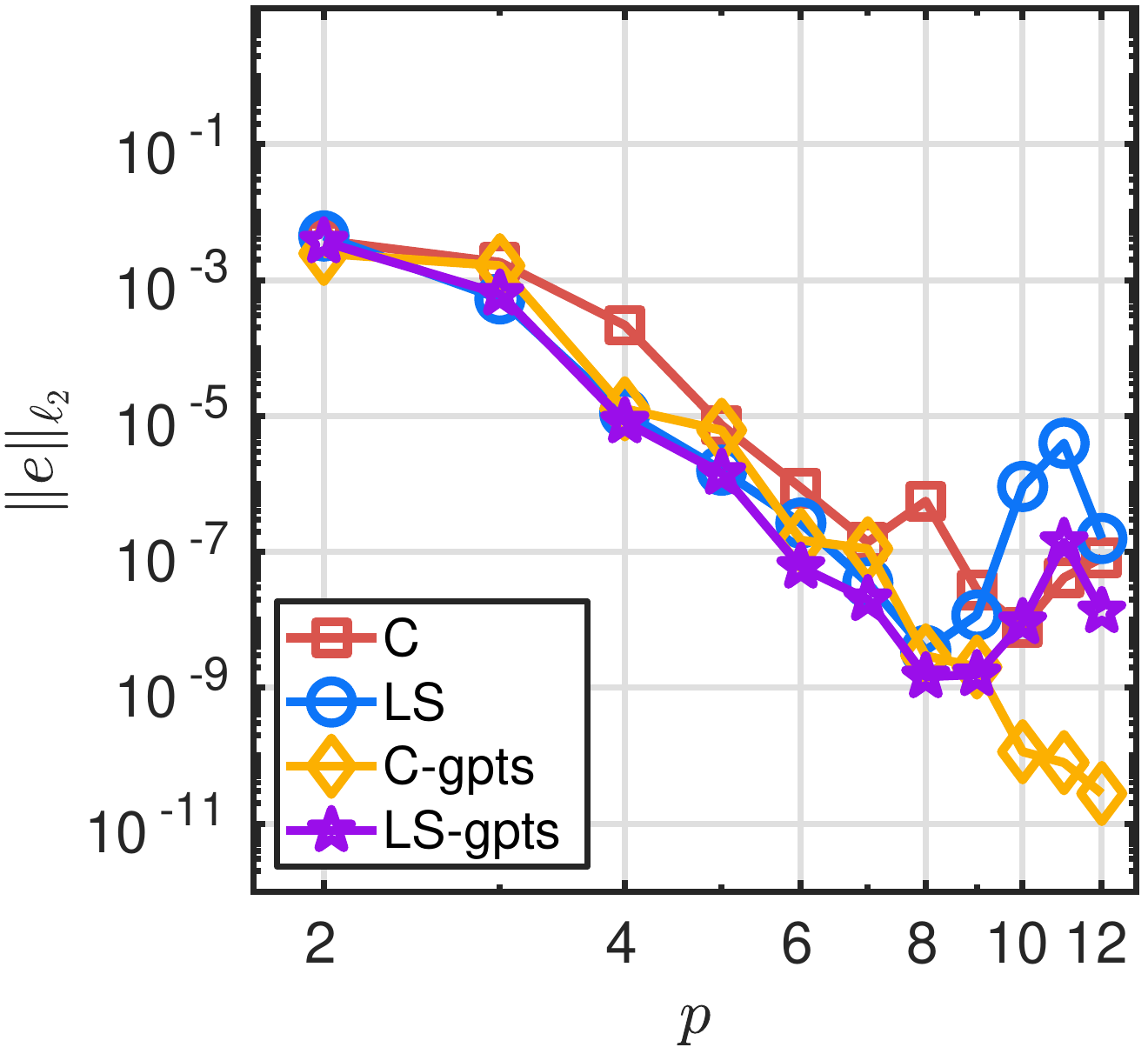} 
    \end{tabular}
        
    \caption{The relative error in the $p$-refinement mode: the polynomial degree used to form the stencil-based interpolation matrix \eqref{eq:M} is increased, while the internodal distance between the stencil points is fixed at 
    $h=0.08$ and $h=0.02$.}
    \label{fig:experiments:Poisson:p-refinement:error}
\end{figure}
In \cref{fig:experiments:Poisson:p-refinement:error}, we see that the error for both resolutions and both manufactured solutions is smaller for 
RBF-FD-LS compared with RBF-FD-C. 
For the under-resolved case ($h=0.08)$), there is some improvement of the error when increasing $p$ for the Rational sine function. 
The results are worse for the truncated Non-analytic function that has large derivatives and requires higher resolution. In this case, the error increases for $p>4$. 
In the well-resolved case ($h=0.02$), we observe convergence with $p$ in all cases except RBF-FD-C for the Non-analytic function. 
For $p\geq 8$, round-off errors prevent further convergence for RBF-FD-LS. The convergence trend for RBF-FD-C levels out earlier than for RBF-FD-LS. 
\new{Comparing RBF-FD-LS-Ghost with RBF-FD-C-Ghost we can see that RBF-FD-LS-Ghost is more accurate in the case that we are considering.}


%


\begin{figure}[!htb]
    \centering
    \begin{tabular}{cc}
        \hspace{0.075\linewidth}$h=0.08$ & \hspace{0.075\linewidth}$h=0.02$ \\
        \includegraphics[width=0.35\linewidth]{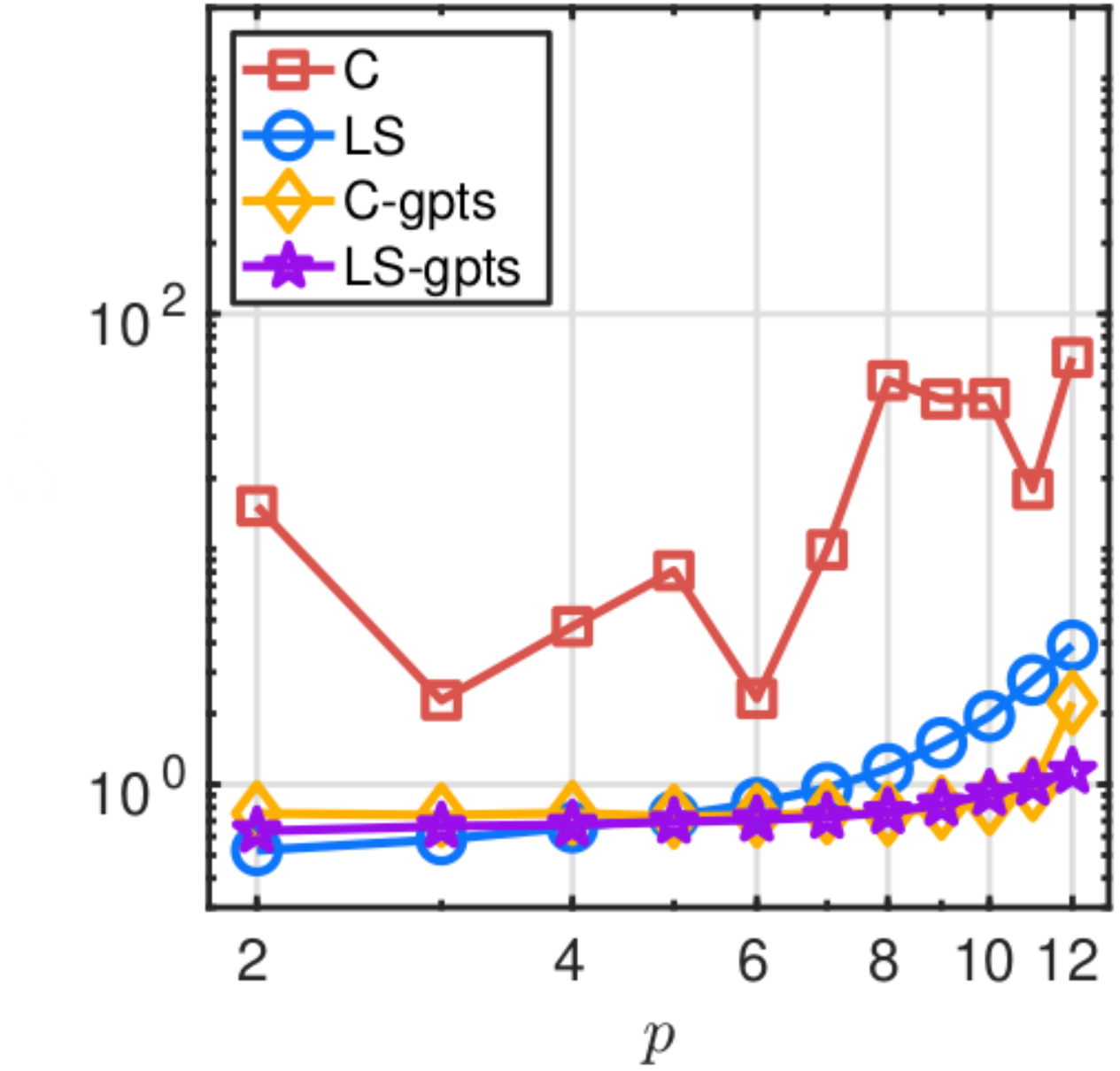} &
        \includegraphics[width=0.35\linewidth]{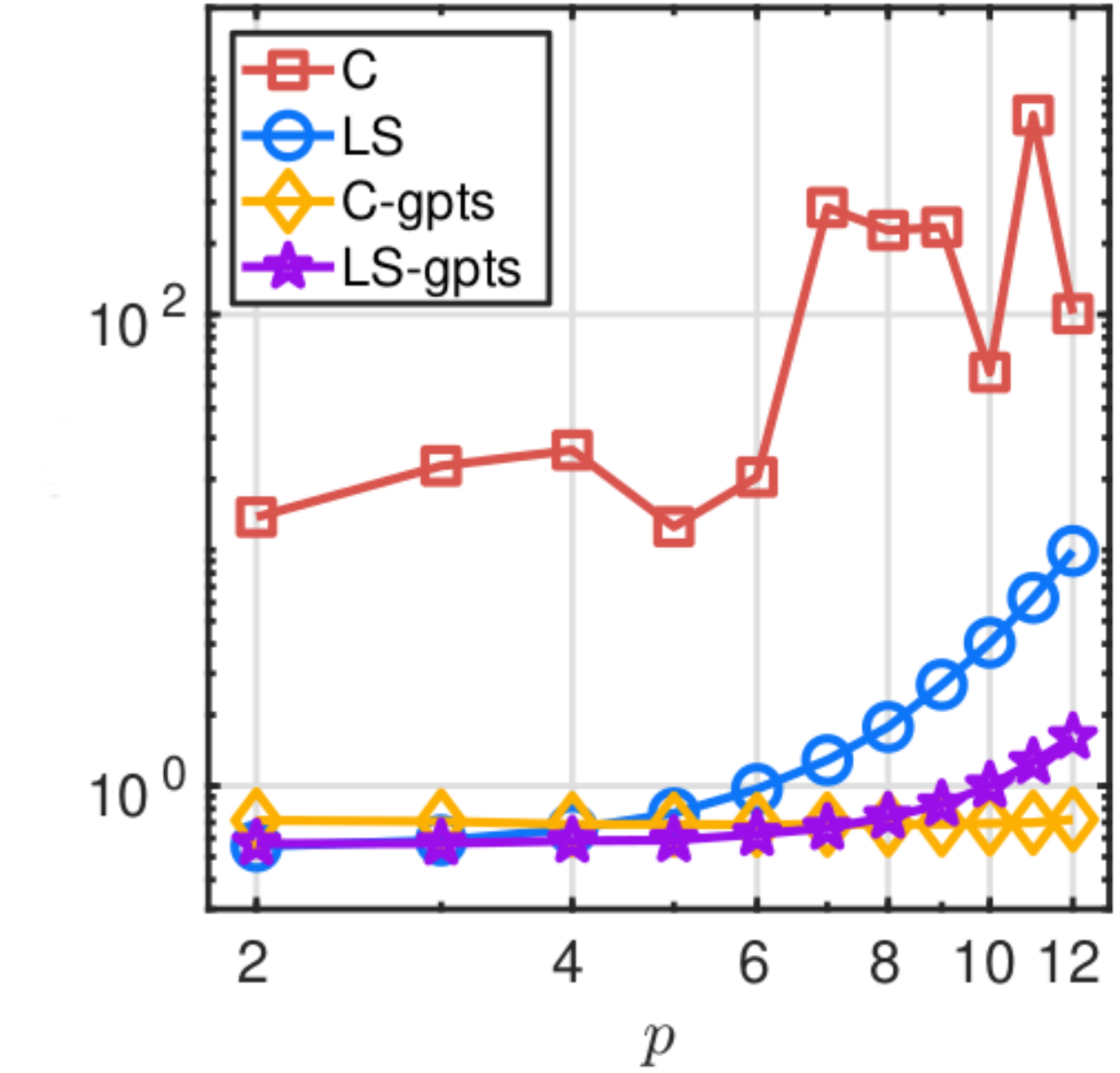} 
    \end{tabular}
        
    \caption{The stability norm in the $p$-refinement mode: the polynomial degree used to form the stencil-based interpolation matrix \eqref{eq:M} is increased, while the internodal distance between the stencil points $h$ is fixed at 
    $h=0.08$ and $h=0.02$.}
    \label{fig:experiments:Poisson:p-refinement:stabilityNorm}
\end{figure}
\begin{figure}[!htb]
    \centering
    \begin{tabular}{ccc}
        \raisebox{0.98cm}{\includegraphics[width=0.33\linewidth]{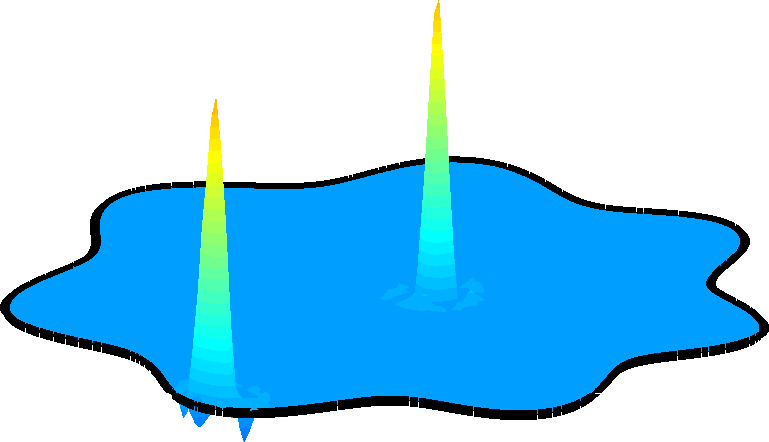}}&
        \includegraphics[width=0.33\linewidth]{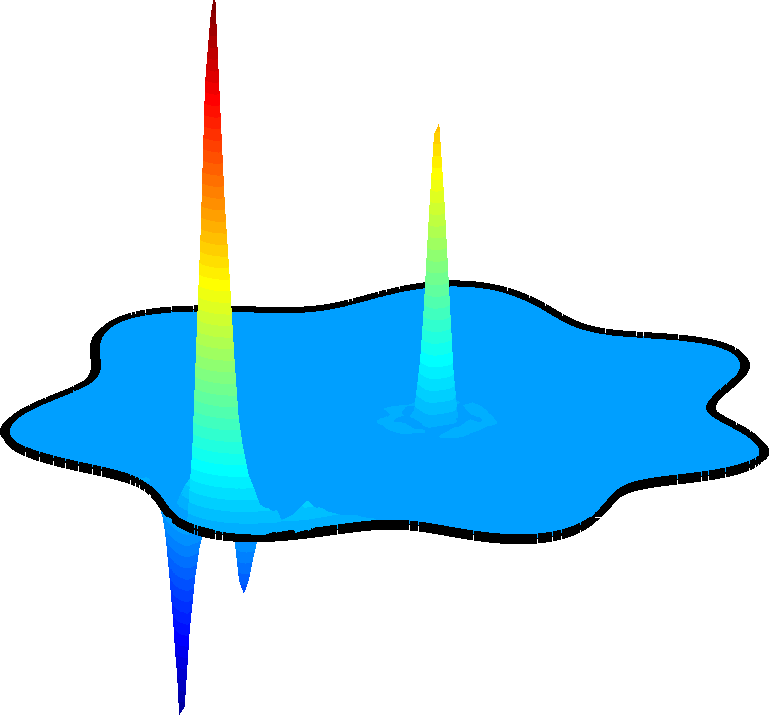} &
        \raisebox{1cm}{\includegraphics[width=0.05\linewidth]{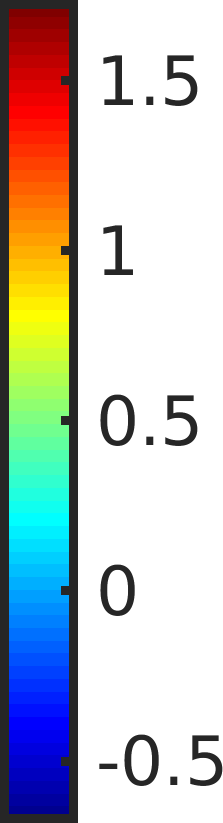}}    
    \end{tabular}
        
    \caption{Example of two cardinal functions placed on the boundary of $\Omega$ and in the interior of $\Omega$ for $p=3$ (left) and $p=12$ (right). For both cases the internodal distance 
      is $h=0.06$.
    }
    \label{fig:experiments:stabilityNorm_carindalFunctions}
\end{figure}
The stability norm as a function of $p$ is shown in \cref{fig:experiments:Poisson:p-refinement:stabilityNorm}. It has an increasing trend \new{for all methods}. 
Based on~\cref{eq:matform} and~\cref{theor:Dv}, we expect the bounds to grow with the stencil size, which depends on $p$. 
In the ideal case, the maximum value for any cardinal function is one, but here, there is an exponential growth of these functions with $p$, especially for cardinal functions close to the boundary. 
The largest weight for $p=12$ has $|w|\approx 3.2$. Cardinal functions for different values of $p$ are illustrated in \cref{fig:experiments:stabilityNorm_carindalFunctions}. 
The condition numbers of the local interpolation matrices $\tilde{A}_k$ also grow exponentially with $p$, and for $p\gtrsim 12$ prevent accurate numerical evaluation of the weights.  
\new{The results also show that both methods with ghost points have a smaller growth in the stability norm for an increasing $p$, compared to the methods which do not use ghost points. 
Since all of the stencils around the boundary are less skewed when using ghost points, it is expected that the stability improves. 
}

\subsection{\new{Approximation properties under node refinement in three dimensions}}
\new{In this section we solve the PDE problem \eqref{eq:methods:Poisson} in the same way as in the previous sections, but now in three dimensions. 
We compare RBF-FD-LS and RBF-FD-C with ghost points and without ghost points. The stencil sizes are in all cases $n=2m$, where $m$ is the dimension of the polynomial space.} 

\new{The norm scaling for the least-squares methods is in the 3D case  
chosen as $\beta = h_y^{3/2}$ for the Laplacian equations and $\beta = h_y$ for both boundary conditions, according to the relation \eqref{eq:method:scaling}.
The equation scaling $\beta_0$, $\beta_1$ and $\beta_2$ for the Dirichlet boundary, Neumann boundary and the Laplacian condition is given by : $\beta_0 = 1/h$, $\beta_1 = 1$, $\beta_2 = 1$, 
which is the same scaling as used in the 2D cases.}

\new{The 3D domain in spherical coordinates is given by:
$$r(\theta,\phi) = \left(1+\sin(2 \sin(\phi) \sin(\theta))^2\, \sin(2 \sin(\phi) \cos(\theta))^2\, \sin(2 \cos(\phi))^2\right)^{\frac{1}{2}},$$
where $\theta$ is the longitude angle $\theta \in [-\pi, \pi)$ and $\phi$ is a latitude angle $\phi \in [-\frac{\pi}{2}, \frac{\pi}{2})$.
The right-hand-sides of the PDE \eqref{eq:methods:Poisson} are computed based on a solution function:
$$u = \sin(3\pi\,x\,y\,z).$$
The coordinates $(x,y,z)$ of the mixed boundary conditions are given by: 
\begin{equation*}
    (x,y,z)|_{\partial\Omega_0} = \{(x,y,z)|_{\partial\Omega}\, |\, z < 0.7\},\quad
    (x,y,z)|_{\partial\Omega_1} = \{(x,y,z)|_{\partial\Omega}\, |\, z \geq 0.7\},
\end{equation*}
where the first set corresponds to the location of the Dirichlet boundary condition and the second to the location of the Neumann boundary condition.
}

\new{In the previous experiments in 2D we mostly used the oversampling parameter $q=3$. An equivalent sampling in three dimensions is motivated by knowing that 
$10$ points in 1D sample a fixed domain $[\Omega]^1$ equivalently well as $100$ points sample $[\Omega]^2$ in 2D and $1000$ points sample $[\Omega]^3$ in 3D. It follows that the relation $q_3 = q_2\sqrt{q_2}$ holds, where 
$q_2$ is a sampling in 2D and $q_3$ is an equivalent sampling in 3D.
All of our 3D computations are therefore based on a choice of an oversampling parameter $q=\lceil 3\sqrt{3} \rceil = 6$.}
\begin{figure}[!htb]
    \centering
        \includegraphics[width=0.4\linewidth]{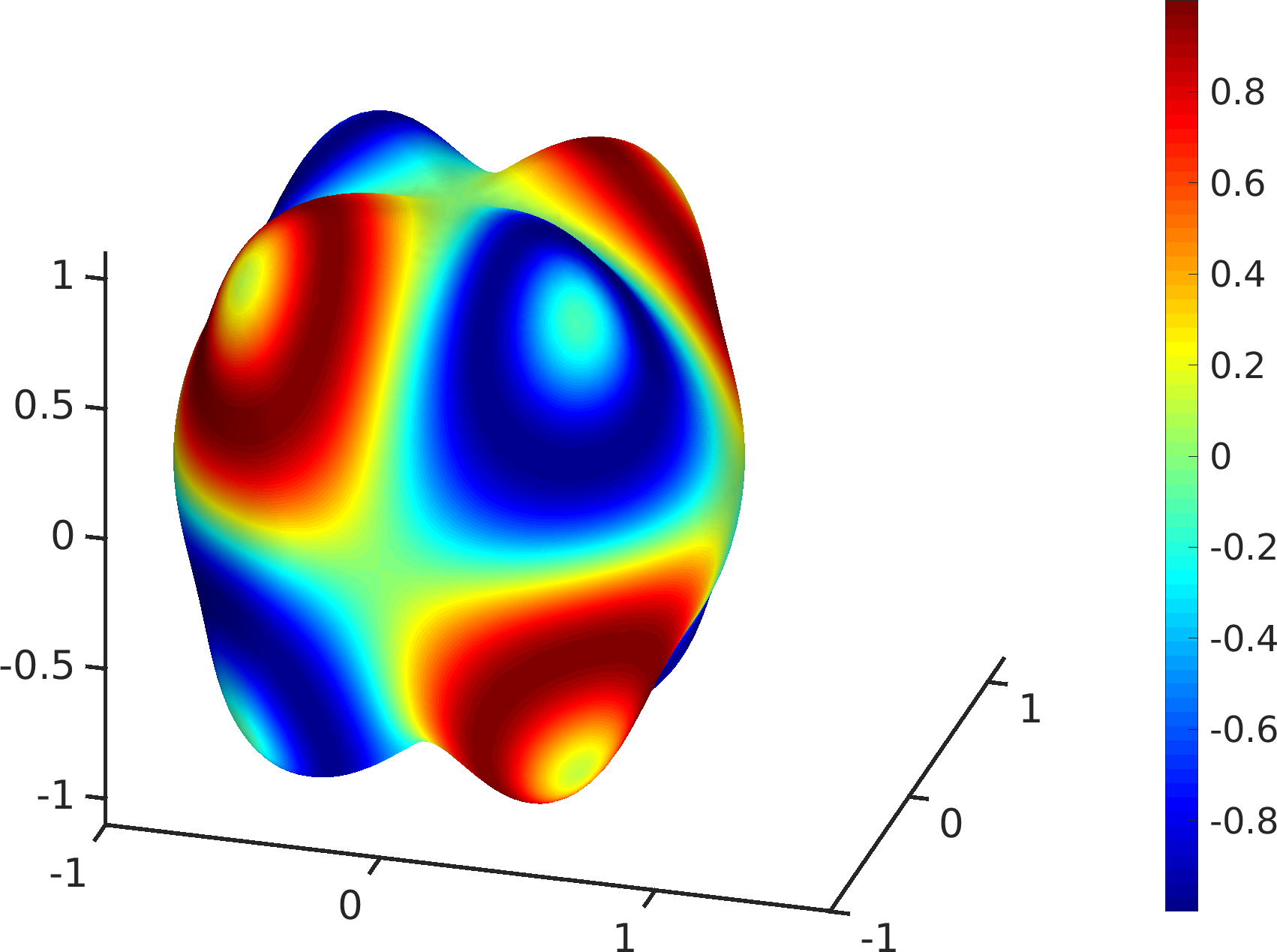}
        \includegraphics[width=0.4\linewidth]{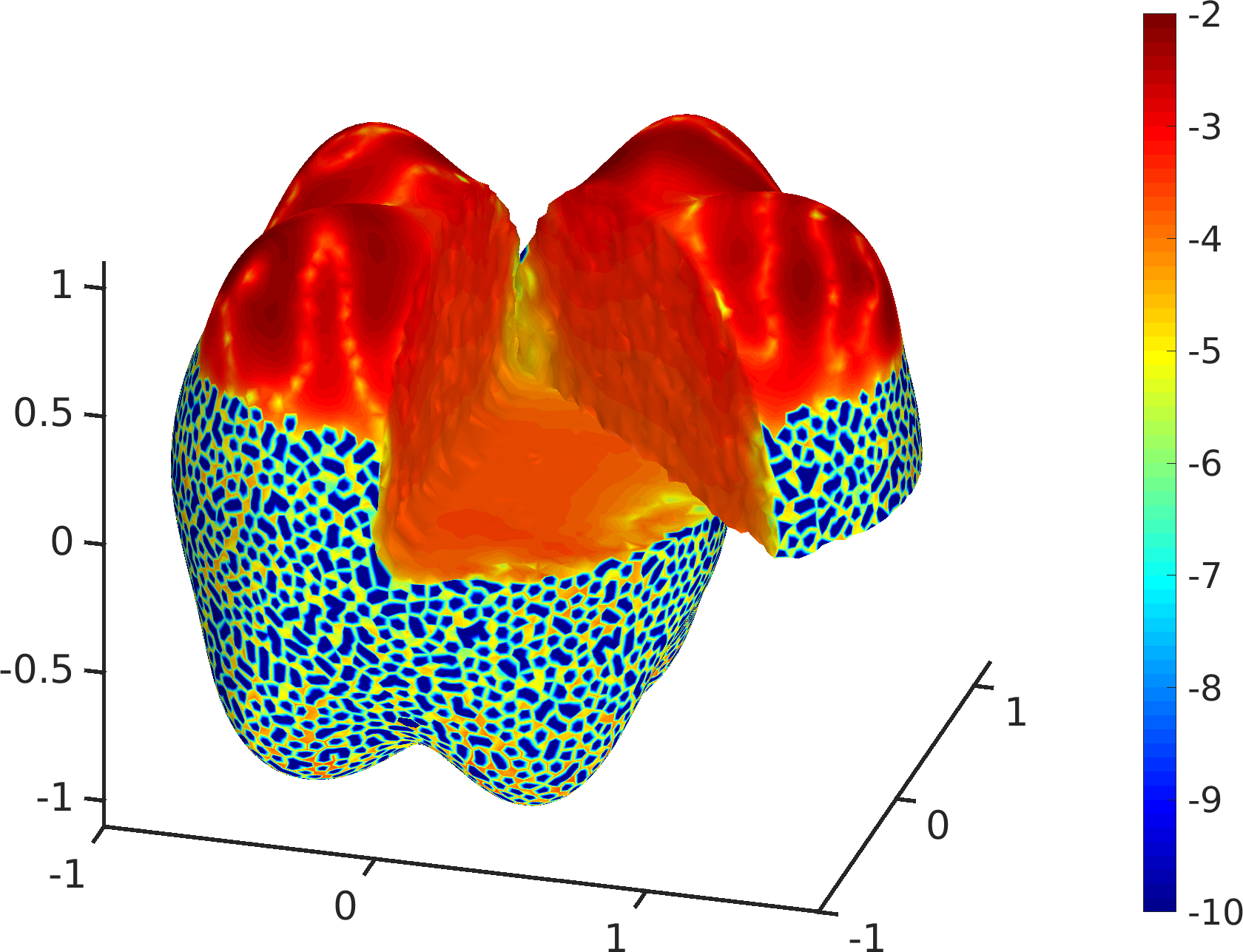}        
    \caption{Solution function (left) and spatial distribution of the relative absolute error in the logarithmic scale (right) when $h$ is chosen such that it corresponds to $N=24000$ points placed over $\Omega$. The polynomial degree chosen to 
    construct local approximations is $p=5$ and the oversampling parameter is $q=6$.}
    \label{fig:experiments:3D:solution_error}
\end{figure}
\new{An instance of a solution function, together with a relative absolute error in the logarithmic scale is given in \cref{fig:experiments:3D:solution_error}. 
We can observe that the largest errors appear close to the boundary $\partial\Omega_1$, where the Neumann condition is imposed. The error at the locations of 
the Dirichlet boundary $\partial\Omega_0$ is very small in some points, since those are the points where this condition is enforced exactly. The error in other points at the same boundary is 
larger, since in those points the Dirichlet condition is not satisfied exactly.}
\new{In \cref{fig:experiments:3D:solution_error,} we compute the error as the internodal distance $h$ is decreased and the polynomial degrees $p=2$, $3$ and $4$ are used 
to construct 
the local approximations. Our numerical results show that RBF-FD-LS, RBF-FD-LS-Ghost and RBF-FD-C-Ghost methods share a similar accuracy, while RBF-FD-C behaves unpredictably.}
\begin{figure}[!htb]
  \centering
  \small
    \begin{tabular}{ccc}
        \hspace{0.075\linewidth}$p=2$ & \hspace{0.075\linewidth}$p=3$ & \hspace{0.075\linewidth}$p=4$ \\
    \includegraphics[width=0.28\linewidth]{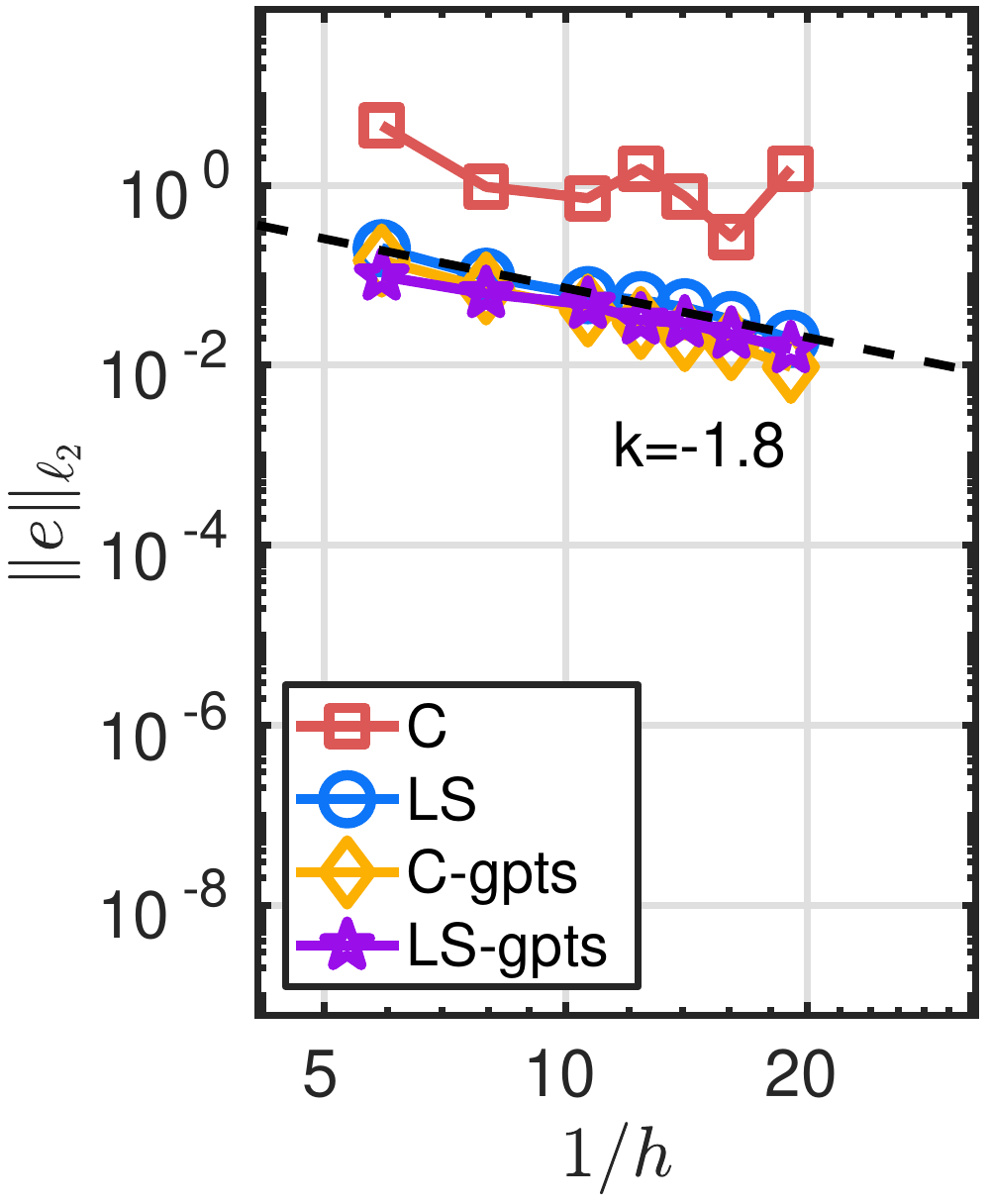} &
        \includegraphics[width=0.28\linewidth]{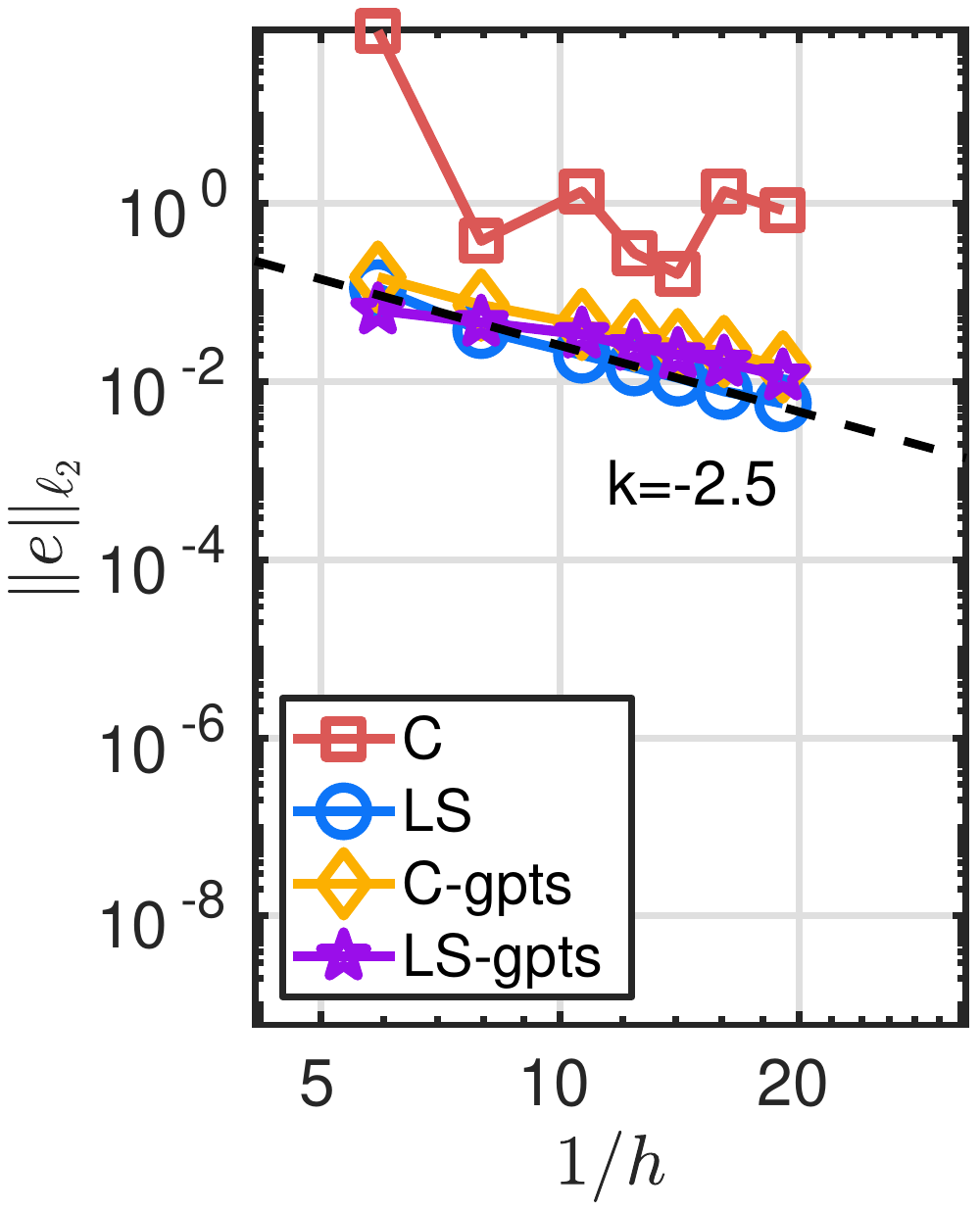} &
        \includegraphics[width=0.28\linewidth]{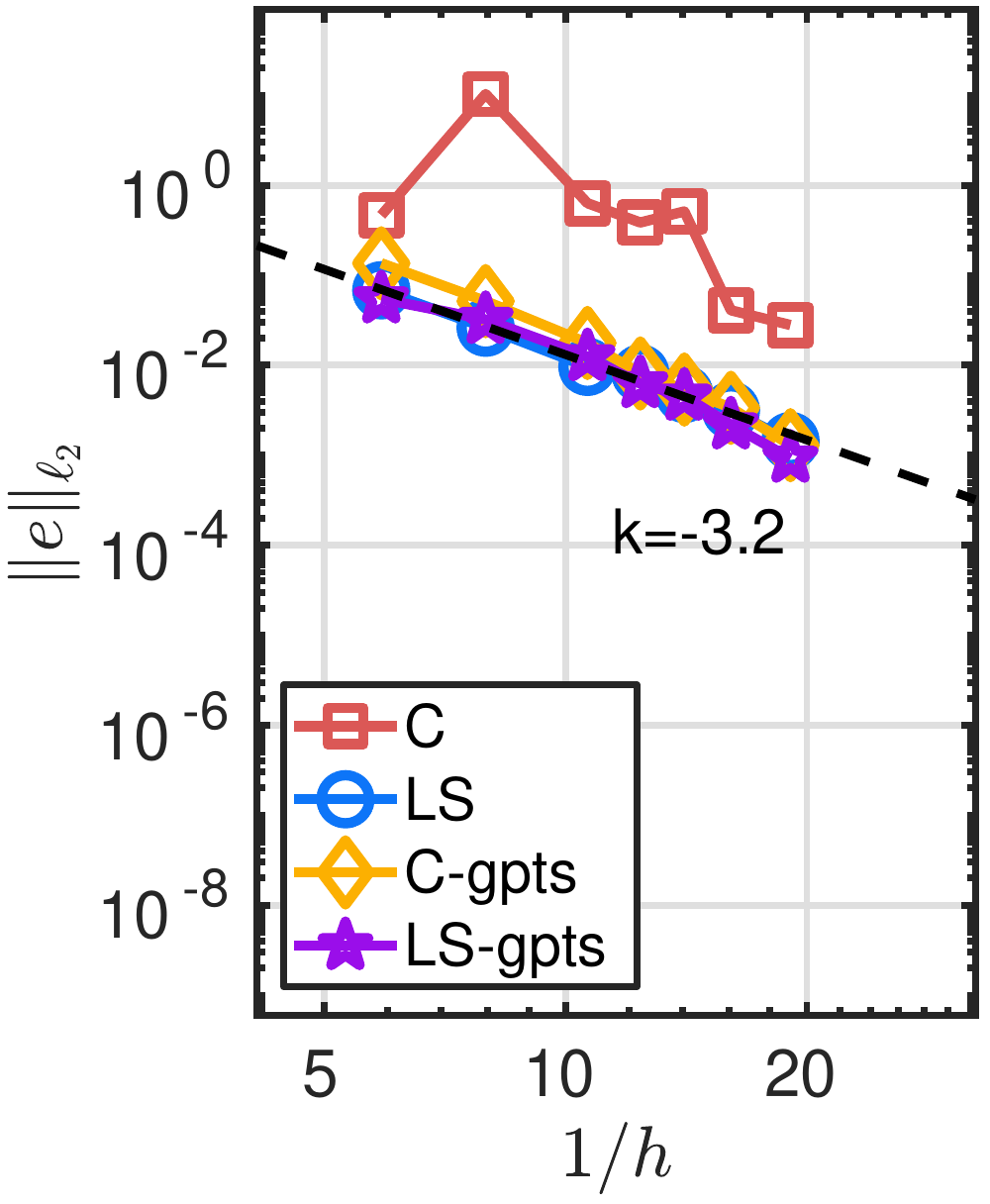} 
    \end{tabular}
    \caption{Error under nodal refinement in a three dimensional case for polynomial degrees $p=3$, $4$ and $5$. The oversampling parameter 
    is $q=6$ and the chosen $1/h$ corresponds to $N=1000$, $2000$, $4000$, $6000$, $8000$, $12000$ and $24000$ points spread over a 3D domain.}
    \label{fig:experiments:3D:error_hrefinement}
\end{figure}

\new{A study related to the stability norm in 3D is given in \cref{fig:experiments:3D:stability_hrefinement}. The experiment confirms that the numerical (and theoretical) observations in 2D generalize to 3D as well, since 
the stability norm of RBF-FD-LS does not have an unpredictable behavior.}
\begin{figure}[!htb]
  \centering
  \small
    \begin{tabular}{ccc}
        \phantom{xxxx}$p=2$ & \phantom{xxxx}$p=3$ & \phantom{xxxx}$p=4$ \\
    \includegraphics[width=0.28\linewidth]{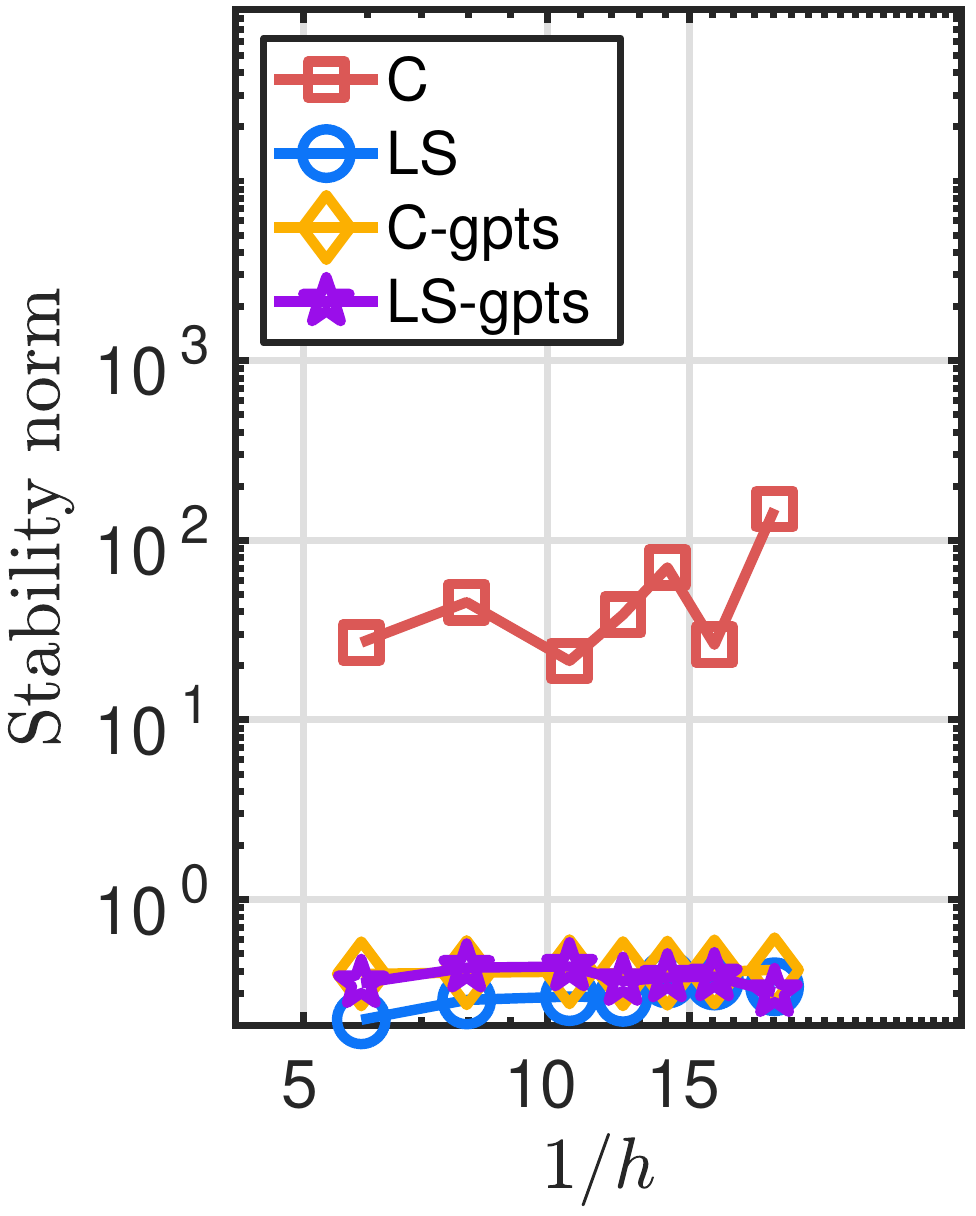} &
        \includegraphics[width=0.28\linewidth]{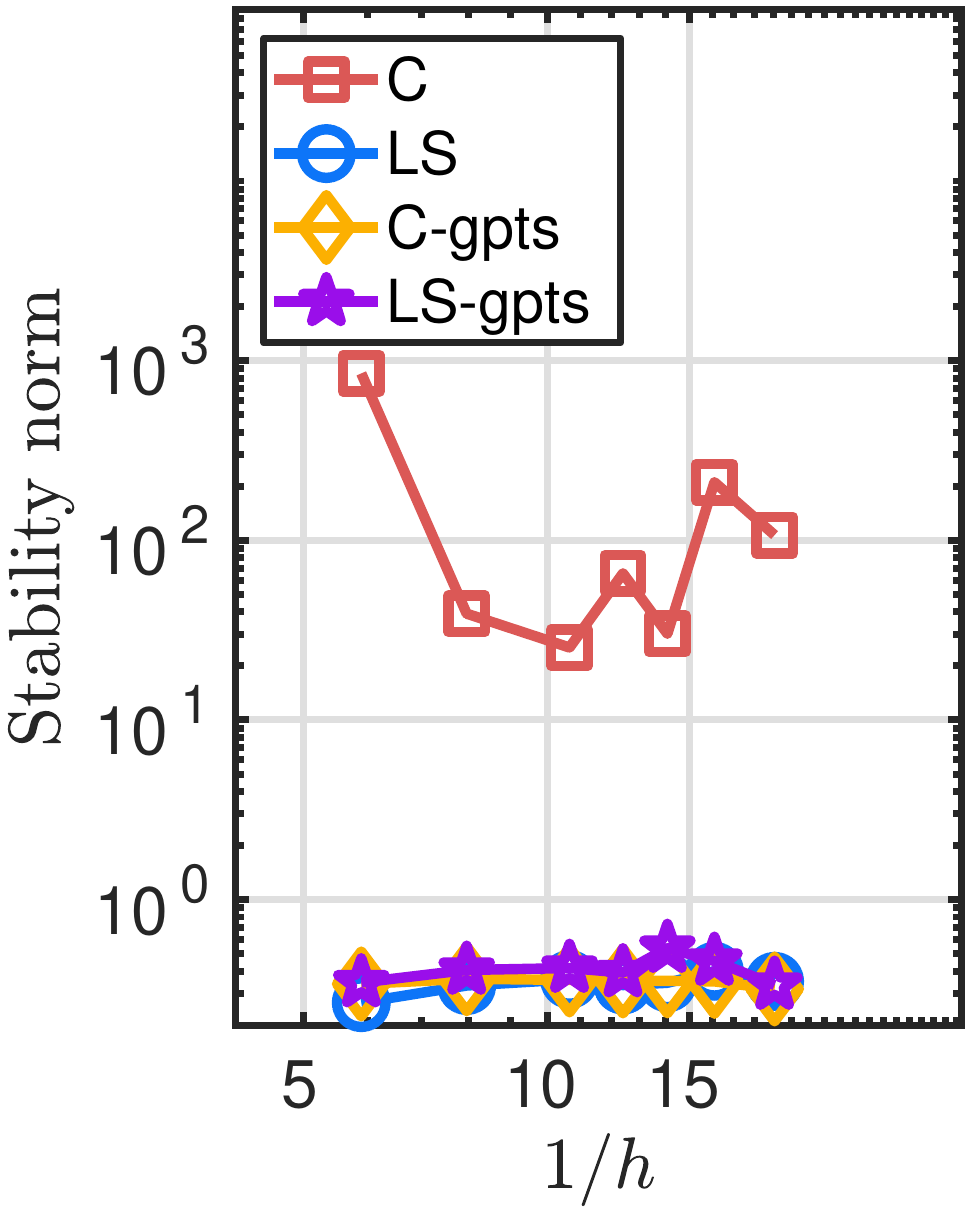} &
        \includegraphics[width=0.28\linewidth]{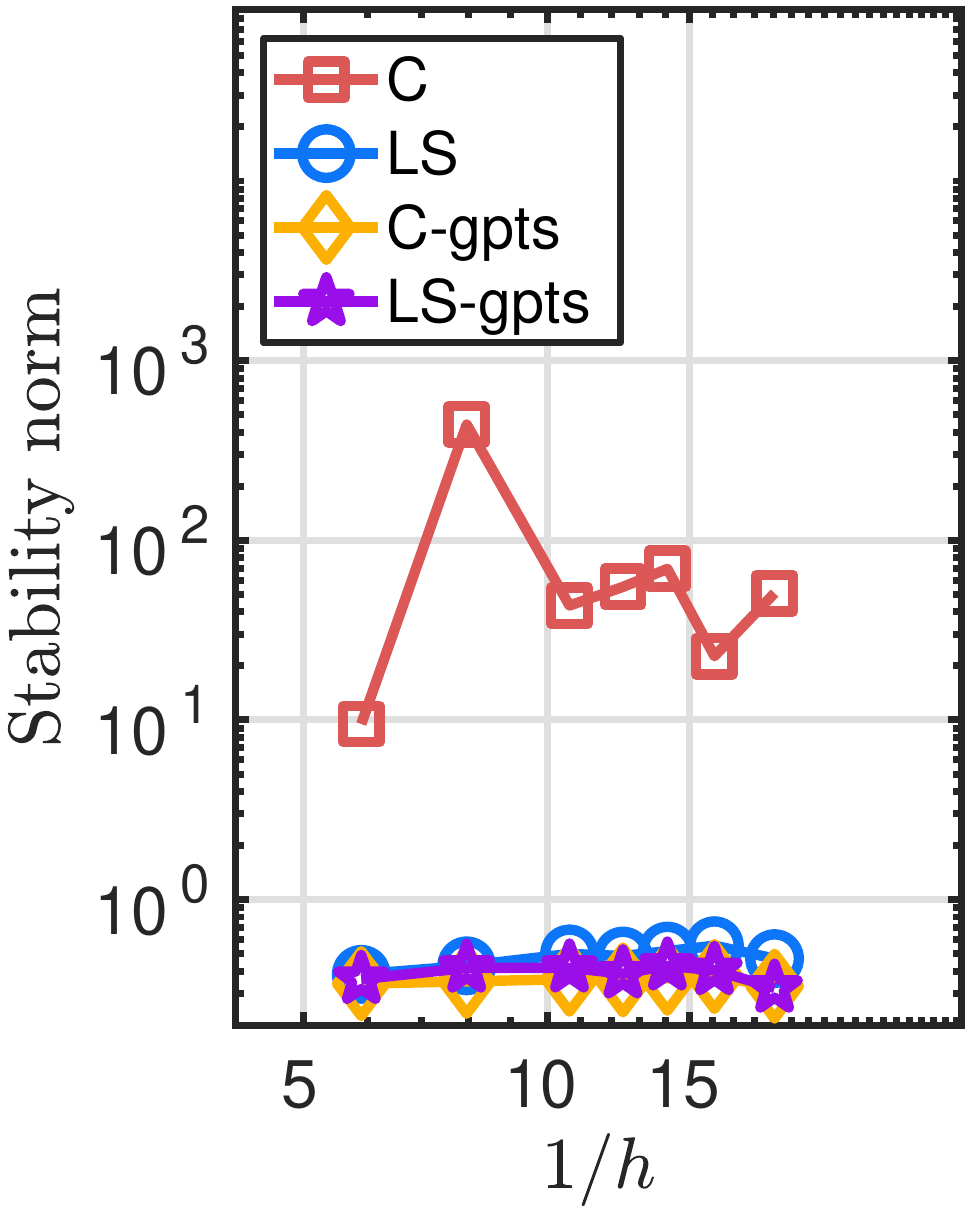} 
    \end{tabular}
    \caption{Stability norm under nodal refinement in a three dimensional case for polynomial degrees $p=3$, $4$ and $5$. The oversampling parameter 
    is $q=6$ and the chosen $1/h$ corresponds to $N=1000$, $2000$, $4000$, $6000$, $8000$, $12000$ and $24000$ points in a 3D domain.}
    \label{fig:experiments:3D:stability_hrefinement}
\end{figure}

\section{Final remarks}
\label{sec:finalremarks}
In this paper we introduced an enhancement of the collocation based RBF-FD method 
where we instead use a least-squares approach. The main method parameters are the node distance $h$, the evaluation node distance $h_y$, and the polynomial degree $p$ used to form the stencil approximations. 
The least squares formulation led us to 
characterize the RBF-FD trial space as a piecewise continuous space with jumps that vanish together with the local 
approximation error,  
and to understand that $\bar{D}_h^T \bar{D}_h u = \bar{D}_h^T f$ reproduces the $L_2$ inner-products of the continuous least-squares problem 
up to an error governed by
$h_y$.
This allowed us to prove well-posedness (stability) of RBF-FD-LS for an elliptic problem when $h_y$ is small enough in relation to $h$.
We also derived an error estimate in terms of the node distance, where the error decays with no less than order 
$p-1$ for the 
Poisson problem with Dirichlet and Neumann boundary conditions.

The experiments confirmed the theoretical observations in terms of the convergence trend as a sequence of $h$ gets increasingly small. 
We also confirmed that as $h$ is fixed at a small value, the stability norm and the error are improved as $h_y \to 0$, until 
both level out. This happens when the effect of the numerical integration becomes negligible. 

\new{An experimental comparison of RBF-FD-LS and RBF-FD-C with ghost points revealed that both methods are comparable in robustness. 
Even though, we believe that RBF-FD-LS has an advantage due to a better theoretical understanding (at least at the present moment) compared to RBF-FD-C with ghost points.}

\new{Overall,} the numerical experiments indicated that RBF-FD-LS \new{(no ghost points)} for our model problem 
performs better than RBF-FD-C \new{(no ghost points)} in terms of:
\begin{itemize}
\item the error against the exact solution for $p$-refinement and $h$-refinement,
\item the stability properties,
\item the efficiency.
\end{itemize}
The most important strength 
of the least-squares formulation is the robustness of the numerical solution as $h$ is decreased, which, according to 
our experience, is often lacking in the collocation formulation, especially in the presence of Neumann boundary conditions.


\section*{Acknowledgments}
For valuable discussions we thank Murtazo Nazarov, Gunilla Kreiss and Eva Breznik from Uppsala University, and Axel M{\aa}lqvist from Chal\-mers University of Technology. 
The third author thanks University of Massachusetts Dartmouth for financially supporting his sabbatical leave and 
Department of Information Technology, Uppsala University, for hosting his sabbatical visit in Spring 2019.

\bibliographystyle{siamplain}
\bibliography{refs}

\end{document}